\documentclass[10pt]{amsart}
\usepackage{amsfonts}
\usepackage{amsmath,amsthm,extarrows} 
\usepackage{hyperref} 
\usepackage{latexsym}
\usepackage{array}
\usepackage{amssymb}
\usepackage{enumerate}
\usepackage[english]{babel}
\theoremstyle{plain} \newtheorem{theorem}{Theorem}[section]
\newtheorem{lemma}[theorem]{Lemma}
\newtheorem{proposition}[theorem]{Proposition}
\newtheorem{corollary}[theorem]{Corollary} 
\newtheorem{definition}[theorem]{Definition} \theoremstyle{remark}
\newtheorem{remark}[theorem]{Remark}
\newtheorem{example}[theorem]{Example}

\usepackage{color}

\newcommand{\p}{\partial}

\renewcommand{\(}{  \big(   }
\renewcommand{\)}{  \big)   }    
\newcommand{\R}{  \mathbb{R}}
\newcommand{\PP}{  \mathbb{P}\,}
\newcommand{\Tg}{  {\bold T}_\rho  }
\newcommand{\sa}{  strongly admissible}

\newcommand{\eps}{\varepsilon}
\newcommand{\e}{  \text{e}   }

\newcommand{\C}{  \mathbb{C}   }
\newcommand{\Z}{  \mathbb{Z}   }
\newcommand{\N}{  \mathbb{N}   }
\newcommand{\A}{  \mathcal{A}   }

\newcommand{\M}{  \mathcal{M}   }
\newcommand{\F}{  \mathcal{F}   }
\newcommand{\Ca}{  \mathcal{C}   }
\newcommand{\J}{  \mathcal{J}   }
\newcommand{\fJ}{  \frak J  }
\newcommand{\bJ}{ \,\mathbb J  }

\newcommand{\no}{ \;\text{not}\;   }
\newcommand{\ap}{ a^{\prime\prime}   }
\newcommand{\aaa}{ a, a', a^{\prime\prime}  }

\renewcommand{\H}{  \mathcal{H}   }
\renewcommand{\O}{  \mathcal{O}   }
\newcommand{\B}{  \mathcal{B}   }

\newcommand{\T}{  \mathbb{T}   }
\newcommand{\bH}{  \bold{H}  }

\newcommand{\bB}{  \bold{B}   }

\newcommand{\bS}{  \bold{S}   }
\newcommand{{\Tc}}{  \mathcal{T}   }
\newcommand{\Rc}{  \mathcal{R}   }
\newcommand{\Cc}{  \mathcal{C}   }

\newcommand{\D}{  \mathcal{D}}
\newcommand{\Dj}{\tilde{Q}_j}

\newcommand{\wU}{\widehat {\bold U}}
\newcommand{\bU}{ {\bold U}}

\newcommand{\DD}{  \tilde{Q}_{l}  }

\newcommand{\Da}{  \mathcal{D}_{c_1}}
\newcommand{\Db}{  \mathcal{D}_{c_2}}

\newcommand{\dd}{  \text{d}   }

\newcommand{\om}{  \omega   }
\newcommand{\Om}{  \Omega}
\newcommand{\hren}{  \alpha}

\newcommand{\an}{ \, \angle\,}
\newcommand{\ann}{ \, \angle\!\angle\,}

\newcommand{\tl}{  \theta_\ell }
\renewcommand{\a}{  \alpha   }
\renewcommand{\b}{  \beta   }

\newcommand{\ga}{\gamma   }
\newcommand{\s}{  \sigma   }
\newcommand{\ka}{  \kappa   }
\newcommand{\ls}{  \lambda_s   }
\newcommand{\am}{  \lambda}
\newcommand{\Am}{  \Lambda}
\newcommand{\la}{  \lambda_a   }
\newcommand{\lb}{  \lambda_b   }
\newcommand{\La}{  \Lambda_a   }
\newcommand{\Lb}{  \Lambda_b   }
\newcommand{\li}{  \lambda_i   }
\newcommand{\lj}{  \lambda_j   }
\newcommand{\lk}{  \lambda_k   }
\newcommand{\lel}{  \lambda_\ell   }

\renewcommand{\phi}{  \varphi  }
\renewcommand{\L}{  \mathcal{L}   }

\newcommand{\tkd}{  { K}^d   }
\newcommand{\tkn}{   K ^{n/d}  }

\newcommand{\zz}{\mathfrak z}
\newcommand{\yy}{\rho}
\newcommand{\cc}{\frac1{2\sqrt2}}

\newcommand{\diag}{\operatorname{diag}}
\newcommand{\meas}{\operatorname{meas}}

\newcommand{\dist}{\operatorname{dist}}
\newcommand{\be}{\begin{equation}}
\newcommand{\ee}{\end{equation}}
\newcommand{\ben}{\begin{equation*}}
\newcommand{\een}{\end{equation*}}
\newcommand{\ban}{\begin{align*}}
\newcommand{\ean}{\end{align*}}
\numberwithin{equation}{section}

\def\der#1#2{\frac{d^{#1}\omega_{#2}}{dm^{#1}}}
\def\norma#1{\left\| #1\right\|}

\author{L. Hakan   Eliasson}
\address{
Univ. Paris Diderot, Sorbonne Paris Cit\'e\\
Institut de Math\'emathiques de Jussieu-Paris rive gauche, UMR 7586\\
CNRS\\
Sorbonne Universit\'es, UPMC Univ. Paris 06\\
F-75013, Paris, France
} 
\email{hakan.eliasson@imj-prg.fr}

 \author{ Beno\^it Gr\'ebert}
\address{Laboratoire de Math\'ematiques Jean Leray, Universit\'e de Nantes, UMR CNRS 6629\\
2, rue de la Houssini\`ere \\
44322 Nantes Cedex 03, France}
\email{benoit.grebert@univ-nantes.fr}

\author{ Sergei B. Kuksin }
\address{
CNRS\\
Institut de Math\'emathiques de Jussieu-Paris rive gauche, UMR 7586\\
Univ. Paris Diderot, Sorbonne Paris Cit\'e\\
Sorbonne Universit\'es, UPMC Univ. Paris 06\\
F-75013, Paris, France
}
\email{sergei.kuksin@imj-prg.fr}

\title[KAM for the nonlinear beam equation]{KAM for the nonlinear beam equation 1:
small-amplitude solutions.
}

\begin{document}

\begin{abstract}
In this paper we  prove a KAM result for the non linear beam equation on the d-dimensional torus
$$u_{tt}+\Delta^2 u+m u + g(x,u)=0\ ,\quad  t\in {  \mathbb{R}}
, \; x\in \T^d, \qquad \qquad
(*)
$$
 where $g(x,u)=4u^3+ O(u^4)$.  Namely, we show  that,
for generic $m$, many of the small amplitude invariant finite dimensional  tori of the linear equation $(*)_{g=0}$,
written as the system
$$
u_t=-v,\quad v_t=\Delta^2 u+mu,
$$
persist as invariant tori  of the nonlinear equation  $(*)$, re-written similarly. 
If $d\ge2$, then not all the persisted tori are  linearly stable,
and  we construct explicit examples  of partially hyperbolic  invariant tori.
The unstable invariant tori, situated in the vicinity of the origin, create around 
them   some local instabilities, in agreement with the popular belief  in the nonlinear 
physics that small-amplitude solutions of  space-multidimensonal hamiltonian
PDEs behave in a chaotic way. 

 The proof uses an abstract KAM theorem from another our publication \cite{EGK1}. 
 
 \begin{center} {\bf \large 8/12/ 2015}\end{center}\end{abstract}

\subjclass{ }

\maketitle
\tableofcontents

\section{Introduction}
\subsection{The beam equation  and the KAM for PDE theory}
The paper deals with small-amplitude solutions of 
 the multi-dimensional nonlinear 
 beam equation on the torus:
\be \label{beam}u_{tt}+\Delta^2 u+m  u =   -  g(x,u)\,,\quad u=u(t,x), \ 
  t\in \R, \ x\in \T^d=\R^d/2\pi\Z^d,
\ee
 where  $g$ is a real analytic function on $\T^d\times I$ for some neighbourhood $I$ of the origin in $\R$, satisfying 
 \be\label{g}
 g(x,u)=4u^3+g_0(x,u)\,, \qquad g_0=
  O(u^4).
 \ee
 $m$ is the mass parameter and we assume that $m\in[1,2]$. 
 This equation is interesting by itself. Besides, it is a good model for the Klein--Gordon equation 
 \be\label{KG}
 u_{tt} - \Delta u+mu=-g(x,u),\qquad x\in\T^d, 
 \ee
 which is among the most important equations of mathematical physics. We are certain that the ideas and methods 
 of our work apply -- with additional technical efforts -- to  eq.~\eqref{KG} (but the situation with the nonlinear wave 
 equation \eqref{KG}${}_{m=0}$, as well as with the zero-mass beam equation, may be quite different).
 
 Our goal is to develop a general KAM-theory for small-amplitude solutions of \eqref{beam}.  To do this we compare them 
     with time-quasiperiodic solution of the linearised at zero equation 
  \be\label{linear}
  u_{tt} +\Delta^2 u+ mu=0\,.
  \ee
 Decomposing real functions $u(x)$ on $\T^d$ to Fourier series 
 $$
 u(x)= \sum_{s\in\Z^d} u_s e^{is\cdot x}\ +\text{c.c.}
 $$
 (here c.c. stands for  ``complex conjugated"), we write time-quasiperiodic solutions for \eqref{linear}, corresponding 
 to a finite set of excited wave-vectors $\A \subset \Z^d $, $|\A|=:n$,  as
 \be\label{sol}
 u(t,x) = \sum_{s\in\A} (a_s e^{i\lambda_s t}+ b_se^{-i\lambda_s t}) e^{is\cdot x}
 + \text{c.c.},
 \ee
 where $\lambda_s = \sqrt{|s|^4+m}\,$.   We examine these solutions and their perturbations in eq.~\eqref{beam}
 under the assumption that the  action-vector 
 $
I= \{\tfrac12( a_s^2 +b_s^2),\ s\in\A\}\ 
 $
 is small. 
  In our work this goal is achieved  provided that 
 
 \noindent
 - the  finite set $\A$ is typical in some mild  sense;
  
 \noindent
 - the mass parameter $m$ does not belong to a certain set of zero measure.

  The linear stability of the obtained solutions for  \eqref{beam} is under control. If $d\ge2$, and $|\A|\ge2$, 
  then some of them are  linearly unstable. 
 \smallskip

 The specific choice of a hamiltonian PDE with the mass parameter which we work with   -- the beam equation \eqref{beam} -- 
 is sufficiently arbitrary. This is simply the easiest non-linear space-multidimensional equation from mathematical
 physics for which we can perform our programme of the KAM-study of small-amplitude solutions in space-multidimensional 
 hamiltonian PDEs, and obtain for them the results, outlines above. We are certain that out picture of the KAM-behaviour of 
 small solutions, as well as the method, developed to prove it, are sufficiently, general. In particular, 
 we believe that out method applies to the Klein-Gordon equation \eqref{KG}. 
 \smallskip
 
 Before to give exact statement of the result, we discuss the state of affairs in the KAM for PDE theory. The theory 
 started in late 1980's and originally applied to 1d hamiltonian PDEs, see in \cite{Kuk1, Kuk2, Cr}. The first works 
 on this theory  treated 
 
 \noindent 
 a) perturbations of linear hamiltonian PDE, depending on a vector-parameter of the dimension, equal to 
  the number 
 of frequencies of the unperturbed quasiperiodic solution of the linear system (for solutions \eqref{sol} this is
 $|\A|$).
 
 Next the theory was applied to 
 
 \noindent
 b) perturbations of integrable hamiltonian PDE, e.g.  of the KdV or Sine-Gordon equations, see \cite{Kuk3}. 
 
 In paper \cite{BoK} 
 
 \noindent
 c) small-amplitude solutions of the 1d Klein-Gordon equation \eqref{KG} with $g(x,u)=-u^3+O(u^4)$ 
  were treated as perturbed
  solutions of the Sine-Gordon equation,\footnote{Note that for suitable $a$ and $b$ we
  have $mu-u^3+O(u^4) = a\sin bu+O(u^4)$. So the 1d equation \eqref{KG} is the Sine-Gordon equation,
  perturbed by a small term $O(u^4)$.
  }
  and a singular version of the KAM-theory b) was developed to study them.

  It was proved in \cite{BoK}   that for a.a. values of $m$ and for 
  any finite set $\A$  most of the  small-amplitude solutions \eqref{sol} for the 
  linear Klein-Gordon equation (with $\lambda_s=\sqrt{|s|^2+m}$) persist as linearly stable time-quasipe\-rio\-dic 
  solutions for \eqref{KG}. In \cite{KP} it was realised that it is easier to study small solutions of 1d
  equations like   \eqref{KG} not as perturbations of solutions for an integrable PDE, but rather as perturbations of solutions 
  for a Birkhoff--integrable system, after the equation is normalised by a Birkhoff transformation. The paper \cite{KP} deals not with 
  1d Klein-Gordon
   equation \eqref{KG}, but with 1d NLS equation, which is similar to \eqref{KG} 
   for the problem under discussion; in \cite{P} the method of \cite{KP} was applied to the 1d equation \eqref{KG}.
     The approach of \cite{KP}
  turned out to be very efficient and later was applied to many other 1d hamiltonian PDEs. 
  \smallskip
  
  Space-multidimensional KAM for PDE theory started  10 years later with the paper \cite{B1} and, next, publications 
  \cite{B2}
  and \cite{EK10}. The just mentioned works deal with parameter-depending linear equations (cf. a)\,). The approach of 
  \cite{EK10} is different from that of \cite{B1, B2} and allows to analyse the linear stability of the obtained KAM-solutions. 
  Also see \cite{Ber1, Ber2}.  Since integrable space-multidimensional PDE (practically) do not exist, then no 
  multi-dimensional analogy of the 1d theory b) is available. 
  
  Efforts to create space-multidimensional analogies of the KAM-theory c) were made in \cite{WM} and \cite{PP1, PP2}, using the
  KAM-techniques of \cite{B1, B2} and \cite{EK10}, respectively. Both works deal with the NLS equation. Their main 
  disadvantage compare to the 1d theory c) is severe restrictions on the finite set $\A$ (i.e. on the class of unperturbed solutions 
  which the methods allow to perturb).
  The result of \cite{WM} gives examples 
  of some sets $\A$ for which the KAM-persistence of the corresponding small-amplitude solutions \eqref{sol} holds,
  while the result of \cite{PP1, PP2} applies to solutions \eqref{sol}, where the set $\A$ is nondegenerate in certain very 
  non-explicit way. The corresponding 
   notion of non-degeneracy is so complicated that it is  not easy to give  examples of 
  non-degenerate sets $\A$. 
  
  Some KAM-theorems for small-amplitude solutions of multidimensional beam equations \eqref{beam}   
  with typical $m$   were obtained in
  \cite{GY1, GY2}. Both works treat equations  with a constant-coefficient nonlinearity 
  $g(x,u)=g(u)$, which  is significantly easier than the general case (cf. the linear theory, where constant-coefficient
  equations may be integrated by the Fourier method). Similar to \cite{WM, PP1, PP2}, the theorems of  \cite{GY1, GY2}
  only allow to perturb solutions \eqref{sol} with very special sets $\A$ (see also Appendix B). Solutions of \eqref{beam}, constructed in these works, 
  all are  linearly stable.

 \subsection{Beam equation in the complex variables}\label{s_complex}
 Introducing $v=u_t\equiv\dot u$ we rewrite 
 \eqref{beam} as 
\be\label{beam'}
 \left\{\begin{array}{ll}
 \dot u &= - 
 v,\\
 \dot v &=\Lambda^2 u    +g(x,u)\,,
\end{array}\right.
\ee
where $\Lambda=(\Delta^2+m)^{1/2}$. 
Defining 
 $
 \psi(t,x) =\frac 1{\sqrt 2}(\Lambda^{1/2}u  + i\Lambda^{-1/2}v) $
 we get for the complex function 
  $\psi(t,x)$ the equation
$$
\frac 1 i \dot \psi =\Lambda \psi+ \frac{1}{\sqrt 2}\Lambda^{-1/2}g\left(x,\Lambda^{-1/2}\left(\frac{\psi+
\bar\psi}{\sqrt 2}\right)\right)\,.
$$
Thus, if we endow the space   $L_2(\T^d, \C)$ with the standard  real symplectic structure,  given by the two-form
$\ 
-id\psi\wedge d\bar \psi =- d\tilde u\wedge d\tilde v,
$
where $\psi =\frac1{\sqrt2} (\tilde u+i\tilde v)$, then equation 
 \eqref{beam} becomes a hamiltonian system 
$$\dot \psi=i \,{\partial H} /{\partial \bar\psi}
$$
with the hamiltonian function
$$
H(\psi,\bar\psi)=\int_{\T^d}(\Lambda \psi)\bar\psi \dd x +\int_{\T^d}G\left(x,\Lambda^{-1/2}\left(\frac{\psi+\bar\psi}{\sqrt 2}\right)\right)\dd x.
$$
Here $G$ is a primitive of $g$ with respect to the variable $u$: 
$$
g=\partial_u G\,, \quad G(x,u)=u^4+O(u^5)\,.
$$
The linear operator $\Lambda$ is diagonal in the complex Fourier basis  
$$
\{\phi_s(x)= {(2\pi)^{-d/2}}e^{is\cdot x}, \ s\in\Z^d\}.
$$
Namely, 
$$
\Lambda \phi_s=\ls \phi_s,\;\;\ls= \sqrt{|s|^4+m},
\qquad  \forall\,s\in\Z^d\,.
$$

Let us decompose $\psi$ and $\bar\psi$  in the   basis $\{\phi_s\}$:
$$
\psi=\sum_{s\in\Z^d}\xi_s \phi_s,\quad \bar\psi=\sum_{s\in\Z^d}\eta_s \phi_{-s}\,.
$$
We fix any $d^*>d/2$ and define the space
\be\label{YC}
Y^C = \{(\xi,\eta)\in \ell^2(\Z^d,\C)\times\ell^2(\Z^d,\C) \mid 
\sum_s \max(1, |s|^2)^{d^*} (|\xi_s|^2 +|\eta_s|^2)  <\infty \}\,,
\ee
corresponding to the Fourier coefficients of complex functions 
$(\psi(x), \bar\psi(x))$ from the Sobolev space 
 $ H^{d^*}(\T^d, \C^2) =: H^{d^*}$. Let us endow $Y^C$ with the  complex 
symplectic structure ${ -}i\sum_s \dd\xi_s\wedge\dd\eta_s$, and  consider there the hamiltonian system
\be \label{beam2} \left\{\begin{array}{ll}\dot \xi_s&=i\frac{\partial H}{\partial \eta_s}\\ 
\dot \eta_s&=-i\frac{\partial H}{\partial \xi_s}\end{array}\right. \quad s\in\Z^d\,,
\ee
where the hamiltonian function $H$ is given by
$
H=H_2+P
$
with
\be\label{H1}
H_2=\sum_{s\in\Z^d}\ls \xi_s\eta_s,\quad P=\int_{\T^d}G\left(x,\sum_{s\in\Z^d}\frac{\xi_s\phi_s+\eta_{-s}\phi_{s}}{\sqrt{ 2\ls}}\right)\dd x.\ee
Then the  beam equation \eqref{beam'}, considered in the Sobolev space $\{(u,v) \mid(\psi, \bar\psi) \in H^{d^*}\}$, is 
 equivalent to the  hamiltonian system \eqref{beam2}, restricted to the real subspace 
\be\label{PR}
Y^R :=\{(\xi,\eta)\in Y^C\mid  \eta_s=\bar\xi_s, \ s\in\Z^d\}.
\ee

The leading quartic  part of $P$ at the origin, 
 \be\label{quatr}
 P_4=\int_{\T^d}u^4\dd x= \int_{\T^d}\left(\sum_{s\in\Z^d}\frac{\xi_s\phi_s+\eta_{-s}\phi_{s}}{\sqrt{ 2\ls}}\right)^4\dd x,
 \ee
 satisfies the {\it zero momentum condition}, i.e. 
 $$
 P_4=\sum_{i,j,k,\ell\in\Z^d}C(i,j,k,\ell) (\xi_i+\eta_{-i})
( \xi_j+\eta_{-j}) 
(\xi_k+\eta_{-k})
(\xi_\ell+\eta_{-\ell})\,,
 $$
 where $C(i,j,k,\ell)\ne0$ only if $i+j+k+\ell=0$. 
 If $g$ does not depend on $x$, then $P$ satisfies a similar property at any order.
  This  condition turns out to  be  useful to restrict the set of small divisors that  have to be controlled.

 \subsection{Admissible and\sa\  sets $\A$}\label{s_admiss}
Let $\A$ be a finite  subset of $\Z^d$, $|\A|=:n\ge0$. We define
$$
\L = \Z^d\setminus \A\,,
$$
and decompose the spaces $Y^C$ and $Y^R$ as 
$$
Y^C = Y^C_\A\oplus  Y^C_\L\,,\; Y^R = Y^R_\A\oplus  Y^R_\L\,,\; \text{where}\;
 Y^C_\A  = \{(\xi_a, \eta_a), \ a\in\A \mid (\xi,\eta) \in Y^C\}\,, \; \text{etc.}
$$

Let us  take a vector  with positive components $I=(I_a)_{a\in\A}\in \R^n_+$.
The $n$-dimensional real torus
\ben
T^n_I=
 \left\{\begin{array}{ll}
\xi_a=\bar\eta_a,\;
 |\xi_a|^2 =I_a,\quad &a\in \A\\
\xi_s=\eta_s=0,\quad & s\in \L  \,,
\end{array}\right.
\een
is invariant for the linear hamiltonian flow  when $P=0$ (i.e. $g=0$ in \eqref{beam}). 
Our goal is to prove    the persistency of most of the
 tori $T^n_I$ 
  when the perturbation $P$  turns on, assuming that the set of nodes $\A$ is  {\it admissible} 
  or{\it \sa}\ in the sense, discussed  below in this section.

\begin{definition}\label{adm}
A finite set $\A\in\Z^d$, $|\A|=:n\ge0$,  is called admissible if
$$
 j, k\in\A, \ j\ne k
\Rightarrow |j|\neq |k|\,.
$$
\end{definition}
Certainly if $n\le1$, then $\A$  is admissible.

For any $n$, large admissible sets $\A$ with $n$ elements  are typical in the following sense.
For $R\ge1$ denote by $B(R)$ the $R$-ball $\{ x\in \R^d\mid |x| \le R\}$, by $\bB(R)$ -- the integer 
ball $\bB(R)= B(R) \cap \Z^d$, denote by $S(R)$ the sphere $S(R) = \p B(R)$, and by 
$\bS(R)$ -- the integer sphere $\bS(R) =S(R)\cap \Z^d$ (so $\bS(R)=\emptyset$ if $R^2\notin\Z$). 
Let $\xi^{1},\dots, \xi^n$, $\xi^j = \xi^{j\omega}$, 
 be independent random variables, uniformly distributed in $\bB(R)$, $R\ge1$. Consider the event
 $$
 \Omega_+ =\{\xi^i\ne \xi^j \quad\text{if}\quad  i\ne j\}\,.
 $$
 Then $\A^\omega=\{\xi^{1\omega}, \dots, \xi^{n\omega}\},\ \omega\in\Omega_+$, is an $n$-points random set. 
 We will call it an {\it $n$-points  random $R$-set}.

 Obviously
 $$
 \PP \Omega_+ \ge 1 - C(n,d) R^{-d}\,.
 $$
Now consider the event 
$$
\Omega_1 = \{|\xi^i| \ne |\xi^j| \quad \text{for all}\quad i\ne j\} \subset\Omega_+\,.
$$
The conditional probability $\PP(\Omega_1\mid \Omega_+)$ is the probability that an $n$-points random  $R$-set 
$\A^\omega$ is admissible.  In Appendix~E we show that 
\be\label{admis}
\PP(\Omega_1\mid \Omega_+) \ge 1 - C(n,d) R^{-1}\,.
\ee
So for any $n$ and $d$ 
$$
\text{
  admissible $n$-points random $R$-sets  with $R\gg1$ are typical.}
 $$
\smallskip

Now we define a subclass of admissible sets and start with a notation. For  vectors $a,b \in\Z^d$ we 
write 
\be\label{ddd}
a \an b \quad \text{iff} \quad \#\{x\in \bS(|a|) \mid |x-b| = |a-b|\} \le 2\,,
\ee
and 
$$
a \ann b \quad \text{iff} \quad a \an a+b\,.
$$
Relation $a\an b$ means that 
the integer sphere of radius $|b-a|$ with the centre at $b$ intersects $\bS(|a|)$  
in at most two points. Obviously,
\be\label{obvi}
0\an b \quad \text{and}\quad 0\ann b\quad \forall\, b\,.
\ee
If $d=2$, then for any $a$ we have 
$a\an b$ provided that $b\ne0$, and $a\ann b$ if $a+b\ne0$. 

\begin{definition}\label{sadm}
An admissible set $\A$ is called strongly admissible if either $|\A|\le1$, or $|\A|\ge2$ and for any $a,b\in\A$, $a\ne b$,
we have $a\ann b$. 
\end{definition}

Since $a+b\ne0$ for any two different points of an admissible set, then 
$$ \text{
for $d\le2$ every admissible set is 
strongly admissible.  }
$$
In high dimension this is not any more true, e.g. see the set \eqref{AAA} in Appendix B.
Still strongly admissible $n$-points random $R$-sets with $R\gg1$ 
are typical.  Namely, consider again the random points  $\xi^1,\dots, \xi^n$ in $\bB(R)$, and consider the event 
$$
\Omega_2 = \{ \xi^i\ann\xi^j \quad\text{for all}\quad i\ne j\}\,.
$$
Then the random sets, corresponding to $\omega\in\Omega_1\cap \Omega_2$, are strongly admissible, and 
$\PP(\Omega_1\cap\Omega_2 \mid \Omega_+)$ is the probability that an $n$-point
 random  $R$-set $\A^\omega$ 
is strongly admissible. Clearly
\be\label{w22}
\PP(\Omega\setminus\Omega_2) \le n(n-1) (1- \PP\{ \xi^1\ann \xi^2\})\,.
\ee
In Appendix E we prove that 
\be\label{w3}
 1- \PP\{ \xi^1\ann \xi^2\} \le CR^{-\ka}\,,
\ee
where $C=C(n,d)>0$ and $\ka=\ka(d)>0$ (e.g. $\ka(3) = 2/9$). By \eqref{admis}, 
 \eqref{w22} and \eqref{w3}, for any $n$ and  $d$ 
 $$
 \text{
   strongly admissible $n$-points 
 random $R$-sets with $R\gg1$  are typical.}
 $$

 \subsection{Statement of the main results}\label{s_1.2}
We recall that $\L=\Z^d\setminus \A$ and define two subsets of $\L$,
important for our construction:
\be\label{L+}
 {\L_f}=\{s\in\L \mid \exists\ a\in\A \text{ such that } |a|=|s|  \}\,,\quad \L_\infty = \L\setminus \L_f.
\ee
Clearly $\L_f$  is a finite subset of $\L$. For example, if $d=1$ and $\A$ is admissible, then $\A\cap-\A\subset\{0\}$, so
\be\label{ex_d0}
\text{
if $d=1$,  then 
$
 {\L_f} = -(\A\setminus\{0\})
 $. }
\ee

In a neighbourhood of an invariant torus $T^n_I$  in the real space 
$\{(\xi_a=\bar\eta_a, a\in\A)\}\subset \C^{2n}
$
 we introduce the real  action-angle variables $(r_a,\theta_a)_\A$ by the relation 
\begin{align*} 
\xi_a=\sqrt{I_a+r_a}  \, e^{i \theta_a}
\end{align*}
(note that $-i\sum_{a\in\A}d\xi_a\wedge d\eta_a=-dI\wedge d\theta$). 
We will write
\be\label{acc-ann}
\xi^\A= \sqrt{I+r}  \, e^{i \theta}\,,\  \eta^\A= \sqrt{I+r}  \, e^{-i \theta}\,;\quad
\xi^\A=\{\xi_a, a\in\A\}\,, \ \eta^\A=\{\eta_a, a\in\A\}\,.
\ee
We will often denote the internal frequencies by $\om$, i.e. $\ls=\om_s$ 
for $s\in\A$, and we will keep the notation $\ls$ for the external frequencies with $s\in\L=\Z^d\setminus \A$.
Then the  quadratic part of the Hamiltonian  becomes, up to a constant,
$$
H_2=\sum_{a\in\A}\om_a r_a+  \sum_{s\in\L}\ls \xi_s\eta_s.
$$
The perturbation $P$ is an analytic   function of all variables and reads
$$ 
P(r,\theta,\xi,\eta)=\int_{\T^d} 
G(x, \hat u_{I,m}(r,\theta,\xi,\eta)) \dd x \,,
$$
where $\hat u_{I,m}(r,\theta,\xi,\eta)$ is $u(x)= \Lambda^{-1/2}{(\psi+\bar\psi)}/{\sqrt2}$, expressed in
the variables $(r,\theta,\xi_s,\eta_s)$:
\be\label{uhat}
\begin{split}
 \hat u_{I,m} =
 \sum_{s\in\A}\sqrt{I_a+r_a}\, 
 \frac{e^{-i\theta_a}\phi_a( x)+e^{i\theta_a}\phi_{-a}(x)}{\sqrt 2\,(|a|^4+m)^{1/4}} 
 +\sum_{s\in\L}\frac{\xi_s \phi_s( x)  + \eta_{-s} \phi_{s}(x)}{ \sqrt2\,(|s|^4+m)^{1/4}}.
 \end{split}
 \ee

For any  $I\in\R_+^n$, $m\in[1,2]$ and $\theta^0\in\T^d$ 
 the curve
 \be\label{linearr}
 r_a(t)=0, \;\; \theta_a(t)=\theta_a^0+t\omega_a  \;\text{for} \; a\in\A;
 \quad \xi_s(t)=\eta_s(t)=0  \;\text{for} \; s\notin\A\,,
 \ee
 is a solution  of the linear beam equation \eqref{linear}, lying on the torus $T^n_I$. Our goal is to perturb the 
 solutions \eqref{linearr} to solutions of the nonlinear equation \eqref{beam}. The first step  is to put the 
 nonlinear problem to a Birkhoff 
 normal form in the vicinity of a small torus $T^n_I$. To do this we write 
 $ I=\nu\yy, \yy  \in [c_*,1]^{\A} =:\D$, where 
 $0<\nu\ll1$ and $c_*\in(0,1/2]$ is a fixed parameter, 
 and in a small neighbourhood of $T^n_I$ make a symplectic change of
 variables which simplifies the Hamiltonian $H=H_2+P$. The corresponding result is obtained in 
 Sections~3-4 and may be loosely stated as follows:
 
 \begin{theorem}\label{NFTl}
 There exists a zero-measure Borel set $\Cc\subset[1,2]$ such that for any $m\notin\Cc$, any
  admissible set $\A$, $|\A|:=n\ge1$,  any $c_*\in(0,1/2]$ 
 and any analytic  nonlinearity  \eqref{g}, 
 there exist $\nu_0>0$ and $\beta_{*0}>0$, and for any $0<\nu\le\nu_0$, $0<\beta_*\le\beta_{*0}$ there exists a closed 
 domain $\tilde Q \subset \D$ which  is a semi-analytic 
 set,\footnote{More precisely, there is a polynomial $\Rc$ of $\sqrt\yy_j$, $1\le j\le n$,
 and a  $\delta>0$ such that $\tilde Q= \{\yy\in\D\mid \Rc  \ge \delta\}$.}
  such that $\meas(\D\setminus \tilde Q)\le  C\nu^{\beta_*}$, and for every $\yy\in\tilde Q$ there exists an analytic 
 symplectic change of variables 
 $$
 \tilde\Phi_\yy : (r',\theta',u,v) \mapsto (r, \theta, \xi, \eta), 
 $$
 $C^\infty$--Whitney smooth in $\yy$, with the following property:
 
 i)  The transformed Hamiltonian $H_\yy = H\circ\tilde\Phi_\yy$ reads 
 \be\label{transff}
\begin{split}
H_\yy&=
(\omega +\nu M\yy) \cdot r' +\frac12 
\sum_{a\in\L_\infty} \Lambda_a(\yy)(u_a^2+v_a^2)\\
&+ \frac\nu2   \Bigg(\sum_{b\in\L^e_f} \Lambda_b(\yy)      \Big( u_{b}^2 + v_{b}^2\Big)
 +  \Big\langle \widehat K(\yy) \left(\begin{array}{ll}u^h\\ v^h \\  \end{array}\right), \left(\begin{array}{ll}u^h\\ v^h \\  \end{array}\right)
 \Big\rangle\Bigg)
+\tilde f(r',\theta', \tilde\zeta; \yy)\,,
\end{split}
\ee
where
$\ 
\L= \L_\infty \cup \L_f\,,\quad  \L_f = \L_f^e\cup \L_f^h\,,\quad  u^h=(u_a, a\in\L_f^h), \; v^h=(v_a, a\in\L_f^h)\,,
$
and the decomposition $ \L_f = \L_f^e\cup \L_f^h$ depends on the component of the domain $\tilde Q$ (one of the sets
$ \L_f^e, \L_f^h$ may be empty). The matrix $M$ is explicitly defined in \eqref{Om},  and each 
$\La(\yy)$ is $C\nu(|a|+1)^{-2}$-close to $\la$. The function $\tilde f(\cdot;\yy)$ is analytic and is much smaller than
the quadratic part.

ii) The real symmetric matrix $\widehat K(\yy)$ smoothly depends on $\yy$ and for all $\yy$ satisfies 
$\|\widehat K(\yy)\|\le C \nu^{-c_1\beta_*}$. 
If $\L_f^h\ne\emptyset$,\footnote{otherwise the operator $J\widehat K$ is trivial.}
then the hamiltonian operator $J\widehat K(\yy)$ is hyperbolic, and the moduli of real parts of its 
eigenvalues are bigger than $C^{-1} \nu^{c_2\beta_*}$. It may be complex-diagonalised by means of a smooth in $\yy$
complex transformation $U(\yy)$ such that 
$\|U(\yy)\| + \|U(\yy)^{-1}\| \le C\nu^{-c_3\beta_*}
$.

iii) The matrix $\widehat K(\yy)$ and the domain $\tilde Q$ do not depend on the component $g_0$ of the function $g$. 
 \end{theorem}
 
 For exact statement of the normal form result see Theorem \ref{NFT}. 
 
 Applying to the normal form above an abstract KAM theorem for multidimensional PDEs, proved in \cite{EGK1}, we obtain the 
 main results of this work. To state them we  recall that a Borel subset ${\frak J}\subset \R^n_+$ is said to have 
  a positive density  at the origin if
 \be\label{posdens}
 \liminf_{\nu\to0}\frac{\meas(\fJ\cap \{x\in\R^n_+\mid \|x\|<\nu\})}{\meas\{x\in\R^n_+\mid \|x\|<\nu\}} >0\,.
 \ee
 The set $\fJ$ has the density one at the origin if the $ \liminf$ above equals one (so the ratio of the measures
 of the two sets converges to one as $\nu\to0$).

 \begin{theorem}\label{t72} There exists a  zero-measure 
 Borel set $\Cc\subset[1,2]$ such that for any\sa\  set $\A\subset\Z^d $, $|\A|=:n\ge1$,  any analytic 
 nonlinearity  \eqref{g}, any constant $\a_*>0$ and any $m\notin\Cc$ there exists 
 a Borel set $\fJ\subset \R^n_+$,  having  density one  at the origin, with the following property: 
 
 \noindent 
 There exist
 constants $C,c>0$,   a continuous    mapping
 $\ 
 U: \T^n\times \fJ   \to Y^R = Y^R_\A\oplus  Y^R_\L
 $ 
 (see \eqref{PR}), 
 analytic in the first argument,  satisfying
 \be\label{dist1}
 \big| U(\T^n\times\{I\}) - ( \sqrt{I}  \, e^{i \theta}\,, \sqrt{I}  \, e^{-i \theta}, 0) 
  \big|_{Y^R}
  \le C |I|^{1 -\a_*}\
  \ee
  (see \eqref{acc-ann}), 
 and a continuous 
 vector-function 
 \be\label{dist11}
  \om' : \fJ \to \R^n\,,\qquad 
 |\om'(I) -\om - MI | \leq C |I|^{1+ c\a_*}\,,
\ee
where the matrix  $M$ is the same as in \eqref{transff},  such that

i) for any $I\in \fJ$ and $\theta\in\T^n$  the  parametrised  curve
\be\label{solution}
 t \mapsto U(\theta +t\omega'(I), I)
\ee
is a  solution of the beam equation \eqref{beam2}. 
 Accordingly, for each $I\in \fJ$ the analytic $n$-torus $U(\T^n\times\{I\}) $
 is invariant for eq.~\eqref{beam2}.  
 
 ii) The set $ \fJ$ may be written as a countable  disjoint  union of positive-measure  Borel sets $ \fJ_j$, 
 such that the restrictions of the mapping $U$  to the sets $\T^n\times \fJ_j$ and of $\omega'$ to the sets 
 $\fJ_j$ are Whitney  $C^1$-smooth.
  
iii) The  solution \eqref{solution}  is linearly stable if and only if in \eqref{transff} the operator $\widehat K(\yy)$ is 
 trivial (i.e. the set  $\L_f^h$ is emty). 
 The set $\fJ_e$ 
  of $\yy$'s in $\fJ$ for which $\widehat K(\yy)$ is trivial is of positive measure, and it equals $\fJ$ if 
 $d=1$ or $| \A |=1$. For $d\ge2$ and for some choices of the  set $\A$, $| \A |\ge2$,  the complement $\fJ\setminus\fJ_e$ 
 has positive measure. 
   \end{theorem}
 
 We recall that for $d\le2$ every admissible set is\sa, but for higher dimension this is not the case.
 Still if $\A$ is admissible and $d\ge3$, then a weaker version of the theorem above is true:

 \begin{theorem}\label{t73} There exists a  zero-measure 
 Borel set $\Cc\subset[1,2]$ such that for any admissible set $\A\subset\Z^d $ with $d\ge3$ and 
 $|\A|=:n\ge1$,  any analytic 
 nonlinearity  \eqref{g}, any constant $\a_*>0$ and any $m\notin\Cc$ there exists 
 a Borel set $\fJ \subset \R^n_+$,  having positive density at the origin, such that all assertions of Theorem~\ref{t72} are true.
  \end{theorem}
  
  \begin{remark} \label{r_1}
  1) The torus $T^n_I$, invariant for the linear beam equation \eqref{beam2}${}_{G=0}$, 
  is of the size $\sim\sqrt I$. If $\a_*<1/2$, then the constructed invariant torus $U(\T^n\times\{I\})$
  of the nonlinear beam equation is a small perturbation of $T^n_I$ 
   since by \eqref{dist1} the   Hausdorff   distance between $U(\T^n\times\{I\})$ and 
    $T^n_I$ is smaller than $C |I|^{1 -\a_*}$. 
  
  2) Our result applies to eq. \eqref{beam} with any $d$. Notice that   for $d$ sufficiently large the global in time 
 well-posedness of this equation is unknown. 
 
  3) The construction of  solutions \eqref{solution} crucially depends on certain equivalence relation in 
  $\Z^d$, defined in terms of the set $\A$ (see \eqref{class}). 
   This equivalence is trivial if $d=1$ or $|\A|=1$ and is non-trivial  otherwise. 
   
   4) The operator $J \widehat K(\yy)$ is complex-conjugated to the hyperbolic part of a complex hamiltonian 
   operator $i J K(\yy)$, corresponding to the complex Birkhoff normal form \eqref{HNF} for the beam equation.
   The operator $K$ is symmetric and {\it real}. So it seems that ``typically" (for some $\yy$)  $i J K(\yy)$ has a
   nontrivial hyperbolic part, and accordingly  $J \widehat K(\yy)\ne0$. We cannot prove this, but discuss in 
   Appendix~B examples of sets $\A$ for which the operators  $i J K$ have nontrivial hyperbolic parts.

    5) The solutions \eqref{solution} of eq. \eqref{beam2},  written in terms of the $u(x)$-variable as solutions $u(t,x)$ of eq.~\eqref{beam}, 
  are $H^{d^*+1}$-smooth as functions of $x$ and analytic as functions of $t$. Here $d^*$ is
   a parameter of the   construction for which we can take any real number $>d/2$ (see \eqref{YC}).   The set $\fJ $ depends on
  $d^*$, so the theorem's assertion does not imply immediately that the  solutions $u(t,x)$ 
    are  $C^\infty$--smooth in $x$. 
  Still, since 
    $$
    -(\Delta^2+m) u = u_{tt}+g(x,u),
    $$
    where $g$ is an analytic function, then the theorems 
     imply by induction that the     solutions $u(t,x)$      define analytic curves
    $\R\to H^p(\T^d)$, for any $p$. In particular, they are smooth functions. 
  \end{remark}

 \noindent 
{\bf Notation.} {\it Abstract sets.} 
 We denote a cardinality of a set $X$ as $|X|$ or as $\,\# X$.
 \smallskip
 
 \noindent 
 {\it Matrices}.   For any matrix $A$, finite or infinite, we denote by ${}^t\!A$ the transposed matrix; 
 in particular, 
$  {}^t(a,b)= \left(\begin{array}{ll}a\\ b \\  \end{array}\right)$. 
If $A$ is a finite matrix, then $\|A\|$ stands for its operator-norm. 
 By $J$ we denote the symplectic 
 matrix $\left(\begin{array}{cc} 0&1\\-1&0\end{array}\right)$ as well as various block-diagonal
matrices \ diag$\,\left(\begin{array}{cc} 0&1\\-1&0\end{array}\right)$, while $I$ stands for  the identity 
matrix of any dimension. 
\smallskip

\noindent{\it Norms and pairings}. For a linear space $X$ of dimension $N\le\infty$, 
interpreted as a space of real or complex sequences, we denote  by  
 $\langle\cdot,\cdot\rangle$ the natural   bi-linear paring:  if 
 $X\ni v^j=(v^j_1,\dots, v^j_N)$,  $j=1,2$, then 
 $
 \langle v^1, v^2\rangle = \sum_j v^1_j v^2_j\,.
 $
 Finite-dimensional  spaces $X$ as above and the lattices $\Z^N$ 
  are given the Euclidean norm which 
we denote $|\cdot|$, and the corresponding distance. The tori are provided with the  Euclidean
distance. For $a\in\Z^N$ we denote $\langle a\rangle =\max(1, |a|)$.
\smallskip

\noindent{\it  Analytic mappings}.  We call analytic mappings between domains  in complex Banach 
spaces {\it holomorphic} to reserve the name {\it analytic} for mappings between domains  in real Banach 
spaces.  A holomorphic 
mapping is called {\it real holomorphic} if it maps real-vectors of the space-domain to real vectors of
the space-target. Note that when we work with spaces, formed by sequences of complex 2-vectors,
we use  two different reality conditions. The right one will be clear from the context. A mapping, defined 
on a closed subset of a Banach space is called analytic (or holomorphic) if it extends to an analytic (holomorphic) 
map, defined in some open neighbourhood of that set. 
\smallskip

\noindent{\it Parameters}. 
Our functions depend on a parameter $\yy\in\D$, where 
$
\D \subset \R^p
$
is a compact set (or, more generally, 
 a  bounded Borel set) of positive Lebesgue  measure, with a suitable  $p\in\N$.
 Differentiability of functions on $\D$
is understood in the sense of   Whitney. That is, $f\in C^k(\D)$ if it  extends to a $C^k$-smooth
function $\tilde f$ on $ \R^p$, and $|f|_{ C^k(\D)}$ is the infimum of  $|\tilde f|_{ C^k(\R^p)}$,
taken over all $C^k$-extensions $\tilde f$ of $f$.

\medskip
 \noindent 
{\bf Acknowledgments.} We are thankful for discussion to P.~Milman, L.~Parnovski and V.~\v{S}ver\'ak. 
Our research was supported by  l'Agence Nationale de la Recherche through the grant
  ANR-10-BLAN~0102.

\section{Small divisors} \label{s_2} 

\subsection{Non resonance of  basic frequencies}

In this subsection we assume that the set $\A\subset\Z^d$ is admissible, i.e. it only 
contains  integer vectors with different norms (see Definition \ref{adm}).\\
We consider the  vector of basic frequencies 
\be\label{-om}
\om\equiv\om(m)=(\om_a(m))_{a\in\A}\,, \quad m\in [1,2]\,,
\ee
where $\om_a(m)=\lambda=
\sqrt{|a|^4+m}$.  The goal of  this section   is to prove the following result: 
\begin{proposition}
\label{NRom}
Assume that $\A$ is an admissible subset of $\Z^d$ of cardinality $n$ included in $\{a\in\Z^d\mid |a|\leq N\}$. 
Then for any $k\in\Z^\A\setminus\{0\}$, any $\ka>0$ and  any $c\in \R$ we have
\begin{equation*}
 \meas\ \left\{m\in[1,2]\ \mid \
\left|\sum_{a\in\A} k_a\omega_a(m)+c\right|\leq {\ka}\right\}\leq C_n \frac{N^{4n^2} \ka^{1/n}}{|k|^{1/n}}\,,
\end{equation*} 
where $|k|:=\sum_{a\in\A}|k_a|$ and $C_n>0$ is a constant, depending only on $n$. 
\end{proposition}

The proof follows  closely that of Theorem 6.5 in
\cite{Bam03} (also see  \cite{BG06}); a weaker form of the result was obtained
earlier in \cite{bou95}. Non of  the constants $C_j$ etc. in this section depend on the set $\A$. 

\begin{lemma}\label{l_det} Assume that $\A\subset\{a\in\Z^d\mid |a|\leq N\}$.
For any $p\leq n= |\A|$, consider $p$ points $a_1,\cdots,a_p$ in $\A$.
Then the modulus of the  following determinant 
\begin{equation*}
D:=\left|
\begin{matrix}
\frac{d \om_{a_1}}{dm} & \der \null {a_2} & .& .&.&\der \null {a_p}
\\
\frac{d^2 \om_{a_1}}{dm^2} & \frac{d^2 \om_{a_2}}{dm^2} & .& .&.&\frac{d^2 \om_{a_p}}{dm^2}
\\
.& .& .& .& .&.
\\
.& .& .& .& .&.
\\
\der{p}{a_1}& \der{p}{a_ 2}& .& .&.&\der {p}{a_p}
\end{matrix}
\right| 
\end{equation*}
is bounded from below:
$$
|D|\geq C N^{-3p^2+p}\,,
$$
where $C=C(p)>0$ is a constant
depending  only on $p$.
\end{lemma}

\proof First note that, by explicit computation, 
\be\label{-9}
\frac{d^j\omega_i}{dm^j}= (-1)^{j} \Upsilon_j\(|i|^4+m\)^{\frac12 -j}\,, \qquad \Upsilon_j=\prod_{l=0}^{j-1} \frac{2l-1}2\,.
\ee
Inserting this expression in $D$,  we deduce by factoring from each $l-th$ column the term
$(|a_\ell|^4+m)^{-1/2}=\om_\ell^{-1}$, and from each $j-th$ row the term $\Upsilon_j$
 that the determinant, up to a sign, equals
\begin{eqnarray*}
\left[\prod_{l=1}^{p}\omega_{a_\ell}^{-1}\right]
    \left[\prod_{j=1}^{p}  \Upsilon_j         
    \right]  
\times
\left|
\begin{matrix}
1& 1& 1 &. & . & . & 1
\cr
x_{a_1}& x_{a_2}& x_{a_3}&.&.&.&x_{a_p}
\cr
x_{a_1}^2& x_{ a_2}^2& x_{a_3}^2&.&.&.&x_{a_p}^2
\cr
.& .& .& .& .&.&.
\cr
.& .& .& .& .&.&.
\cr
.& .& .& .& .&.&.
\cr
x_{a_1}^{p}& x_{a_2}^{p}& x_{a_3}^{p}&.&.&.&x_{a_p}^{p}
\end{matrix}
\right|,
\end{eqnarray*}
where we denoted  $x_{a}:=(|a|^4+m)^{-1}= \omega_{a}^{-2}$.
Since $|\om_{a_k}|\le2|a_k|^2\le 2 N^2$ for every $k$, the first factor  is bigger than $(2N^2)^{-p}$. The second is a constant, 
 while the third  is the Vandermond determinant, equal to 
\begin{equation*}
\prod_{1\leq l<k\leq p}(x_{a_\ell}-x_{a_k})=\prod_{1\leq l<k\leq p}
\frac{|a_k|^4-|a_\ell|^4}{\omega_{a_\ell}^2\omega_{a_k}^2} =: V\,.
\end{equation*}
Since $\A$ is admissible, then
$$
|V| \ge \prod_{1\leq l<k\leq p}
\frac{|a_k|^2+|a_\ell|^2}{\omega_{a_\ell}^2\omega_{a_k}^2}  \ge \(\frac14\)^{p(p-1)} N^{-3p(p-1)}\,,
$$
where we used that each factor is bigger than $\frac1{16} N^{-6}$ using again that $|\om_{a_k}|\le2|a_k|^2\le 2 N^2$ for every $k$. 
This yields  the assertion.
\endproof

\begin{lemma}\label{m1.1}
 Let
$u^{(1)},...,u^{(p)}$ be $p$ independent vectors in $\R^p$  of norm at most one, and
 let $w\in\R^p$ be any non-zero vector. Then there exists  $i\in[1,...,p]$ such that
$$
|u^{(i)}\cdot w|\geq
 C_p |w|  |\det(u^{(1)},\ldots,u^{(p)})|\,.
$$
\end{lemma}
\begin{proof}
Without lost of generality we may assume that $|w|=1$. 

Let $|u^{(i)}\cdot w|\le a$ for all $i$. Consider the $p$-dimensional parallelogram $\Pi$, 
generated by the vector  $u^{(1)},...,u^{(p)}$ in $\R^p$ (i.e., the set of all linear combinations
$\sum x_j u^{(j)}$, where $0\le x_j\le 1$ for all $j$). It lies in the strip of width $2pa$, perpendicular
to the vector $w$, and its projection to to the $p-1$-dimensional space, perpendicular to $w$, lies 
in the ball around zero of radius $p$. Therefore the volume of $\Pi$ is bounded by 
$C_p p^{p-1} (2pa)=C_p' a$. Since this volume equals $|\det(u^{(1)},\ldots,u^{(p)})|$, then 
$a\ge  C_p   |\det(u^{(1)},\ldots,u^{(p)})|$. This implies the assertion. 
\end{proof}

Consider vectors $\frac{d^i\omega}{dm^i}(m)$, $1\le i\le n$, denote 
$K_i= | \frac{d^i\omega}{dm^i}(m) |$ and set
$$
u^{(i)}= K_i^{-1}\frac{d^i\omega}{dm^i}(m) , \qquad 1\le i\le n\,.
$$
From \eqref{-9} we see that\footnote{In this section $C_n$ denotes any positive constant depending only 
on $n$.} 
$\ 
K_i\le  C_n$ for all\ $1\le i\le n\,
$
(as before, the constant does not depend on the set $\A$). Combining Lemmas~\ref{l_det} and \ref{m1.1}, we find that for any vector $w$ and any  $m\in[1,2]$ there exists $r=r(m)\le n$ such that 
\be\label{m1.2}
\begin{split}
\Big|\frac{d^r\omega}{dm^r}(m) \cdot w\Big| =K_r \big| u^{(r)}\cdot w\big| \ge 
K_r C_n |w| (K_1\dots K_n)^{-1}| D|\\
\ge C_n |w| N^{-3n^2+n}\,. 
\end{split}
\ee

Now we need the following result (see Lemma B.1 in \cite{H98}):

\begin{lemma}
\label{v.112}
Let  $g(x)$ be a $C^{n+1}$-smooth function on the segment [1,2] 
such that $|g'|_{C^n} =\beta$ and $\max_{1\le k\le n}\min_x|\p^k g(x)|=\sigma$. Then 
$$
\meas\{x\mid |g(x)|\le\rho\} \le C_n \(\frac{\beta}{\sigma}+1\) \(\frac{\rho}{\sigma}\)^{1/n}\,.
$$
\end{lemma}

Consider the function $g(m)=|k|^{-1}\sum_{a\in\A}k_a \om_a(m) +|k|^{-1}c$.
Then $|g'|_{C^n}\le C'_n$, and  
$\max_{1\le k\le n}\min_m|\p^k g(m)|\ge C_n N^{-3n^2+n}$ in view of  \eqref{m1.2}. 
Therefore, by Lemma \ref{v.112}, 
\begin{equation*}
\begin{split}
\meas\{m\mid |g(m)|\le \frac{\kappa}{|k|}  \} &\le  C_n N^{3n^2-n} \(\frac{\kappa}{|k|} N^{3n^2-n}\)^{1/n}\\
=&
C_n N^{3n^2+2n-1}
\(\frac{\kappa}{|k|} \)^{1/n}\,.
\end{split}
\end{equation*}
This implies the assertion of the proposition.

\subsection{Small divisors estimates}

We recall the notation \eqref{L+},    \eqref{-om}, and note the elementary  estimates 
\be\label{estimla}
\langle a\rangle^2 <   \la(m)< \langle a\rangle^2  +  \frac{m}{2 \langle a\rangle ^2}\qquad \forall\, a\in\Z^d\,,\ m\in[1,2]\,,
\ee
where $\langle a\rangle =\max(1,|a|^2)$.
  In this section we study  four type of linear combinations of the frequencies $\la(m)$:
\begin{align*}
D_0=&\om \cdot k, \quad k\in \Z^\A\setminus\{0\}\\
D_1=&\om \cdot k+\la, \quad k\in \Z^\A,\; a\in \L\\
D_2^\pm=&\om\cdot k+\la\pm\lb, \quad k\in \Z^\A,\; a, b\in \L\,.
\end{align*}
In subsequent sections they will become  divisors for our constructions, so we call these linear combinations
``divisors".

\begin{definition}\label{Res-kab} Consider independent formal variables $x_0, x_1, x_2,\dots$. Now  take
any divisor of the form $D_0$, $D_1$ or $D_2^\pm$, write 
there each $\omega_a, a\in\A$,  as $\lambda_a$, and then replace every  $\lambda_a, a\in\Z^d$, by
$x_{|a|^2}$. Then the divisor is
called resonant if the obtained algebraical sum of the variables $x_j, j\ge0$,  is zero. Resonant 
divisors are also called  trivial resonances.
\end{definition}

Note that a $D_0$-divisor cannot be resonant since $k\ne0$ and the set $\A$ is admissible;
a $D_1$-divisor $(k;a)$ is  resonant only if $a\in {\L_f}$, $|k|=1$ 
 and $\omega\cdot k=-\omega_b$, 
where  $|a|=|b|$. Finally, a $D_2^+$-divisor or a $D_2^-$ divisor with $k\ne0$ 
may be   resonant  only when $(a,b)\in {\L_f}\times {\L_f}$, while the divisors $D_2^-$
of the form $\la-\lb$, $|a|=|b|$, all are resonant. 
So there are   finitely many trivial resonances of the form $D_0, D_1, D_2^+$ and of the form $D_2^-$ with 
$k\ne0$, but infinitely many of them of the form $D_2^-$ with $k=0$. 
\smallskip

Our first  aim is to remove from the segment  $[1,2]=\{m\}$ a small subset to guarantee that for the remaining $m$'s 
  moduli of all non-resonant  divisors  admit  positive lower bounds.  
  Below in this section 
\be\label{agreement}
\begin{split}
&\text{constants $C, C_1$ etc. depend on the admissible set $\A$,}\\
&\text{while the exponents $c_1, c_2$ etc depend only on $|\A|$. Borel  }\\
&\text{sets $\Cc_\ka$ etc. depend on the indicated arguments and $\A$.}
\end{split}
\ee

 We begin with the easier divisors 
 $D_0$, $D_1$ and $D_2^+$. 

\begin{proposition}\label{D1D2}
Let $1\ge\ka>0$. There exists a Borel set $\Cc_\ka \subset[1,2]$ and positive constants 
$C$ (cf. \eqref{agreement}),  satisfying 
$\ 
\meas\  \Cc_\ka  \leq C \ka^{1/(n+2)} ,
$ 
 such that for all $m\notin\Cc_\ka$,  all $k$ and  all $a,b \in \L$ we have
\be\label{D0}
 |\om \cdot k|\geq  \ka {\langle k\rangle}^{-n^2}, \qquad
 \text{ except if  } k=0,
 \ee
 \be\label{D1}
 |\om \cdot k+\la|\geq \ka {\langle k\rangle}^{-3(n+1)^3}, \quad
\text{ except if the divisor  is a trivial resonance}, 
\ee
\be \label{D2}
|\om\cdot k+\la+\lb|\geq \ka{\langle k\rangle}^{-3(n+2)^3}, 
\text{ except if 
 the divisor   is a trivial resonance}.
\ee
Here 
${\langle k\rangle}=\max(|k|,1)$. 

Besides, for each $k\ne0$ there exists a set
${\frak A}^k_\ka$ whose measure is $\ \le C\ka^{1/n}$ such that for $m\notin {\frak A}^k_\ka$ 
we have 
\be\label{D22}|\om\cdot k+j  |\geq \ka{\langle k\rangle}^{-(n+1)n } 
\text{for all $j\in\Z$ }.
\ee
\end{proposition}
\proof 
We begin with the   divisors \eqref{D0}.  
By Proposition \ref{NRom} for any non-zero $k$   we have 
$$
\meas \{m\in[1,2]\mid |\om \cdot k|\leq  \ka  |k|^{-n^2} \} < C {\ka^{1/n}}{|k|^{-n-1/n}}  \,.
$$
Therefore the relation \eqref{D0} holds for all non-zero $k$ if $m\notin\frak A_0$, 
where $\meas\frak A_0\le  C  \ka^{1/n} \sum_{k\ne0}  |k|^{-n-1/n} =C  \ka^{1/n}$.

Let us consider the divisors \eqref{D1}. For $k=0$ the required estimate holds 
trivially. If $k\ne0$, then the relation, opposite to \eqref{D1} implies that $|\la|\le C|k|$.  So
we may assume that $|a|\le C|k|^{1/2}$. If $|a|\notin\{|s|\mid s\in\A\}$, then Proposition~\ref{NRom} with 
$n:=n+1$, $\A:=\A\cup \{a\}$ and $N=C|k|^{1/2}$ implies that 
\begin{equation*}
\begin{split}
\meas& \{m\in[1,2]\mid |\om \cdot k+\la|\leq \ka|k|^{-3(n+1)^3}\}\\
\le& C 
 \ka^{1/(n+1)} |k|^{2(n+1)^2- 3(n+1)^2 -\frac1{n+1}  }
\leq C\kappa^{1/(n+1)} |k|^{-(n+1)^2}\,. 
\end{split}
\end{equation*}
This relation with $n+1$ replaced by $n$ 
 also holds if $|a|=|s|$ for some $s\in\A$, but $\om \cdot k+\la$ is not a trivial resonant. 
Since for fixed $k$  the set$\{\la\mid |a|^2\leq C|k|  \}$ has cardinality less than $2C|k|$, then the relation 
$
|\om\cdot k+\la|\le\ka |k|^{-3(n+1)^3}
$
holds for a fixed $k$ and all $a$ if we remove from [1,2] a set of measure 
$\le C\ka^{1/(n+1)}|k|^{-(n+1)^2+1}\leq C\ka^{1/(n+1)}|k|^{-n-1}$. 
So we achieve that
 the relation \eqref{D1}  holds for all $k$ if we remove from $[1,2]$ a set $\frak A_1$ whose 
measure is bounded by $C\ka^{1/(n+1)} \sum_{k\ne0} |k|^{-n-1} =C\ka^{1/(n+1)}$.

For a similar reason there exist  a Borel set $\frak A_2$ whose
measure is bounded by $C\ka^{1/(n+2)}$ and  such that \eqref{D2} holds
for $m\notin \frak A_2$. Taking 
$\Cc_\ka = \frak A_0\cup  \frak A_1\cup \frak A_2$ we get \eqref{D0}-\eqref{D2}.
Proof of \eqref{D22} is similar. 
\endproof

Now we  control divisors $D_2^-=\om \cdot k+\la-\lb$.

\begin{proposition}\label{prop-D3}   
There exist positive constants $C,c, c_-$ and for $0<\ka$ 
there is  a Borel set $\Cc'_\ka \subset[1,2]$  (cf. \eqref{agreement}),  satisfying 
\be\label{meas-estim2}
\meas\ \Cc'_\ka \leq C \ka^{c},
\ee
 such that  for all $m\in [1,2]\setminus  \Cc'_\ka$, 
 all $k\ne0$ and  all $a,b \in \L$  we have
\be\label{D3}
R(k;a,b):=
|\om\cdot k +\la-\lb|\geq  \ka |k|^{-c_- }, 
\ee
 except if the divisor is a trivial  resonance 
\end{proposition}
\proof
We  may assume that 
 $|b|\geq |a|$. We get from  \eqref{estimla} that 
$$
|\la-\lb-(|a|^2-|b|^2)|\leq  {m}{|a|^{-2}}\leq 2 |a|^{-2}.
$$
Take any $\ka_0\in (0,1]$ and construct the set $\frak A^k_{\ka_0}$ as in Proposition~\ref{D1D2}.
Then $\meas  {\frak A}^k_{\ka_0}\le C\ka_0^{1/n}$ and  for  any 
 $m\notin {\frak A}^k_{\ka_0}$  we have 
$$
R:=
R(k;a,b)\ge  \big| \om\cdot k +|a|^2-|b|^2\big|-2|a|^{-2}\ge
 \ka_0|k|^{-(n+1)n} - 2 |a|^{-2}\,.
$$
So $R \ge\frac12 \ka_0|k|^{-(n+1)n}$ and  \eqref{D3} holds if 
$$
|b|^2\ge |a|^2 \ge 4\ka_0^{-1}|k|^{(n+1)n}=:Y_1. 
$$

If $|a|^2\le Y_1$, then 
$$
R \ge \lb - \la -C|k| \ge |b|^2-Y_1-C|k|-1. 
$$
Therefore  \eqref{D3} also holds if $|b|^2\ge Y_1 +C|k|+2$, and 
 it remains to consider the case when $|a|^2\le Y_1 $ and 
  $|b|^2\le Y_1 +C|k|+2$. That is (for any fixed non-zero $k$), 
consider the pairs $(\la,\lb)$, satisfying 
\be\label{above}
|a|^2\le Y_1,\qquad  |b|^2\le  Y_1+2+C|k| =:Y_2 \,.
\ee
There are at most $CY_1Y_2$ pairs like that. Since the divisor $\om\cdot k +\la-\lb$
 is not  resonant, then in view of Proposition~\ref{NRom} with $N =Y_2^{1/2}$
and $|\A|\le n+2$,
for any  $\tilde\ka>0$ there exists
a set ${\frak B}^k_{\tilde \ka}\subset [1,2]$, whose measure is bounded by
$$
C  \tilde\ka^{1/(n+2)} \ka_0^{-c_1} |k|^{c_2},\qquad c_j=c_j(n)>0,
$$
such that  $R \ge \tilde\ka$ if $m\notin   {\frak B}^k_{\tilde \ka}\, $
for all pairs $(a,b)$ as in \eqref{above} (and $k$ fixed). 

Let us choose $\tilde\ka =\ka_0^{2c_1(n+2)}$. Then 
$
\meas {\frak B}^k_{\tilde \ka}\le C \ka_0^{c_1}|k|^{c_2}
$
and $R\ge \ka_0^{2c_1(n+2)}$ for $a,b$ as in \eqref{above}. Denote
$  \frak C^k_{\ka_0}= \frak A^k_{\ka_0}\cup {\frak B}^k_{\tilde \ka}\, $. Then
$\meas  \frak C^k_{\ka_0}\le C \(\ka_0^{1/n} + \ka_0^{c_1}|k|^{c_2}\)$, and 
for $m$ outside this set and all $a,b$ (with $k$ fixed) 
 we have 
$
R\ge \min\(\frac12 \ka_0|k|^{-(n+1)n}, \ka_0^{2c_1(n+2)}\)\,.
$
We see that if $\ka_0=\ka_0(k)=2\ka^{c_3} |k|^{-c_4}$ with suitable $c_3,c_4>0$, then 
$$
\meas\(  \Cc'_\ka = \cup_{k\ne0}  \frak C^k_{\ka_0}\) \le C\ka^{c_3}  \,,
$$
and, if $m$ is outside $\Cc'_\ka$, 
$R(k;a,b)\ge \ka |k|^{-c_-}$ with suitable $c_->0$. 
\endproof

It remains to consider the  divisors $D_2^-$ with $k=0$,  i.e. $D_2^-=\la-\lb$. Such a
divisor is resonant if $|a|=|b|$. 
\begin{lemma}\label{lem:D3-k=0}
Let $m\in [1,2]$ and the divisor $D_2^-=\la-\lb$ is non-resonant, i.e. $|a|\neq|b|$. 
Then
$
\left| {\la-\lb}\right|\ge \frac1 4.$
\end{lemma}
\proof  We have 
\begin{align*}
\left|\la-\lb\right|= \frac{\left||a|^4-|b|^4\right|}{\sqrt{|a|^4+m}+\sqrt{|b|^4+m}}\geq  
\frac{|a|^2+|b|^2}{\sqrt{|a|^4+m}+\sqrt{|b|^4+m}}\ge  \frac 1 4.
\end{align*}
\endproof

By construction the  sets $\Cc_\ka$ and $\Cc'_\ka$ decrease with $\ka$. 
Let us denote 
\be\label{setC}
\Cc = \bigcap_{\ka>0} (\Cc_\ka \cup \Cc'_\ka)\,.
\ee
From Propositions \ref{D1D2}, \ref{prop-D3} 
and Lemma~\ref{lem:D3-k=0}  we get:

\begin{proposition}\label{prop-m} 
The set $\Cc$ is a Borel subset of $[1,2]$ of zero measure. For any $m\notin\Cc$ there exists 
$\ka_*=\ka_*(m)>0$ such that the relations \eqref{D0}, \eqref{D1}, \eqref{D2} and \eqref{D3} hold
with $\ka=\ka_*$.
\end{proposition}

In particular, if $m\notin\Cc$, then any of the divisors 
$$
\omega\cdot s,\;\; \omega\cdot s\pm\la,\;\;\omega\cdot s\pm\la \pm\lb ,\quad s\in\Z^d,\; a,b\in\L,
$$
vanishes only if this is a trivial resonance. If it is not, then its modulus 
 admits a qualified estimate from below.

 The zero-measure Borel set $\Cc$ serves a fixed admissible set $\A$, $\Cc=\Cc_\A$. But since the set of all
 admissible sets is countable, then replacing $\Cc$ by $\cup_\A\Cc_\A$ we obtain a zero-measure Borel set which 
 suits all admissible sets $\Cc$.  For further purposes we modify $\Cc$ as follows:
 \be\label{modif}
 \Cc=: \Cc\cup \{\tfrac43, \tfrac 53\}\,.
 \ee

\section{The normal form}\label{BNF}

In  Sections \ref{BNF} and \ref{s_4}  we construct a symplectic change of variable that puts the  Hamiltonian
 \eqref{H1} to a normal form, suitable 
   to apply the abstract KAM theorem that we have proved in \cite{EGK1}. Our notation mostly agrees with \cite{EGK1}. 
   Constants in the estimates may depend on the dimension $d$, but this dependence is not indicated. 
   
\subsection{Notation and statement of the result}\label{s3.1}

We start with recalling some notation from \cite{EGK1}. Let  $\L$ be any subset of $\Z^d$ 
 (it is not excluded that $\L=\Z^d$). 
 We   fix any constant\footnote{The constants in the estimates below may depend on $d^*$, but this dependence 
 never is indicated.}
 $$d^* > \tfrac{1}2 d\,, $$
and
for $\ga\in[0,1]$  denote by $Y_\gamma^\L$  the following weighted {  complex}
 $\ell_2$-space
\be\label{Y}
Y_\ga^\L=
\{\zeta^\L= {  \Big(
\zeta_s =  \left(\begin{array}{ll}\xi_s\\ \eta_s \\  \end{array}\right)
\in\C^2,}\ s\in \L\Big)   \mid \|\zeta^\L\|_\ga<\infty\} ,
\ee
where\footnote{We recall that $|\cdot|$ signifies the Euclidean norm.}
$$
\|\zeta^\L\|_\ga^2=\sum_{s\in\L}|\zeta_s|^2\langle s\rangle^{2d^*}e^{2\gamma |s|},\qquad
\langle s\rangle= \max (|s|,1).
$$
We will often drop the upper index $^\L$ and write $Y_\ga$ and $\zeta$ instead of
 $Y_\ga^\L$ and $\zeta^\L$.

 In a space $Y_\gamma=Y_\ga^\L$ we define the complex conjugation as the involution 
\be\label{inv}
\zeta={}^t(\xi, \eta)\mapsto {}^t(\bar\eta, \bar\xi)\,.
\ee
Accordingly, the real subspace of $Y_\gamma$ is the space 
\be\label{reality}
Y_\ga^R=  Y_\ga^{\L R} =
\Big\{
\zeta_s =  \left(\begin{array}{ll}\xi_s\\ \eta_s \\  \end{array}\right)\mid \eta_s=\bar\xi_s,
\ s\in \L\Big\} \,.
\ee
Any mapping defined on (some part of) $Y_{\ga}$ with values in a  complex Banach space with 
a given real part 
is called {\it real} if  it gives real values to real arguments.

\medskip

We denote by $\M_\ga$ the set of infinite symmetric matrices $A:\L\times \L\to \M_{2\times 2}$
  valued in the space of  $2\times 2$ matrices and  satisfying
$$
|A|_\ga := \sup_{a,b\in \L}|A_a^b| \max([a-b], 1)^{d_*} 
e^{\ga [a-b]}<\infty,
$$
where
$$
[a-b]=\min (|a-b|,|a+b|)
$$
(this is a pseudo-metric  in $\Z^d$). Let us define the operator 
$$
D=\text{diag}\{  \langle s\rangle I, s\in\L\}
	$$
 (here $I$ stands for the identity $2\times2$-matrix).
 We denote by $\M^D_\ga$ the set of infinite matrices  $A\in\M_\ga$ such that $DAD\in\M_\ga$, and  set
$$
|A|^D_\ga=|DAD|_\ga=
\sup_{a,b\in\L} 
\langle a\rangle \langle b\rangle  |A_a^b|  \max([a-b], 1)^{d_*} 
e^{\ga [a-b]}.
$$
We note that in \cite{EGK1} instead of the norm $|\cdot|^D_\ga$ we use the norm $|\cdot|^\varkappa_\ga$
which for $\varkappa=2$ is ``weakly equivalent"  to $ |\cdot|^D_\ga $ in the sense that 
$$
 |\cdot|^D_\ga\le  C_{\ga'} |\cdot|^2_{\ga'}\quad \forall\, \ga'>\ga\,,
 \qquad 
  |\cdot|^2_\ga\le  C_{\ga'}  |\cdot|^D_{\ga'}\quad \forall\, \ga'>\ga\,.
$$

\medskip

For a Banach space $B$ (real or complex) we denote
$$
\O_s(B)=\{x\in B\mid \|x\|_B<s\}\,,
$$
and for  $\sigma,\ga,\mu\in(0,1]$  we set
\begin{align*}
\T^n_\s=&\{\theta\in\C^n/2\pi\Z^n\mid |\Im \theta|<\sigma\},\\
\O^\ga(\s,\mu)=& \O_{\mu^2}(\C^n) \times \T^n_\s   \times \O_\mu(Y_\ga)=\{(r, \theta, \zeta)\},
\\
\O^{\ga\R}(\s,\mu)=&\O^\ga(\s,\mu)\cap \{\R^n\times \T^n \times Y^R_\ga\}.
\end{align*}
The introduced domains depend on the set $\L$. To indicate this dependence we will sometime 
write them as
$$
\O^\ga(\s,\mu) = \O^\ga(\s,\mu)^\L,\quad  \O^{\ga\R}(\s,\mu)=\O^{\ga\R}(\s,\mu)^\L.
$$
We will denote the points in $\O^\ga(\s,\mu)$ as $x=(r,\theta,\zeta)$. \\

\begin{example}\label{analyt}
If 
$\hat f=(\hat f_s, s\in\Z^d)\in Y_\sigma=Y_\sigma^{\Z^d}$, then the function 
 $f(y) = \sum \hat f_se^{is\cdot y}$ is a holomorphic vector-function on $\T^n_\sigma$
 and its norm is bounded by $C_d\|\hat f\|_\sigma$. Conversely, if  
  $f: \T^n_\sigma\to\C^2$ is a bounded holomorphic function, then 
   its Fourier coefficients satisfy  $|\hat f_s|\le\,$Const$\, e^{-|s|\sigma}$, so 
  $\hat f\in Y_{\sigma'}^{\Z^d}$ for any $\sigma'<\sigma$. 
\end{example}
\medskip

  Let $h:\O^0(\s,\mu)\times\D\to \C$ be a $C^1$-function, real holomorphic
 (see Notation)  in the first variable $x=(r,\theta,\zeta)$,
 such that for all $0\le\ga'\le\ga$ and all $\yy\in\D$ the gradient-map
 $$
 \O^{\ga'}(\s,\mu)\ni x\mapsto \nabla_\zeta f(x,\yy)\in Y_{\ga}$$
 and the hessian-map 
 $$
 \O^{\ga'}(\s,\mu)\ni x\mapsto \nabla^2_{\zeta} f(x,\yy)\in \M^D_{\ga}
 $$
 also are real holomorphic. We denote this set of functions by  ${\Tc}^{\ga,D}(\s,\mu,\D)= {\Tc}^{\ga,D}(\s,\mu,\D)^\L$.

  For a function $h\in {\Tc}^{\ga,D}(\s,\mu,\D)$ we define
 the norm 
 $$[h]^{\ga,D}_{\s,\mu,\D}$$
 through 
\be\label{schtuk}
\sup_{
\begin{subarray}{c}
0\le\ga'\le \ga\\ j=0,1
\end{subarray}}
\sup_{
\begin{subarray}{c}
x\in O^{\ga'}(\s,\mu)\\ \yy\in\D
\end{subarray}}
\max( |\partial^j_\yy h(x,\yy)|,
{ \mu} \|\partial^j_\yy \nabla_\zeta h(x,\yy)\|_{\ga'},
{\mu^2}|\partial^j_\yy \nabla^2_\zeta h(x,\yy)|_{\ga'}^{D}).
\ee
For any function $h\in  {\Tc}^{\ga,D}(\s,\mu,\D)$ we denote by $h^T$ its Taylor polynomial 
at $r=0, \zeta=0$, linear in $r$ and quadratic in $\zeta$:
$$
h(x,\yy)=h^T(x,\yy)+ O(|r|^2+\|\zeta\|^3+|r|\|\zeta\|).  
$$

For $\L= \Z^d$ we  denote 
\be\label{Tga}
 {\Tc}^{\ga,D}(\mu) = \{ f\in {\Tc}^{\ga,D}(\s,\mu,\D)^{\Z^d}: f=f(\zeta)
 \}
 \ee
 (i.e., $f$ is independent from 
$\theta, r$ and $\yy$). The norm \eqref{schtuk}, restricted to the space 
  $ {\Tc}^{\ga,D}(\mu) $,   will be denoted $[h]^{\ga,D}_{\mu}$. 
   
Let $P$ be the hamiltonian function defined in \eqref{H1}.

\begin{lemma}\label{lemP}
$ P\in {\Tc}^{\ga_*,D}({\mu_*}) $ for suitable $\ga_*, \mu_*\in(0,1]$,  depending on the nonlinearity $g(x,u)$. 
\end{lemma}

Lemma in proven in Appendix A.

The goal of this section is to get a normal form for the Hamiltonian $H_2+P$ of the beam equation, written in the form 
\eqref{beam2}, 
 in toroidal domains in the spaces $Y_\ga^{\Z^d}$
which are complex  neighbourhoods of the finite-dimensional  real tori
\be\label{tor}
T_\yy=
\{\zeta=({}^t(\xi_s, \bar\xi_s), s\in\Z^d)\mid  |\xi_a|^2=\nu\yy_a \;\text{if}\; a\in\A,\ \xi_s=0 \;\text{if}\; s\in\L\}\,,
\ee
invariant for the linear equation. 
Here $\nu>0$ is small and $\yy=(\yy_a,a\in\A)$ is a vector-parameter of the problem, belonging to the
domain
$
\D=[c_*,1]^\A.
$
We arbitrarily enumerate the points of $\A$, i.e. write $\A$ as 
\be\label{labA}
\A = \{a_1,\dots, a_n\}\,,
\ee
and accordingly write $\D$ as 
\be\label{DDD}
\D=[c_*,1]^n\,.
\ee

In the vicinity of a torus \eqref{tor} in the space $Y_\gamma^{\Z^d}$ 
we pass from the complex variables $(\zeta_a, a\in\A)$, to the corresponding complex
action-angles $(I_a, \theta_a)$, using the relations
\be\label{ac-an}
\xi_a=\sqrt I_a e^{i\theta_a},\qquad \eta_a=\sqrt I_a e^{-i\theta_a}\,,\quad a\in\A\,.
\ee
Note that in the variables $(I,\theta,\xi,\eta)$, where $I=(I_a, a\in\A)$, $\xi=(\xi_b, b\in\L)$ etc, the involution
\eqref{inv} reads
\be\label{inv1}
(I,\theta,\xi,\eta)\to (\bar I,\bar\theta,\bar\eta,\bar\xi)\,.
\ee
So a vector $(I,\theta,\xi,\eta)$ is real if $I=\bar I, \theta=\bar\theta, \xi=\bar\eta$. 

The complex  toroidal vicinities of the tori $T_\rho$ (see \eqref{tor}) in the space $Y_\gamma^{\Z^d}$ 
 will be of the form 
\be\label{a-a}
\Tg=\Tg(\nu,\sigma,{\mu},\gamma)=\{\zeta\mid |I
-   \nu\rho|<\nu  c_*^2 {\mu}^2, |\Im\theta|<\sigma, 
\|\zeta^\L\|_\ga<\nu^{1/2}c_*{\mu}\}\,,
\ee
where $I=(I_a, a\in\A)$, $\theta=(\theta_a, a\in\A)$ and $\zeta^\L=\{\zeta_s, s\in\L\}$. 
Since $c_*\leq \yy_j\leq 1$ for each $j$, then
\be\label{nbh}
\Tg(\nu,\sigma, \mu, \gamma) \cap Y^R_\gamma \subset 
\{\zeta\in Y^R_\gamma \mid \dist_\gamma (\zeta,T_\yy) <C\sqrt\nu \mu\}
\ee
if $\mu\le1$, where $C>0$ is an absolute constant. 

We recall (see \eqref{L+}) that we have split the set $\L$ to the union
$\L=\L_f\cup\L_\infty$. Accordingly, we will write vectors $\zeta^\L=\{\zeta_s, s\in\L\}$
as $\zeta^\L=(\zeta_f, \zeta_\infty)$, where 
$$
\zeta_f=\{\zeta_s, s\in\L_f\},\qquad \zeta_\infty=\{\zeta_s, s\in\L_\infty\} \,.
$$

\begin{proposition}\label{thm-HNF} There exists a  zero-measure Borel set $\Cc\subset[1,2]$
such that for any admissible   set  $\A$, any  $c_*\in(0,1/2]$ and   $m\notin\Cc$ we can find 
  real numbers   $\ga_* ,\nu_0\in(0,1]$, where  $\ga_*$ depends only on $g(\cdot)$ 
  and $\nu_0$ depends on $\A, m, c_*$ and $g(\cdot)$, such that 

(i) For $0<\nu\le\nu_0$ and  $\yy\in\D=[c_*,1]^n$ there exist  real holomorphic  transformations
$$
\Phi_\yy:  \O^{\ga} \Big({\frac 12}, {\frac{c_*}{2\sqrt2}}\Big)^\L \to \Tg(\nu,  1,1,\ga)\,,\qquad 0\le \ga\le\ga_*\,,
$$
which do not depend on $\ga$ in the sense that they 
 coincide on the set $\O^{\ga_*}({\frac 12}, \frac{c_*}{2\sqrt2})^\L$, 
 and are 
 diffeomorphisms on their images, analytically depending on $\yy$
 and transforming  the symplectic structure $-id\xi\wedge d\eta$ 
 on $\Tg(\nu,  1,1,\ga_*)$ to the 2-form 
$$
-\nu\sum_{\ell\in\A}dr_\ell\wedge d\tl \ -i\ \nu\sum_{a\in\L}d\xi_a\wedge d\eta_a.
$$
The 
 change of variable $\Phi_\yy$  is close to the {  scaling by the factor $\nu^{1/2}$ } on the $\L_\infty$-modes 
but not on the $(\A\cup {\L_f})$-modes, where it is close to a certain affine 
  transformation, depending 
on $\theta$. For each $\gamma$, $\Phi_\yy$ as  a function of $\yy$  holomorphically
 extends to the complex domain
\be\label{Dset}
\Da=\{ \yy\in \C^\A\mid |\Im\yy_j|<c_1, \, c_*-c_1<\Re\yy_j<1+c_1\ \forall j\in\A\}\,,\;\;\;
0<c_1<c_*
 \,.
\ee

(ii) $\Phi_\yy$ puts the Hamiltonian  $H_2+P$ (see \eqref{H1}) 
 to a normal form in the following sense:
\footnote{The factor $\nu^{-1}$ in the l.h.s. of \eqref{HNF} corresponds to $\nu$ in the transformed 
symplectic structure in item (i). So the Hamiltonian of the transformed equations with respect to the symplectic 
structure $-dr\wedge d\theta-i\,d\xi\wedge d\eta$ is given by the r.h.s. of \eqref{HNF}. 
}
\be\label{HNF}
\begin{split}
\frac1{\nu} (H_2+P)\circ\Phi_\yy =\Omega(\yy)\cdot r + \sum_{a\in\L_\infty}\Lambda_a (\yy)\xi_a\eta_a
+\frac{\nu}2\, \langle K(\yy) \zeta_f, \zeta_f\rangle
+ f( r,\theta,\zeta;\rho)\,.
\end{split}
\ee
Here the vector  $\Om$ and the scalars 
$\La, a\in\L_\infty$, are affine functions of $\rho$. 
They are 
 defined by  relations  \eqref{Om}, \eqref{Lam},  and 
  after the natural extension to the complex domain 
  $\D_{c_1}$ satisfy there the estimates
\be\label{-10}
|\Om(\yy)-\om|\le C_1\nu,\;\; |\La(\yy)-\la(\yy)|\le C_1\nu\langle a\rangle^{-2}\,.
\ee

(iii) $K$ is a  symmetric  real matrix, acting on vectors      $\zeta_f$. It is a
quadratic polynomial of $\sqrt\yy=(\sqrt\yy_1,\dots,\sqrt\yy_n)$, defined 
by  relation \eqref{K}, and  satisfies
 \be\label{est11}
 \| K(\yy)\| \le C_2\qquad\forall\, \yy\in\Da.
 \ee  
The matrix does not depend on the component $g_0$ of the nonlinearity $g$.

(iv) The remaining  term $f$ belongs  to ${\Tc}^{{\ga}, D}({\frac12}, {{\frac{c_*}{2\sqrt 2}}}, \D)^\L$, analytically extends to 
$\rho\in \D_{c_1}$ and for each $0\le\ga\le\ga_*$ this analytic  extension 
satisfies 
\be\label{est}
[f]^{{\ga},D}_{{\frac12},{{\frac{c_*}{2\sqrt 2}}},\D_{c_1}} \le C_2\nu   \,, \qquad [f^T]^{{\ga},D}_{{{\frac12}},{{\frac{c_*}{2\sqrt 2}}},\D_{c_1}} \le C_2
\nu^{3/2} \,.
\ee

The constants $C_1$ and $c_1$ depend only on $\A$ and $c_*$, 
 while   $C_2$  also depend on $m$ and  the function $g(x,u)$.
\end{proposition}

Note that \eqref{-10} and the Cauchy estimate imply that 
\be\label{labar0} 
 {  |\partial_\yy \Lambda_a(\yy)|} 
\leq C_3\nu\langle a\rangle^{-2}\ \text{  for } a\in\L_\infty\,, \ \yy\in\D\,.
\ee

The rest of this section is devoted to the proof of Proposition \ref{thm-HNF}.

\subsection{Resonances and the Birkhoff procedure}
Let us write  the quartic part $H^4=H_2 + P_4$  of the Hamiltonian $H$ (see \eqref{H1}, 
 \eqref{quatr})
 in the complex variables $\zeta=\zeta^{\Z^d}=\{{}^t(\xi_s, \eta_s),s\in\Z^d\}$:
\begin{align*}
H_2=& 
 \sum_{s\in\Z^d}\ls \xi_s\eta_s,\\
P_4=& (2\pi)^{-d}\sum_{(i,j,k,\ell)\in\J}\frac{(\xi_i+\eta_{-i})(\xi_j+\eta_{-j})(\xi_k+\eta_{-k})(\xi_\ell+\eta_{-\ell})}{4\sqrt{\li\lj\lk\lel}}\,,
\end{align*}
where $\J$ denotes the zero momentum set:$$\J:=\{(i,j,k,\ell)\subset\Z^d\mid i+j+k+\ell=0\}.$$

We decompose $P_4=P_{4,0}+P_{4,1}+P_{4,2}$ according to 
\begin{align*}
P_{4,0}=& \frac 1 4 (2\pi)^{-d}\sum_{(i,j,k,\ell)\in\J}\frac{ \xi_i\xi_j\xi_k\xi_\ell +\eta_i\eta_j\eta_k\eta_\ell}{\sqrt{\li\lj\lk\lel}},\\
P_{4,1}=& (2\pi)^{-d}\sum_{(i,j,k,-\ell)\in\J}\frac{ \xi_i\xi_j\xi_k\eta_\ell +\eta_i\eta_j\eta_k\xi_\ell}{\sqrt{\li\lj\lk\lel}},\\
P_{4,2}=& \frac 3 2 (2\pi)^{-d}\sum_{(i,j,-k,-\ell)\in\J}\frac{ \xi_i\xi_j\eta_k\eta_\ell }{\sqrt{\li\lj\lk\lel}}\,,
\end{align*}
and denote by $R_5$ the remainder term of the the nonlinearity $P$. I.e.
\be\label{PPP}
P=P_4+R_5.
\ee
%
Finally we define
$$
\J_2= \{ (i,j,k,\ell)\subset\Z^d\mid
(i,j,-k,-\ell)\in\J, \; \sharp \{i,j,k,\ell\}\cap \A \geq 2\}\,.
$$
For later use we note that, 
by Proposition~\ref{prop-m}, 
\begin{lemma}\label{res-mon}If $m\notin\Cc$, then there exists $\ka(m)>0$ such that for all $(i,j,k,\ell)\in \J_2$
\begin{align*}
|\li+\lj+\lk-\lel|&\geq \ka(m)\, ;\\
|\li+\lj-\lk-\lel|&\geq \ka(m), \quad \text{except if } \{|i|,|j|  \}=\{|k|,|\ell| \}\, .
\end{align*}
\end{lemma}

For $\ga\geq 0$ we consider the phase space $Y_\ga=Y_\ga^{\Z^d}$, defined as in Section~\ref{s3.1}, and 
endowed it with the  symplectic structure $- i\sum d\xi_k\wedge d\eta_k$. 
Since $d^*>d/2$, then the spaces $Y_\gamma$ are algebras with respect to the convolution,
 see Lemma~1.1 in \cite{EK08}. This implies the following result, where $\langle\cdot,\cdot\rangle$ 
 stands for the complex-bilinear paring of $\C^{2r}$ with itself:

\begin{lemma}\label{XPanalytic}
Let $\ga\geq 0$, $r\in\N$ and  $ P^r$ be a real  homogeneous polynomial on $Y_\gamma$ 
of degree $r$,
$$
 P^r(\zeta)
=\sum_{(j_1,\dots j_r)\in (\L)^r}\langle a_{j_1,\dots,j_r},\, \zeta_{j_1}\otimes \dots \otimes\zeta_{j_r}\rangle \,,
$$
where $a_{j_1,\dots,j_r}\in \C^{2}\otimes\dots\otimes\C^2$ ($r$ times), 
$|a_{j_1,\dots,j_r} |\le M$, and $a_{j_1,\dots,j_r}=0$ unless $j_1+\dots+j_r=0$.
 Then the gradient-map $\nabla P^r(\zeta)$  
 satisfies
 $\|\nabla P^r(\zeta)\|_\ga  \le M C^{r-1}\|\zeta\|_\ga^{r-1}$. So the flow-maps $\Phi^t_{P^r}$, $|t|\le1$, 
 of the  hamiltonian vector-field $X_ {P^r}=iJ\nabla P^r$
 are well defined real holomorphic mappings on a ball $B_\ga(\delta)=\{\|\zeta\|_\ga<\delta\}$, $\delta=\delta(M)>0$, and satisfy
 there
 $$
 \|\Phi^t_{P^r}(\zeta)-\zeta\|_\ga \le C_1 \|\zeta\|_\ga^{r-1}\,,\quad C_1= C_1(M) \,.
 $$
\end{lemma}

 \begin{corollary}\label{c1}
 Consider the polynomial $Q^r(\zeta)= P^r(D^-(\zeta))$, where $D^-$ is the diagonal matrix $\,\text{diag}\,\{|\lambda_s|^{-1/2}I\}$. 
Then the Hessian-map  $\nabla_\zeta^2 Q^r  \in \M^D_\ga$ and
 $|Q^r|_\ga^D\leq M C^{r-2}\|\zeta\|_\ga^{r-2}$ for any $\ga\geq 0$.
 In particular $Q\in {{\Tc}}^{\ga, D}(\mu)$ for any $0<\mu\leq 1$ (see \eqref{Tga}). 
 \end{corollary}
 
 Note that the corollary applies to the monomials, forming $P_4$ (e.g. to $P_4  $). 
 
 \begin{proposition}\label{Thm-BNF}
For   $m\notin\Cc$    there exists a real holomorphic and 
symplectic change of variable $\tau$ in a neighbourhood of the origin
 in $Y_\ga$ that puts the Hamiltonian $H = H_2+P$ into its partial Birkhoff normal form up to order five
in the sense that it  removes from $P_4$ all
non-resonant terms, apart from those who are cubic or quartic in directions 
of $\L$. More precisely, for  $0\le\ga\le\ga_*$, where $\ga_*$ is as in  Lemma~\ref{lemP}, and for 
  a suitable $\delta(m)\le\mu_*$ (depending on $m$ and  $g(x,u)$), 
  the mapping $\tau$ satisfies 
 \be\label{esti1}
 \|\tau^{\pm1}(\zeta)-\zeta\|_{\ga} \le C(m) \|\zeta\|_{\ga}^3
   \qquad \forall\,\zeta\in B_{\ga}(\delta(m))\,.
 \ee
 It transforms the Hamiltonian $H_2+P=H_2+P_4+R_5$  
  as follows:
\be\label{trans}
(H_2 +P)\circ \tau= H_2 + Z_4+ Q_4^3+ R_6^0 +R_5\circ \tau\,,
\ee
where 
\begin{align*}
Z_4=&\frac 3 2(2\pi)^{-d}\sum_{\substack{(i,j,k,\ell)\in\J_2 \\ \{|i|,|j|  \}=\{|k|,|\ell| \}}}\frac{ \xi_i\xi_j\eta_k\eta_\ell }{\li\lj},
\end{align*}
and 
$Q_4^3=Q_{4,1}+Q_{4,2}$ with\footnote{The upper index 3 signifies that $Q_4^3$ is at least cubic
in the transversal directions $\{\zeta_a, a\in \L\}$.}
\begin{align*}
Q_{4,1}=&(2\pi)^{-d} \sum_{(i,j,-k,\ell)\not\in\J_2}\frac{ \xi_i\xi_j\xi_k\eta_\ell +\eta_i\eta_j\eta_k\xi_\ell}{\sqrt{\li\lj\lk\lel}},\\
Q_{4,2}=& \frac 3 2(2\pi)^{-d}\sum_{(i,j,k,\ell)\not\in\J_2}\frac{ \xi_i\xi_j\eta_k\eta_\ell }{\sqrt{\li\lj\lk\lel}}\,.
\end{align*}
The functions $Z_4, Q_4^3, R_6^0, R_5\circ\tau$ are real holomorphic on $B_{\ga}(\delta(m))$. Besides
 $R_6^0$ and $R_5\circ\tau$ are, respectively, functions of order 6 and 5 at the origin. 
 For any $0<\mu\le \delta(m)$ the functions 
  $Z_4, Q_4^3, R_6^0$ and $R_5\circ\tau$
belong to ${\Tc}^{{\ga},D}(\mu)$ (see \eqref{Tga}),  and 
\be\label{Z4}
\big[Z_4\big]_\mu^{{\ga},D} +\big[Q_4^3\big]_\mu^{{\ga},D}\le C\mu^4\,,
\ee
\be\label{R6}
\big[R_6^0\big]_\mu^{{\ga},D} \le C \mu^6\,,
\ee
\be\label{R66}
\big[R_5\circ\tau\big]_\mu^{{\ga},D} \le C \mu^5\,,
\ee
where $C$ depends on $\A$, $m$ and $g$. 
\end{proposition}
\proof
We use the classical Birkhoff normal form procedure. We construct the transformation 
$\tau$ as the time one flow $\Phi^1_{\chi_4}$ of  a Hamiltonian $\chi_4$,  given by
\begin{align}\label{chi4}\begin{split}
\chi_4=& -\frac{i } 4 (2\pi)^{-d}\sum_{(i,j,k,\ell)\in\J}\frac{ \xi_i\xi_j\xi_k\xi_\ell -\eta_i\eta_j\eta_k\eta_\ell}{(\li+\lj+\lk+\lel)\sqrt{\li\lj\lk\lel}}\\
&-i(2\pi)^{-d}\sum_{(i,j,-k,\ell)\in\J_2}\frac{ \xi_i\xi_j\xi_k\eta_\ell -\eta_i\eta_j\eta_k\xi_\ell}{(\li+\lj+\lk-\lel)\sqrt{\li\lj\lk\lel}}\\
&- \frac {3i} 2 (2\pi)^{-d} \sum_{\substack{(i,j,k,\ell)\in\J_2 \\ \{|i|,|j|  \}\neq\{|k|,|\ell| \}}} \frac{ \xi_i\xi_j\eta_k\eta_\ell }{(\li+\lj-\lk-\lel)\sqrt{\li\lj\lk\lel}}
\end{split}\end{align}
By  Lemma  \ref{res-mon} 
and Lemma \ref{XPanalytic}
for  $m\notin\Cc$  the vector-field   $X_{\chi_4}$  
is   real holomorphic in $Y_{\ga}$   and of order three at the origin. Hence $\tau=\Phi^1_{\chi_4}$ is a real holomorphic 
and  symplectic change of coordinates, defined  in $B_{\ga}(\delta(m))$, a neighbourhood  of the origin in $Y_{\ga}$. 
By Lemma~\ref{XPanalytic} it satisfies \eqref{esti1}. 

{ 
Since the Poisson bracket, corresponding to the symplectic form $-id\xi\wedge d\eta$  is
$
\{F,G\} = i \langle \nabla_\eta F, \nabla_\xi G\rangle -i \langle \nabla_\xi F, \nabla_\eta G\rangle,
$
and since $\nabla_{\eta_s}H_2=\lambda_s \xi_s$,   $\nabla_{\xi_s}H_2=\lambda_s \eta_s$, then} we 
 calculate
\begin{align*}
\{H_2,\chi_4\}= -& \frac{1 } 4 (2\pi)^{-d}\sum_{(i,j,k,\ell)\in\J}\frac{ \xi_i\xi_j\xi_k\xi_\ell +
\eta_i\eta_j\eta_k\eta_\ell}{\sqrt{\li\lj\lk\lel}}\\
-&(2\pi)^{-d}\sum_{(i,j,-k,\ell)\in\J_2}\frac{ \xi_i\xi_j\xi_k\eta_\ell +\eta_i\eta_j\eta_k\xi_\ell}{\sqrt{\li\lj\lk\lel}}\\
-& \frac {3} 2 (2\pi)^{-d}
\sum_{\substack{(i,j,k,\ell)\in\J_2 \\ \{|i|,|j|  \}\neq\{|k|,|\ell| \}}}   \frac{ \xi_i\xi_j\eta_k\eta_\ell }{\sqrt{\li\lj\lk\lel}}\,.
\end{align*}
Therefore
\begin{align*}
(H_2+ P_4)\circ \tau=& H_2+ P_4  { -} \{H_2,\chi_4\}  { -}  \{P_4,\chi_4\}\\
&+\int_0^1 (1-t)\{\{H_2+ P_4,\chi_4\},\chi_4\}\circ \Phi_{\chi_4}^t \dd t\\
=& H_2 + Z_4+ Q_4^3+ R_6^0
\end{align*}
with $Z_4$ and $Q_4^3$ as in the statement of the proposition and 
$$
R_6^0= \{P_4,\chi_4\}+ \int_0^1 (1-t)\{\{H_2+ P_4,\chi_4\},\chi_4\}\circ \Phi_{\chi_4}^t \dd t.$$
 The reality of the functions $Z_4$ and $Q_4^3$ follow from the explicit formulas for them,
while the inclusion of these functions to ${\Tc}^{{\ga}, D}(\mu)$ for any $0<\mu\leq 1$
 and the estimate \eqref{Z4} hold
by Corollary~\ref{c1}. Concerning $R_6^0$, by construction this is  a holomorphic  function of order $\ge6$
at the origin.  Its reality follows from the equality \eqref{trans}, where all other functions are real. 
The  inclusion  $R_6^0\in{\Tc}^{{\ga}, D}(\mu)$ for any $0<\mu\le \delta(m)$  and the 
estimate \eqref{R6} follow from the following three facts:

\begin{itemize}
\item[(i)] $\{H_2+ P_4,\chi_4\}=Z_4+ Q_4^3$ and  $\chi_4$ belong to ${\Tc}^{{\ga}, D}(1)$ by Corollary \ref{c1}.
\item[(ii)] $\{{\Tc}^{{\ga}, D}(1),{\Tc}^{{\ga}, D}(1)\} \in {\Tc}^{{\ga}, D}(\tfrac12)$ (see Proposition~2.6
in  \cite{EGK1}).
\item[(iii)] ${\Tc}^{{\ga}, D}(\tfrac12)\circ \Phi^t_{\chi_4} \in  {\Tc}^{{\ga}, D}(\tfrac12\delta(m))$.
\end{itemize}
In  \cite{EGK1}, Proposition~2.7, and \cite{Kuk3}, Lemma~10.7, the assertion (iii) is proven for a special
class of Hamiltonians $\chi_4$, but the proof easily generalises to Hamiltonian $\chi_4$ as above.
\medskip

Finally,  since by 
 Lemma \ref{lemP} 
 the function $R_5$ belongs to ${\Tc}^{{\ga}, D}(\mu_*)$,
 then in view of (iii)  $R_5\circ\tau\in {\Tc}^{{\ga}, D}(\frac12\delta(m))$. 
 Re-denoting $\ \frac12\delta(m)$ to $\delta(m)$ we get \eqref{Z4}-\eqref{R66}. 
\endproof

 Due to \eqref{esti1}, if $\zeta\in\Tg(\nu,1/2,1/2,{\ga})$, $0\le\ga\le \ga_*$, 
 where $\nu\le C^{-1} \delta(m)^2$ and $C$ is an absolute 
 constant  (see \eqref{a-a}), 
 then $\|\tau^{\pm1}(\zeta)-\zeta\|_\ga\le C'(m)\nu^{\frac 3 2}$. Therefore
\be\label{prop1}
\tau^{\pm1} ( \Tg(\nu,1/2,1/2,\ga))\subset  \Tg(\nu,1 ,1,\ga)\,,
\ee
provided that $\nu\le C^{-1} \delta(m)^2$ and  $\yy\in\Da$, where $c_1=c_1(\A,m,g(\cdot),c_*)$ is sufficiently small.

\subsection{Normal form, corresponding to  admissible   sets $\A$} \label{s_3.3}
Everywhere in  Section \ref{s_3.3}--\ref{s_44}
 the set  $\A$ is assumed to be  admissible 
in the sense of Definition~\ref{adm}.

The Hamiltonian $Z_4$ contains the integrable part formed by monomials of the form $ \xi_i\xi_j\eta_i\eta_j=I_iI_j$ that only depend on the actions $I_n=\xi_n\eta_n$, $n\in\Z^d$. Denote it $Z_4^+$ and denote the rest $Z_4^-$. It is not hard to see 
that 
\be\label{Z4+}
Z_4^+=\frac 3 2(2\pi)^{-d} \sum_{\ell\in\A,\ k\in\Z^d} (4-3\delta_{\ell,k})\frac{I_\ell I_k}{\lambda_\ell\lambda_k}.
\ee

To calculate  $Z_4^-$, we  decompose it  according to the number of indices in $\A$:
a monomial $\xi_i\xi_j\eta_k\eta_\ell $ is in $Z_4^{-r}$ ($r=0,1,2,3,4$) if $(i,j,-k,-\ell)\in\J$ and $\sharp \{i,j,k,\ell\}\cap\A=r$. 
We note that, by construction, $Z_4^{-0}=Z_4^{-1}=\emptyset$. 

Since  $\A$ is admissible, then  in view of Lemma \ref{res-mon}
for  $m\notin\Cc$ the set $Z_4^{-4}$ is empty. The set $Z_4^{-3}$ is empty as well:
 
\begin{lemma}
If $m\notin\Cc$,  then
$Z_4^{-3}=\emptyset.$
\end{lemma}
\proof
Consider any term $\xi_i\xi_j\eta_k\eta_\ell \in Z_4^{-3}$, i.e.  $\{i,j,k,\ell\}\cap \A=3$. Without lost of generality we can assume that $i,j,k\in\A$ and $\ell\in\L$. 
Furthermore we know that $i+j-k-\ell=0$ and
$\{|i|,|j|\}=\{|k|,|\ell|\}$.
In particular  we must have $|i|=|k|$ or $|j|=|k|$ and thus, since $\A$ is admissible, $i=k$ or $j=k$.
Let for example, $i=k$. Then  $|j|=|\ell|$.
Since  $i+j=k+\ell$  we conclude that $\ell=j$ which contradicts our hypotheses.
\endproof

Recall that the finite set $ {\L_f}\subset\L$ was defined in \eqref{L+}.  The mapping
\be\label{lmap}
\ell: \L_f \to \A,\quad a\mapsto \ell(a)\in\A \text{ if } \ |a|=|\ell(a)|,
\ee
is well defined since the set $\A$ is admissible.  Now we define two subsets 
of $ {\L_f}\times {\L_f}$:
\begin{align}
\label{L++} ( {\L_f}\times {\L_f})_+=&\{(a,b)\in  {\L_f}\times {\L_f}\mid  \ell(a)+\ell(b)=a+b\}\\
\label{L+-} ( {\L_f}\times {\L_f})_-=&\{(a,b)\in  {\L_f}\times {\L_f}\mid a\neq b \text{ and }\ell(a)-\ell(b)=a-b\}.
\end{align}
\begin{example}\label{Ex39}
If $d=1$, then in view of   \eqref{ex_d0} $\ell(a)=-a$ and  the sets $( {\L_f}\times {\L_f})_\pm$ are empty.
If $d$ is any, but $\A$ is a one-point set $\A=\{b\}$, then $\L_f$ is the punched discrete sphere $\{a\in\Z^d\mid |a|=|b|, a\ne b\}$,
$\ell(a)= b$ for each $a$, and the sets $( {\L_f}\times {\L_f})_\pm$ again are empty. If $d\ge2$ and $|\A|\ge2$, then in general
the sets $( {\L_f}\times {\L_f})_\pm$ are non-trivial. See in Appendix~B.
\end{example}
  
   Obviously
 \be\label{obv}
 ( {\L_f}\times {\L_f})_+\cap ( {\L_f}\times {\L_f})_-=\emptyset\,.
 \ee
For further reference we note that
\begin{lemma}\label{L++-}
If $(a,b)\in ( {\L_f}\times {\L_f})_+\cup  ( {\L_f}\times {\L_f})_-$ then $|a|\neq|b|$.
\end{lemma}
\proof If $(a,b)\in ( {\L_f}\times {\L_f})_+$ and $|a|=|b|$ then $\ell(a)=\ell(b)$ and we have 
$$|a+b|=|2\ell(a)|=2|a|=|a|+|b|$$
which is impossible since $b$ is not proportional to $a$. 
 If $(a,b)\in ( {\L_f}\times {\L_f})_-$ and $|a|=|b|$ then $\ell(a)=\ell(b)$ and we get $a-b=0$ which is impossible in $ ( {\L_f}\times {\L_f})_-$. 
\endproof
According to the decomposition
$
\L=\L_f \cup \L_\infty,
$
the space $Y_\ga$, defined in \eqref{Y}, decomposes in the direct sum
\be\label{YY}
Y_\ga = Y_\ga^f\oplus Y_\ga^\infty,\quad  Y_\ga^f = \, \text{span}\, \{\zeta_s, s\in\L_f\}\,,\;\;
Y_\ga^\infty =\, \overline{\text{span} } \{\zeta_s, s\in\L_\infty\}\,.
\ee

\begin{lemma}\label{lem:adm}
For $m\notin\Cc$ the part $Z_4^{-2}$ of the Hamiltonian $Z_4$ equals 
\be\label{Z421}
\begin{split}
3{(2\pi)^{-d}} \Big(& \sum_{(a,b)\in  ( {\L_f}\times {\L_f})_+} \frac{ \xi_{\ell(a)}\xi_{\ell(b)}\eta_a\eta_b+ \eta_{\ell(a)}\eta_{\ell(b)}\xi_a\xi_b}{\la\lb}\\
+ 2&\sum_{(a,b)\in  ( {\L_f}\times {\L_f})_-}  \frac{ \xi_{a}\xi_{\ell(b)}\eta_{\ell(a)}\eta_b }{\la\lb}\Big)\,.
\end{split}
\ee
\end{lemma}
\proof
Let $\xi_i\xi_j\eta_k\eta_\ell $ be a monomial in $Z_4^{-2}$. We know that  $(i,j,-k,-\ell)\in\J$ and  $\{|i|,|j|\}=\{|k|,|\ell|\}$. 
 If $i,j\in\A$ or $k,\ell\in\A$ then we obtain the finitely many monomials as in the first sum in  \eqref{Z421}.
 Now we assume that 
 $i,\ell\in\A$ and $ j,k\in\L.$
Then we have that, 
{ either }   $|i|=|k|$ and 
$|j|=|\ell|$  which leads to finitely many monomials as in the second sum in  \eqref{Z421}.
{ Or} $i=\ell$ and $|j|=|k|$. 
In this last case, the zero momentum condition implies that $j=k$ which is not possible in $Z_4^{-}$.
\endproof

\subsection{Eliminating the non integrable terms}
For $\ell\in\A$ we introduce the variables $(I_a, \theta_a, \zeta^\L)$ as in \eqref{a-a}. 
Now the symplectic structure $- i d\xi\wedge d\eta$ reads
\be\label{nsympl}
-\sum_{a\in\A} dI_a\wedge d\theta_a  -i  d\xi^\L\wedge d\eta^\L\,.
\ee

In view of \eqref{Z4+}, \eqref{trans} and 
Lemma \ref{lem:adm},  for $m\notin\Cc$ 
the transformed Hamiltonian  may be written as (recall that $\omega=(\lambda_a, a\in\A)$) 
\begin{align*} 
(H_2+P)\circ\tau =&\om\cdot I +\sum_{s\in\L}\ls \xi_s\eta_s+
\frac 3 2(2\pi)^{-d} \sum_{\ell\in\A,\ k\in\Z^d} (4-3\delta_{\ell,k})\frac{I_\ell \xi_k\eta_k}{\lambda_\ell\lambda_k}\\
+&3(2\pi)^{-d} \Big(\sum_{(a,b)\in  ( {\L_f}\times {\L_f})_+} \frac{ \xi_{\ell(a)}\xi_{\ell(b)}\eta_a\eta_b+ \eta_{\ell(a)}\eta_{\ell(b)}\xi_a\xi_b}{\la\lb}\\
+&2 \sum_{(a,b)\in  ( {\L_f}\times {\L_f})_-}  \frac{ \xi_{a}\xi_{\ell(b)}\eta_{\ell(a)}\eta_b }{\la\lb}\Big)\\
+ &Q_4^3 +R^0_5\,,\qquad R^0_5=R_5\circ \tau+R^0_6\,.
\end{align*}
The first line contains the integrable terms. The second and third lines  contain  the lower-order non integrable terms,
depending on the angles $\theta$;  there are finitely many of them. 
The last line contains the remaining high order  terms, where $Q^3_4$ is of total order (at least) 4
and of order 3 in the {normal directions $\zeta$, while $R^0_5$ is of total order at least 5. 
 The latter is the sum of $R^0_6$ which comes from the Birkhoff normal form procedure
 (and is of order 6) and $R_5\circ \tau$ which comes from the term of order 5  in the nonlinearity \eqref{g}. Here $I$ is regarded 
as a variable of order 2, while $\theta$  has zero order. The fourth line  should  be regarded as a perturbation. 

To deal with the non integrable terms in the second and third lines,
following the works on the finite-dimensional reducibility (see \cite{el}),
we introduce a change of variables  
$$
\Psi: (\tilde I, \tilde\theta,    \tilde\xi,\tilde\eta)  \mapsto( I, \theta,  \xi,  \eta)\,,
$$
 symplectic with respect to 
\eqref{nsympl},  but such that its differential at the origin is 
not close to the identity. It is    defined by the following relations:
\begin{equation*}
\begin{split}
&  I_\ell=\tilde I_\ell-\sum_{\substack{|a|=|\ell| ,\  a\neq \ell}}{\tilde\xi}_a \tilde\eta_a,\quad   \tl=\tilde \tl\quad \ell\in\A\,;\\
&  \xi_a={\tilde\xi}_a e^{i \tilde \theta_{\ell(a)}},\quad  \eta_a=\tilde\eta_a e^{-i \tilde \theta_{\ell(a)}} \quad a\in {\L_f}\,;
\qquad \xi_a={\tilde\xi}_a , \quad  \eta_a=\tilde\eta_a  \quad a\in\L_\infty.
\end{split}
\end{equation*}

For any 
$(\tilde I, \tilde\theta, \tilde\zeta)\in \Tg(\nu,\sigma,\mu,\ga)$ denote by
$y=\{y_l, l\in\A\}$ the vector, whose $l$-th component equals 
$y_l=\sum_{|a|=|l|\,, a\ne l}\tilde\xi_a\tilde\eta_a$. Then
$$
|I-  \tfrac12 \nu\rho^2|\le |\tilde I-\tfrac12 \nu\rho^2 |+|y|\le c_*^2 \nu\mu^2 +\sum_{a\in\L_f}|\tilde\xi_a\tilde\eta_a|\le 
2c_*^2\nu\mu^2\,.
$$
This implies that 
\be\label{prop11}
\Psi^{\pm1}(  \Tg\big(\nu, \frac12, \cc ,\ga\big)  ) \subset \Tg\big(\nu,\frac12, \frac12,\ga\big)\,.
\ee
We abbreviate $
 \Tg\big(\nu, \frac12, \cc ,\ga\big)=:  \Tg 
$.\\ 
We  note   that $\Psi(T_I^n)=T_I^n$ and although $\Psi$ is not close to the identity in general, it is close to the identity in variable $(I,\theta)$ in a neighbourhood of the $T_I^n$. Namely, denoting $\Psi(\tilde I, \tilde\theta,    \tilde\zeta^\L)=(I,\theta,\zeta^\L)$, we have
 \be\label{TheRem}
|\tilde I_a-I_a|\leq \|(\tilde\zeta^\L)\|^2\,,\ a\in\A ,\ \theta=\tilde \theta \text{ and } \|\zeta^\L\|_\ga=\|\tilde\zeta^\L\|_\ga\,.\ee


  On the other hand, if $(\tilde\xi, \tilde\eta)\in\Tg$, then for $l\in\A$ 
$$
\xi_l=\sqrt{I_l}\,e^{i\theta_l} = \sqrt{\tilde I_l}\,e^{i \tilde \theta_l} +O (\nu^{-1/2} )\,O(|\zeta^\L|^2).
$$
Therefore, dropping the tildes, we write the restriction to $\Tg$ of the 
 transformed Hamiltonian as
\begin{align*} 
H_1:=&
H\circ\tau\circ \Psi =\om\cdot I +\sum_{a\in\L_\infty}\la {\xi}_a\eta_a\\
&+6(2\pi)^{-d} \sum_{\ell\in\A,\ k\in\L} \frac{1}{\lambda_\ell\lambda_k}(I_\ell-\sum_{\substack{|a|=|\ell| \\ a\in {\L_f}}}\xi_a\eta_a) \xi_k\eta_k\\
&+\frac 3 2(2\pi)^{-d} \sum_{\ell,k\in\A} \frac{4-3\delta_{\ell,k}}{\lambda_\ell\lambda_k}(I_\ell-\sum_{\substack{|a|=|\ell| \\ a\in {\L_f}}}\xi_a\eta_a) (I_k-\sum_{\substack{|a|=|k| \\ a\in {\L_f}}}\xi_a\eta_a)\\
&+3(2\pi)^{-d}\sum_{(a,b)\in  ( {\L_f}\times {\L_f})_+} \frac{ \sqrt{I_{\ell(a)}I_{\ell(b)}}}{\la\lb}(\eta_a\eta_b+
 \xi_a\xi_b)\\
&+6(2\pi)^{-d}\sum_{(a,b)\in  ( {\L_f}\times {\L_f})_-} \frac{ \sqrt{I_{\ell(a)}I_{\ell(b)}}}{\la\lb}\xi_a\eta_b+
Q_4^{3'}  +R^{0'}_5 +\nu^{-1/2}R^{4'}_5
\,.
\end{align*}
Here $Q^{3'}_4$ and $R^{0'}_5$ are the function $Q^{3}_4$ and $R^{0}_5$, transformed by
$\Psi$ (so the former  satisfy the same estimates as the latter), while $R^{4'}_5$ is a function 
of forth order in the normal variables. 
Or, after a  simplification:
\begin{align} \begin{split} \label{H-fin}
H_1= &\om\cdot I +\sum_{a\in\L_\infty}\la \xi_a\eta_a
+\frac 3 2(2\pi)^{-d} \sum_{\ell,k\in\A} \frac{4-3\delta_{\ell,k}}{\lambda_\ell\lambda_k}I_\ell I_k\\
&+3(2\pi)^{-d} \Big( 2\sum_{\ell\in\A,\ a\in\L_\infty} \frac{1}{\lambda_\ell\lambda_a}I_\ell \xi_a\eta_a-  \sum_{\ell\in\A,\ a\in {\L_f}} \frac{(2-3\delta_{\ell,|a|})}{\lambda_\ell\lambda_a}I_\ell \xi_a\eta_a\Big)Ã±Â \\
&+3(2\pi)^{-d} \sum_{(a,b)\in  ( {\L_f}\times {\L_f})_+} \frac{ \sqrt{I_{\ell(a)}I_{\ell(b)}}}{\la\lb}(\eta_a\eta_b+ \xi_a\xi_b)\\ 
&+6(2\pi)^{-d}\sum_{(a,b)\in  ( {\L_f}\times {\L_f})_-} \frac{ \sqrt{I_{\ell(a)}I_{\ell(b)}}}{\la\lb}\xi_a\eta_b
+ Q_4^{3'}  +R^{0'}_5 +\nu^{-1/2}R^{4'}_5
\,.
\end{split}\end{align}

We see that the transformation $\Psi$ removed from $H\circ\tau$ the non-integrable lower-order terms on the price 
of introducing ``half-integrable" terms which do not depend on the angles $\theta$, but depend on the actions $I$ and 
quadratically depend   on 
finitely many variables $\xi_a,\eta_a$ with $a\in {\L_f}$.

The Hamiltonian $H\circ\tau\circ \Psi$ should be regarded as a function of the variables 
$(I,\theta,\zeta^\L )$. Abusing notation, below we drop the upper-index $\L$ and  write 
$\zeta^\L={}^t(\xi^\L,\eta^\L)$ as $\zeta={}^t( \xi,\eta)$.

\subsection{Rescaling the variables and defining  the transformation
 $\Phi$}
 Our aim is to study the Hamiltonian $H_1$ on the domains
  $\Tg=\Tg(\nu, \frac12, \cc ,\ga) $, $0\le\ga\le\ga_*$
 (see \eqref{prop11}). 
 To do this we re-parametrise points of $\Tg$ by mean of the change of variables
 $(I,\theta,\xi,\eta)=\chi_\yy(\tilde r,\tilde\theta,\tilde\xi, \tilde\eta)$, where 
 \begin{align*}
I=\nu\rho+\nu \tilde r,\quad \theta=\tilde\theta,\quad
\xi=\sqrt\nu\, \tilde \xi,\quad \eta=\sqrt\nu\, \tilde\eta\,.
\end{align*}
Clearly,
$$
\chi_\yy: \O^{\ga}(\frac 12, {\frac{c_*}{2\sqrt 2}})   \to  \Tg\,,
$$
 and in the new variables the symplectic structure reads 
$$
-\nu\sum_{\ell\in\A}\tilde dr_\ell\wedge d\tilde\tl \ -i\ \nu\sum_{a\in\L}d\tilde\xi_a\wedge d\tilde\eta_a.
$$
Denoting
$$
\Phi =\Phi_\yy=\tau\circ\Psi\circ\chi_\yy,
$$ 
we see that this transformation is real holomorphic in $\yy\in\Da$ for a suitable $c_1>0$. It
satisfies all assertions of the item (i) of Proposition~\ref{thm-HNF}.\\
We also notice for later use that,using using \eqref{TheRem} and \eqref{esti1}, for ${\frak z}=(r,\theta,z^\L)\in\O^{\ga}(\frac 12, {\frac{c_*}{2\sqrt 2}})$, $\zeta=\Phi_\rho(\frak z)=(\zeta^\A,\zeta^\L)$ satisfies or $\nu$ small enough
\be\label{theRem2}
\|\zeta^\L\|_\ga\leq\nu^{1/2}\|z^\L\|_\ga(1+C\norma{\nu^{1/2}{\frak z}}_\ga^2)\leq 2\nu^{1/2}\|z^\L\|_\ga\,,
\ee
and  
$$
\norma{\zeta^\A-\nu^{1/2}\sqrt{\rho+r}e^{i\theta}}\leq( \sqrt n \nu^{1/2}\|z^\L\|_0)(1+C\norma{\nu^{1/2}{\frak z}}_0^2) \leq 2\sqrt n \nu^{1/2}\|z^\L\|_0
$$
 thus
\be\label{theRem}
\norma{\zeta^\A-\nu^{1/2}\sqrt{\rho}e^{i\theta}}\leq( 2\sqrt n \nu^{1/2}\|z^\L\|_0+\nu^{1/2}\frac{| r|}{2c_*}) \leq \frac2{c_*}\sqrt n \nu^{1/2}(\|z^\L\|_0+|r|)\,.
\ee

We have, dropping the tilde and forgetting the irrelevant   constant $\nu( \om\cdot \rho)$ ,
\begin{align} \begin{split} \label{H-rescall}
H  \circ \Phi &=\nu\Big[ \om\cdot r + \sum_{a\in\L_\infty}\la  \xi_a\eta_a
+(2\pi)^{-d}\nu\,
\Big(\,
\frac 3 2
  \sum_{\ell,k\in\A} \frac{4-3\delta_{\ell,k}}{\lambda_\ell\lambda_k}\rho_\ell r_k\\
+6   &\sum_{\ell\in\A,\ a\in\L_\infty} \frac{1}{\lambda_\ell\lambda_a}\rho_\ell \xi_a\eta_a-3  \sum_{\ell\in\A,\ a\in {\L_f}} \frac{(2-3\delta_{\ell,|a|})}{\lambda_\ell\lambda_a}\rho_\ell \xi_a\eta_a\\
+3  &\sum_{(a,b)\in  ( {\L_f}\times {\L_f})_+} \frac{  {\sqrt{\rho_{\ell(a)}}\sqrt{\rho_{\ell(b)}}}}{\la\lb}(\eta_a\eta_b+ \xi_a\xi_b)\\
+6 & \sum_{(a,b)\in  ( {\L_f}\times {\L_f})_-} \frac{  {\sqrt{\rho_{\ell(a)}}\sqrt{\rho_{\ell(b)}}}}{\la\lb}\xi_a\eta_b\Big) \Big]  \\
&+\Big(\big(Q_4^{3'}  +R^{0'}_5+\nu^{-1/2}R^{4'}_5\big)(I,\theta,\sqrt\nu\zeta)\Big)\mid_{I=\nu\rho+\nu r}\,.
\end{split}\end{align}
So, 
\be\label{f} {\nu}^{-1} H\circ \Phi = h+ f\,,
\ee
 where $h\equiv h( I, \xi,\eta;\yy,\nu)$ is the quadratic part of the Hamiltonian,
 independent from  the angle $\theta$, and $f$ is the perturbation, given by the last line in \eqref{H-rescall}:
 \be\label{ff}
 f=\nu^{-1}\Big(\big(Q_4^{3'}  +R^{0'}_5 +\nu^{-1/2}R^{4'}_5\big)(I,\theta,\nu^{1/2}\zeta)\Big)\mid_{I=\nu\rho+\nu r}\,.
 \ee
 
 We have
\be\label{h}
h=\Omega\cdot r + \sum_{a\in\L_\infty}\Lambda_a \xi_a\eta_a+\nu \langle K(\yy)\zeta_f,\zeta_f\rangle
\ee
where $\Omega=(\Omega_k)_{k\in\A}$ and
\begin{align}\label{Om}\Omega_k=\Omega_k(\yy,\nu)&=\om_k+ \nu\sum_{\ell\in\A} M^\ell_k \rho_l,\quad 
M^\ell_k=\frac{3(4-3\delta_{\ell,k})}{(2\pi)^d\lk\lambda_\ell} \,,
\\
\label{Lam}\Lambda_a=\Lambda_a(\yy,\nu)&= \la+6\nu(2\pi)^{-d}\sum_{\ell\in\A}\frac{\rho_\ell}{\lambda_\ell \la}\,.
\end{align}
Besides, 
$$
\zeta=(\zeta_a)_{a\in\L},\quad \zeta_a=\left(\begin{array}{l}
\xi_a\\
\eta_a
\end{array}\right),\quad \zeta_f=(\zeta_a)_{a\in {\L_f}}\,,
$$
and  $K(\yy) $ is a symmetric complex matrix,  acting in space 
\be\label{Yf}
Y^f_\ga=\{\zeta_f\}\simeq \C^{2|\L_f|}\,,
\ee
such that the corresponding quadratic form is 
\begin{align} \begin{split} \label{K}
\langle K(\yy)\zeta_f,\zeta_f\rangle=\,&3(2\pi)^{-d} 
\Big(\sum_{\ell\in\A,\ a\in {\L_f}} \frac{(
3\delta_{\ell,|a|}-2
)}{\lambda_\ell\lambda_a}
\rho_\ell \xi_a\eta_a\\
+ \sum_{(a,b)\in  ( {\L_f}\times {\L_f})_+}& \frac{ {\sqrt{\rho_{\ell(a)}}\sqrt{\rho_{\ell(b)}}}}{\la\lb}(\eta_a\eta_b
+ \xi_a\xi_b)+ \\
2 \sum_{(a,b)\in  ( {\L_f}\times {\L_f})_-} &\frac{  {\sqrt{\rho_{\ell(a)}}\sqrt{\rho_{\ell(b)}}}}{\la\lb}\xi_a\eta_b \Big).
\end{split}\end{align}
Note that the matrix $M$ in \eqref{Om} is invertible since 
$$
\det M={3^n}{(2\pi)^{-dn}}\big(\Pi_{k\in\A}\lk\big)^{-2}\det \left(4-3\delta_{\ell,k}\right)_{\ell,k\in\A}\ne0\,.
$$
Relation \eqref{-10}  follows from the explicit formulas \eqref{Om}-\eqref{K}, so the items (i) and (ii)
of Proposition~\ref{thm-HNF} are proven.

 It is  clear that  the matrix $K(\yy)$ is analytic in $\yy\in\Da$ 
(see definition in \eqref{Dset}) and satisfies \eqref{est11}. This proves (iii). 

\smallskip

  It remains to verify (iv). By Proposition \ref{Thm-BNF}  the function $f$ belongs to the class 
  ${\Tc}^{{\ga},D}({\frac 12}, \frac{c_*}{2\sqrt2}, \D)$. 
 Since the reminding term $f$ has the form \eqref{ff} then
for $(r,\theta,\zeta)\in\O^{\ga}({\frac 12}, \frac{c_*}{2\sqrt2})$ it satisfies the estimates
\begin{equation*}
\begin{split}
 |f|\le C  \nu  \,,\quad
\| \nabla_\zeta f\|_{\ga} \le C  \nu  \,,\quad
\|\nabla^2_\zeta f\|_{\ga}^D \le C  \nu \,.
\end{split}
\end{equation*}
Now consider the $f^T$-component of $f$. Only the second term in \eqref{ff} contributes to it
and we have that 
$$
|f^T| + \| \nabla_\zeta f^T\|_\ga +\|\nabla_\zeta^2 f^T\|^D_{\ga}\le C\nu^{3/2}\,.
$$

Recall that the function $f$ depends on the parameter $\yy$ through the 
substitution $I=\nu\rho+\nu r$. So $f$ is analytic in $\yy$ and holomorphically  extends to a complex 
neighbourhood of $\D$ of order one, where it satisfies the estimates above with a modified 
constant $C$. Therefore by the Cauchy estimate the gradient of $f$ in $\yy$ satisfies 
in the smaller complex neighbourhood 
 $\Da$ the same estimates as above, again with a modified constant. This implies the 
assertion (iv) of the theorem.
}
\endproof
\medskip

We will provide the domain $ \O^{\ga} \Big({\frac 12}, {\frac{c_*}{2\sqrt2}}\Big)^\L $ with the
coordinates $(r,\theta,\xi,\eta)$ with the 
symplectic structure $- \sum_{\ell\in\A}dr_\ell\wedge d\tl \ -i \sum_{a\in\L}d\xi_a\wedge d\eta_a$. Then
the transformed hamiltonian system, constructed in Proposition~\ref{thm-HNF} has the 
Hamiltonian, given by the r.h.s. of \eqref{HNF}.


\section{The final normalisation.} \label{s_4}

The normal form, provided by Proposition \ref{thm-HNF}, has two disadvantages: it is written in the complex 
variables  with the non-standard reality condition \eqref{reality}, while in the original equation \eqref{beam'} the 
reality condition is standard (namely, $u(t,x)$ and $v(t,x)$ are real functions), and -- which is much more
important --  the hamiltonian operators $iJK(\yy)$, corresponding to different $\yy$, do 
not commute. In this section we pass in the  normal form to the variables with the  usual reality condition,
 construct a $\rho$-dependent transformation which diagonalises the hamiltonian  operator, and examine the 
 smoothness of this  transformation as a function of $\yy$. So here we are concerned with 
 analysis of the finite-dimensional linear hamiltonian system, corresponding to 
  the Hamiltonian \eqref{K}. In difference with
 the previous sections, now the parameter $\yy$ will belong to  subdomains $Q\subset \D$, which are closed
 semi-analytic  sets, defined by single polynomial relation, 
   $Q= \{ \yy\mid P(\sqrt\yy)\ge \delta\}$. We recall that smoothness of functions on such domains is understood 
   in the sense of Whitney.
 We keep assuming that the set $\A$ is admissible. We recall that we provide the phase-space with the 
 symplectic structure $- \sum dr_\ell\wedge d\tl \ -i \sum d\xi_a\wedge d\eta_a$.

\subsection{Matrix $K(\yy)$} \label{s_3.6}
Recalling \eqref{labA} and \eqref{DDD}, we write the 
 symmetric  matrix $K(\yy)$, defined by  relation \eqref{K}, as  a block-matrix, polynomial in  
$
\sqrt\rho = (\sqrt\rho_1, \dots, \sqrt\rho_n)  \,.
$
We write it as  $ K(\yy) =  K^d(\yy) + \tkn(\yy)$. 
Here $\tkd$ is the block-diagonal matrix 
\be\label{.1}
\begin{split}
\tkd (\yy) &=\text{diag}\,\Big( \left(\begin{array}{ll}
0& \mu(a,\rho)\\ \mu(a,\rho)& 0 \end{array}\right), \ a\in \L_f\Big), \\
\mu(a,\rho) &= C_*  \big(\frac32\,\rho_{\ell(a)}\la^{-2}- \la^{-1} \sum_{l\in\A} \rho_l\lambda_l^{-1}\big)\,,\quad C_*=3(2\pi)^{-d}. 
\end{split}
\ee
Note that\footnote{Here and in similar situations below we do not mention the obvious dependence 
on the parameter $m\in[1,2]$.}
\be\label{munew}
\mu(a,\rho)\quad \text{is a function of $|a|$ and $\rho$\,.
}
\ee

The non-diagonal 
matrix $\tkn$ has zero diagonal blocks, while for $a\ne b$ its block $\tkn (\yy)_a^b $ equals 
$$
 C_*\frac{\sqrt{ {\yy_{l(a)} \yy_{l(b)}}}}{\la \lb}  \, \left(
 \left(\begin{array}{ll}
1& 0\\
0& 1
\end{array}\right) \chi^+(a,b)+
 \left(\begin{array}{ll}
0& 1\\
1& 0
\end{array}\right)  \chi^-(a,b)
\right)\,,
$$
where
\ben
\chi^+(a,b)=
 \left\{\begin{array}{ll}
1, \;\; (a,b)\in (\L_f\times\L_f)_+,
\\
0,  \;\; \text {otherwise},
\end{array}\right.
\een
and $\chi^-$ is defined similar in terms of the set $ (\L_f\times\L_f)_-$. In view of \eqref{obv},
$$
\chi^+(a,b)\cdot \chi^-(a,b)\equiv0. 
$$

Accordingly, the hamiltonian matrix 
 $\H(\yy) =  iJK(\yy)$ equals
$ \big( \H^d(\yy) + \H^{n/d}(\yy)\big)$,  where
\begin{align}\begin{split}\label{diag}
\H^d(\yy) &= i\,\text{diag}\, \left( 
 \left(\begin{array}{ll}
\mu(a,\rho)& 0\\
0& -\mu(a,\rho)
\end{array}\right) ,\; a\in\L_f
\right),\\
 \H^{n/d}(\yy) _a^b &= iC_* \frac{{\sqrt{\rho_{\ell(a)}} \sqrt{\rho_{\ell(b)}}}}{\la \lb}\Big[
 J
 \chi^+{(a,b)}
+\left(\begin{array}{ll}
1& 0\\
0& -1
\end{array}
\right) \chi^-{(a,b)}
 \Big] \,.
\end{split}\end{align}
Note that   all elements of  the matrix  $\H(\yy)$ are pure imaginary, and 
\be\label{Lf+=0}
\text{ if 
 $( {\L_f}\times {\L_f})_+=\emptyset$, then  $-i\H(\yy)$ is real symmetric},
 \ee
 in which case   all 
eigenvalues of  $\H(\yy)$  are pure imaginary. In Appendix~B we show that if $d\ge2$,
then, in general, the set 
$( {\L_f}\times {\L_f})_+ $ is not empty and  the matrix $\H(\yy)$ may have hyperbolic eigenvalues. 
\smallskip

\begin{example}\label{Ex41}
In view of Example \ref{Ex39}, if $d=1$ then the operator $\H^{n/d}$ vanishes. We see immediately
that in this case $\H^{d}$ is a diagonal operator with simple spectrum. 
\end{example}

Let us introduce in $\L_f$ the relation $\sim$, where 
\be\label{class}
a\sim b\;\; \text{if and only if }\;\; a=b\;\;\text{or}\;\; (a,b)\in (\L_f\times \L_f)_+\cup  (\L_f\times \L_f)_-\,.
\ee
It is easy to see that this is an equivalence relation. By Lemma~\ref{L++-}
\be\label{.2}
a\sim b,\ a\ne b \;\Rightarrow \; |a| \ne |b|\,\,.
\ee
The equivalence $\sim\,$, as well as the sets $ (\L_f\times \L_f)_\pm$, depends only on the lattice $\Z^d$
and the set $\A$, not on the eigenvalues $\la$ and the vector $\rho$. It
 is trivial if $d=1$ or  $|\A|=1$ (see Example~\ref{Ex39})) and, in general,  is non-trivial otherwise. If  $d\ge2$ and $|\A|\ge2$
 it is rather complicated.

The equivalence relation divides   $\L_f$ into equivalence classes,
$
\L_f = \L_f^1\cup\dots \cup \L_f^M\,.
$
The set $\L_f$ is a union of the punched spheres $\Sigma_a=\{b\in\Z^d\mid |b|=|a|, b\ne a\}$, $a\in\A$, 
and by \eqref{.2}  each equivalence class $\L_f^j$ intersects every punched sphere $\Sigma_a$ 
 at at most one point.

Let us order the sets $\L_f^j$ in such a way that for a suitable $0\le M_*\le M$ we have

-- $\L_f^j = \{b_j\}$ (for a suitable point $b_j\in\Z^d)$ if $j\le M_*$;

-- $|\L_j^f|=n_j\ge2$ if $j>M_*$.\\

Accordingly the complex  space $Y^f = Y^f_0$ (see \eqref{YY}) decomposes as
\be\label{deco}
Y^f = Y^{f1}\oplus\dots \oplus Y^{fM},\quad Y^{fj}= \,\text{span}\, \{\zeta_s, 
s\in\L_f^j\}\,.
\ee
Since each $\zeta_s, s\in\L_f$, is a 2-vector, then 
$$
\dim Y^{f j}=2|\L_j^j|:= 2n_j\,,\qquad \dim Y^f =2|\L_f| = 2\sum_{j=1}^M n_j
:= 2\bold N\,.
$$
So dim$\,Y^{f j}=2$ for $j\le M_*$ and  dim$\,Y^{f j}\ge4$ for $j> M_*$.
In view of \eqref{.2},
\be\label{.0}
|\L^j_f  | = n_j \le |\A|\qquad \forall\, j\,.
\ee

We readily see from the formula for the matrix $\H(\yy)=iJ K(\yy)$ that the spaces $ Y^{fj}$ are 
invariant for the operator  $\H(\yy)$. So
\be\label{decomp}
\H(\yy)= 
\H^1(\yy)\oplus\dots \oplus\H^M(\yy)
\,,\qquad 
\H^j = \H^{j\,d} + \H^{j\,n/d},
\ee
where $\H^j$ operates in the space $Y^{fj}$ and $ \H^{j\,d}$ and $ \H^{j\,n/d}$ are given by the formulas \eqref{diag}
with  $a,b\in\L_f^j$. The hamiltonian operator $\H^j(\yy)$ polynomially depends on $\sqrt\yy$. So
its eigenvalues form an algebraic function of $\sqrt\yy$ (see \eqref{Dset}). Since the spectrum of $\H^j(\yy)$ is an even set,
then we can write this algebraic function as 
$\{ \pm i\Lambda_1^j(\yy),  \dots,  \pm i \Lambda_{n_j}^j(\yy)\}$ (the factor $i$ is convenient for further purposes). 
The eigenvalues of $\H(\rho)$ are given by another algebraic function which we write as 
$\{\pm i\Am_m(\yy), 1\le m\le \bold N=|\L_f|\}$. Accordingly, 
\be\label{spectrum}
\{\pm\Am_1(\yy),\dots,\pm \Am_{\bold N}(\yy)
 \} = \cup_{j\le M} \{\pm\Am_k^j(\yy),  k\le n_j\}\,,
 \ee
 and $\Am_j=\Am^j_1$ for $j\le M_*$. 
 
 The functions $\Am_k$ and $\Am^j_k$ are defined up to multiplication by $\pm1$.\footnote{More precisely, if $\Am_k$
 is not real, then well defined is the quadruple $\{\pm\Am_k, \pm\bar\Am_k\}$; see below.}
  But if $j\le M_*$, then 
  $\L^j_f = \{b_j\}$  and  $\H^j = \H^{jd}$,  so the spectrum of this operator is $\{\pm i \mu(b_j,\rho)\}$, 
 where $\mu(b_j,\rho)$ is a well defined analytic function of $\rho$, given by the explicit formula \eqref{.1}. 
  In this case we specify the choice of
 $\Am^j_1$:
 \be\label{normaliz}
 \text{ if  $\ \L^j_f=\{b_j\}$,  we choose $\Am^j_1(\rho) = \mu(b_j,\rho)$. }
 \ee
 So for $j\le M_*$, 
 $ \Am_j(\yy) = \mu(b_j,\rho)$  is a polynomial of $\sqrt\yy$, which depends only on $|b_j|$ and $\yy$.

 Since the norm of the operator $K(\yy)$ satisfies \eqref{esti1}, then 
 \be
\label{N1} 
 | \Lambda^j_r(\yy)|\le C_2 \quad \forall\,\yy, \ \forall\,r\,,\; \forall\,j\,.
\ee

 \begin{example}\label{n=1}
 In view of \eqref{.0}, 
 if  $\A$ is a one-point set, 
 $\A=\{a_*\}$,  then all sets $|\L^j_f|$ are one-point. So $M_*=M=\bold N$ and 
 $$
 \{\pm\Am_1(\yy),\dots,\pm \Am_{\bold N}(\yy)\} = 
 \{ \pm\mu(a,\rho)\mid a\in\Z^d, |a| = |a_*| , a\ne a_*\}.
 $$
 In  this case the spectrum of the hamiltonian operator $\H(\rho)$ is pure imaginary, multiple and analytically 
 depends on $\rho$. 
 \end{example}

   Let $1\le j_*\le n$ and 
 $\D^{j_*}_0$ be the set 
\be\label{DD}
  \D^{j_*}_0 = \{\rho=(\rho_1,\dots, \rho_n)\mid c_* \le \rho_l\le c_{**}  \;\;\text{if}\;\;
  l\ne j_* \;\;\text{and}\;\;  1-c_{**}  \le \rho_{j_*}\le 1\}\,,
  \ee
where $0<c_*\le\tfrac12 
c_{**}<1/4$. This is a subset of $\D=[c_*,1]^n$ which  lies in the (Const$\,c_{**}$)-vicinity of the point 
 $\yy_*=(0,\dots,1,\dots,0)$ in $[0,1]^n$, where 1 stands on the $j_*$-th place. Since 
 $\tkn(\yy_*)=0$, then $K(\yy_*) = K^d(\yy^*)$. 
 Consider any equivalence class $\L_f^j$ and  enumerate its 
elements   as $b^j_1,\dots,b^j_{n_j}$ $(n_j\le n)$. 
 For $\yy=\yy_*$ the  matrix $\H^j(\yy_*)$ is diagonal with the eigenvalues 
 $\pm i\mu(b^j_r,\yy_*), 1\le r\le n_j$. It suggests that for $c_{**}$ sufficiently small we may uniquely 
  numerate the 
 eigenvalues $\{\pm i \Lambda^j_r(\yy)\} \ (\yy\in \D^{j_*}_0)$ of the matrix $\H^j(\yy)$ in such a way that 
 $\Lambda^j_r(\yy)$ is close to  $\mu(b^j_r,\yy_*)$. Below we justify this possibility.

 Take any $b\in\L_f$ and denote $\ell(b)=a_b\in\A$. 
If $a_b = a_{j_*}$, then 
\be\label{mu1}
\mu(b, \yy_*) = C_*(\frac32  \lambda_{a_{j_*}}^{-2}-\lambda_{a_{j_*}}^{-2})= \tfrac12C_*  \lambda_{a_{j_*}}^{-2}\,.
\ee
If $a_b \ne a_{j_*}$, then
\be\label{mu2}
\mu(b, \yy_*) = -C_* \lambda_{a({b})}^{-1} \lambda_{a_{j_*}}^{-1}.  %
\ee
If $m\in[1,2]$ is different from $4/3$ and $ 5/3$, then   it is easy  to see that $2\la\ne \pm \lambda_{a'}$ 
 for any $a, a'\in\A$. By \eqref{modif} this 
  implies that for $m\in[1,2]\setminus\Cc$ and for $ b, b'\in \L_f$ such that 
 $|b|\ne |b'|$ we have 
$$
|\mu(b, \yy_*)  | \ge 2c^{\#}(m)>0\,,\quad
|\mu(b, \yy_*) \pm \mu(b', \yy_*) | \ge 2c^{\#}(m)\,,\quad
$$
and
\be\label{N110}
|\mu(b,\yy)|\ge c^{\#}(m)>0\,,\qquad |\mu(b,\yy) \pm \mu(b',\yy)  |\ge c^{\#}(m)\;\;\;\text{for}\;\;
\yy\in \D_0^{j_*}\,,
\ee
if  $c_{**}$ is small. In particular, for each $j$
the spectrum $\pm i\mu(b^j_r,\yy_*), 1\le r\le n_j$ of the matrix $\H^j(\yy_*)$
is simple.

\begin{lemma}\label{laK} 
If $c_{**} \in(0,1/2)$ is sufficiently small,\footnote{Its smallness only depends on  $\A, m$ and $g(\cdot)$.}
 then there exists $c^o=c^o(m)>0$ 
such that for each $r$ and $j$, 
 $\Lambda^j_r(\yy)$ is a real analytic function of $\yy\in  \D^{j_*}_0$, satisfying 
\begin{equation}\label{N111}
\begin{split}
|\Lambda^j_r(\rho)  - \mu(b^j_r,\yy)|
 \le C\sqrt{c_{**}}\qquad \forall\, \rho\in D^{j_*}_0\,,
\end{split}
\end{equation}
and
\begin{equation}\label{N11}
|\Lambda^j_r(\rho)|\ge c^o(m)>0
\;\;\text{and}\;\; |\Lambda^j_r(\rho)\pm\Lambda^j_l(\rho)
|\ge c^o(m)\;\; \forall \, r\ne l, \forall\, j, \forall\, \rho\in D^{j_*}_0\,,
\end{equation}
\be\label{aaa}
 |  \Lambda^{j_1}_{r_1}(\yy) + \Lambda^{j_2}_{r_2}(\yy)  | \ge  c^{0}(m)\quad 
\forall \, j_1, j_2, r_1, r_2 \quad \text{and}\;\;
\rho\in D^{j_*}_0\,.
\ee
In particular, 
\begin{equation}
\label{N3}
\Lambda^j_r \not\equiv 0\quad \forall r;\quad \Lambda^j_r \not\equiv \pm\Lambda^j_l \quad
\forall\, r\ne l\,.
\end{equation}
\end{lemma}

The estimate \eqref{N111} assumes that for $\yy\in\D_0^{j_*}$ we fix the sign of the function $\Lambda_r^j$ by
the following agreement:
\be\label{agreem}
\Lambda_j^r(\yy) \in\R\;\;\text{and \ \ sign}\, \Lambda_r^j(\yy) =  \ \text{sign}\,\mu(b_r^j,\yy)\;\;
\forall \yy\in \D_0^{j_*}\,,\; \forall 1\le j_*\le n\,,\; \forall\, r, j\,,
\ee
see \eqref{mu1}, \eqref{mu2}. 

Below we fix any $c_{**} =c_{**}(\A, m, g(\cdot))\in(0,1/2)$ such that the lemma's assertion holds, but the parameter 
$c_*\in(0, \tfrac12 c_{**}]$ will vary during the argument.

\begin{proof}
Since the 
 spectrum of $\H^j(\rho_*)$ is 
simple and the matrix $\H^j(\rho)$ and the numbers  $\mu(b^j_r,\yy)$ are polynomials of $\sqrt\yy$, then the basic perturbation theory 
implies that the functions $\Lambda^j_r(\yy)$ are real analytic in  $\sqrt\yy$ in the vicinity of
$\rho_*$ and we have
$$
|\mu(b^j_r,\yy_*) - \mu(b^j_r,\yy)|\le C\sqrt{c_{**}}\,,\quad
|\Lambda^j_r(\yy_*) - \Lambda^j_r(\yy)| \le C\sqrt{c_{**}}\,,
$$
so \eqref{N111} holds. It is also clear that the functions $\Lambda^j_r(\yy)$ are analytic in 
$\yy\in\D_0^{j_*}$. Relations   \eqref{N111} and \eqref{N110} (and the fact that $\mu(b,\yy)$ depends only on
$|b|$ and $\yy$) imply \eqref{N11} and \eqref{aaa}
 if $c_{**}>0$ is sufficiently small. 
\end{proof}

\begin{remark}\label{r_m}
The differences $|2\la - \lb|$ can be estimated from below uniformly in $a,b$ in terms of the distance from $m\in[1,2]$
to the points $4/3$ and $5/3$. So the  constants $c^{\#}$ and $c^o$ depend only on this distance, and they  can be chosen independent from $m$ if the latter   belongs to the smaller segment $[1, 5/4]$. 
\end{remark}

Contrary to \eqref{aaa}, in general a difference of two eigenvalues $\Lambda^{j_1}_{r_1}- \Lambda^{j_2}_{r_2}$ may vanish identically.
Indeed, if $j,k\le M_*$, then   $\L^k_f$ and $\L^j_f$ are 
one-point sets, $\L^k_f=\{b_k\}$ and $\L^j_f=\{b_j\}$, and $\Am^j_1=\mu(b_j,\cdot)$, $\Am^k_1=\mu(b_k,\cdot)$. 
 So if $|b_j|=|b_k|$, then $\Lambda^j_1 \equiv  \Lambda^k_1$ due to  \eqref{munew}. In particular, in 
  view of Example~\ref{n=1}, if $n=1$ then each $\L^j_f$ is a  one-point set, corresponding to some point
$b_j$ of the same length. In this case all functions $\Lambda_k(\rho)$ 
coincide identically. But if $j\le M_* <k$, or if $\max{j,k}>M_*$ and the set $\A$ is\sa\ (recall that everywhere in this section 
it is assumed to be admissible), then $\Lambda^{j_1}_{r_1}- \Lambda^{j_2}_{r_2}\not\equiv0$.
This is the assertion of the  non-degeneracy lemma below, proved  in Section~\ref{s_2d}. 

\begin{lemma}\label{l_nond}
Consider any two spaces $Y^{f\,r_1}$ and $Y^{f\,r_2}$ such that $r_1\le r_2$ and 
$r_2>M_*$. Then 
\be\label{single}
 \Lambda_j^{r_1} \not\equiv \pm\Lambda_k^{r_2} \qquad \forall\, (r_1,j)\ne (r_2,k)\,,
\ee
provided  that either $r_1\le M_*$, or the set $\A$ is strongly admissible.
\end{lemma}

We recall that for $d\le2$ all admissible sets are\sa. For $d\ge3$ non-\!\!\sa\ sets exist. 
In Appendix~B we give an example \eqref{AAA} of such a set  for $d=3$ 
and show that for it the relation \eqref{single} does not hold.

\subsection{Real variables}  Let us pass in \eqref{HNF} from the complex variables 
$\zeta^\L=(\xi^\L,\eta^\L)$  to the real variables $\tilde\zeta^\L=(u^\L,v^\L)$, where 
\be\label{reall}
\xi_l=\frac1{\sqrt2}\,(u_l+i v_l),\quad \eta_l=\frac1{\sqrt2}\,(u_l-i v_l),\qquad l\in\L\,,
\ee
and denote by $\Sigma$ the mapping 
\be\label{transform}
\Sigma: (r,\theta, u^\L,v^\L)\mapsto (r,\theta, \zeta^\L)\,. 
\ee
Below we write $u,v,\tilde\zeta$ instead of $u^\L,v^\L,\tilde\zeta^\L$, and write $\tilde\zeta=\tilde\zeta^\L$ as 
$\tilde\zeta=(\tilde\zeta_f,\tilde\zeta_\infty)$, where $\tilde\zeta_f$ is formed by the components $(u_l, v_l)$ of
$\tilde\zeta$, belonging to the set $\L_f$, and similar with $\tilde\zeta_\infty$. 

The new variables are real in the sense that now the reality condition, corresponding to the involution \eqref{inv}, 
becomes 
$$
\bar u_l = u_l,\quad \bar v_l=v_l\qquad \forall\, l\in\L\,,
$$
and the composition
$\ 
(r,\theta, u,v)\mapsto (r,\theta, \zeta^\L)  \stackrel{\Phi_\rho}{\longmapsto} \zeta\mapsto (u(x), v(x))
$ 
sends real vectors $(r,\theta, u,v)$ to real-valued  functions $(u(x), v(x))$. 

In the variables $(r,\theta, u,v)$ the symplectic form $-dr\wedge d\theta -id\xi\wedge d\eta$ reads
$$
\omega_2 =-dr\wedge d\theta -du\wedge dv, 
$$
The transformed Hamiltonian is   $\widetilde K \widetilde K$
\be\label{transham}
\begin{split}
(H_2+P)\circ \Phi_\yy = \Omega(\yy)\cdot r &+\frac12 \sum_{a\in\L_\infty} \Lambda_a(\yy)(u_a^2+v_a^2)\\
&+ \frac{\nu}2 \langle \widetilde K(\yy)\tilde\zeta_f,  \tilde\zeta_f\rangle + \tilde f(r,\theta, \tilde\zeta; \yy)\,.
\end{split}
\ee
  The assertion (ii)-(iv) of Proposition~\ref{thm-HNF} in the new 
variables stay almost the same -- we only  note that the hamiltonian operator $\nu\H(\yy)=iJ\nu K$ now reads 
$J\nu \widetilde K$. Here $\langle\widetilde K\tilde\zeta_f, \tilde\zeta_f\rangle$ is the quadratic form 
$\langle  K \zeta_f,  \zeta_f\rangle$, written in the variables $\tilde\zeta_f$.  So the 
spectrum of the operator $\H(\yy)$ equals that of the operator $J\widetilde K(\yy)$. 

The transformation \eqref{reall} obviously respects the decomposition \eqref{decomp}, and 
\be\label{decom}
 J \widetilde K(\yy) =  L^1(\yy)\oplus\dots\oplus  L^M(\yy),
\ee
where each $L^j(\yy)$ is a real operator in the space $Y^{fj}$, corresponding to a set $\L_f^j$.
We identify $Y^{fj}$ with $\C^{2n_j}$, where $n_j=|\L_f^j|$,  and identify the operator  $L^j(\yy)$
with its matrix in $\C^{2{n_j}}$. This is an $(2{n_j}\times2{n_j})$-matrix which polynomially depends
 on $\sqrt\yy$ and is real for real $\yy$.  

 It is easy to see that 
any one-dimensional set $\L_f^j = \{b_j\}\, \  j\le M_*$, contributes to the quadratic form 
$\frac{\nu}2 \langle \widetilde K(\yy)\tilde\zeta_f,  \tilde\zeta_f\rangle$ the term 
$\frac{\nu}2 \mu(b_j, \rho)(u_{b_j}^2 + v_{b_j}^2)$.  So 
\be\label{koko}
 \langle \widetilde K(\yy)\tilde\zeta_f,  \tilde\zeta_f\rangle=
 \sum_{j=1}^{M_*}  \mu(b_j,\rho) 
 \Big( u_{b_j}^2 + v_{b_j}^2\Big)+  \langle \widetilde K_*(\yy)\tilde\zeta_f,  \tilde\zeta_f\rangle\,,
\ee
where 
$$
J\widetilde K_*(\yy)= L^{M_*+1}(\yy)\oplus\dots\oplus  L^M(\yy),
$$
and each operator $L^r$,  $r\ge M_*+1$, is a block of size $\ge2$. The hamiltonian operators $L^j, j\le M_*$, 
corresponding to terms in the sum in \eqref{koko}, are given by the $2\times2$ matrices 
$L^j(\rho)= \mu(b_j,\rho)J$, and 
$$
J\widetilde K(\yy)= L^{1}(\yy)\oplus\dots\oplus  L^M(\yy)\,.
$$

According to \eqref{spectrum} and \eqref{koko} we  numerate the eigenvalues
$\{\pm i\Am_j, 1\le j\le \bold N\}$ in such a way that 
\be\label{kokoko}
\Am_j(\rho)= \Am^j_1(\rho) = \mu(b_j,\rho) \quad\text{if}\quad 1\le j\le M_*\,.
\ee

\subsection{Removing singular values of the parameter $\rho$}\label{rho_sing}

Due to Lemma~\ref{laK} we know that for each $j$  the eigenvalues $\{\pm i\Am_k^j(\yy)$,  $k\le n_j\}$,
  do not vanish identically
in $\yy$ and do not identically coincide. Now our goal is to quantify these statements by removing certain 
singular values 
of the parameter $\yy$. To do this let us first  denote   
$P^j(\yy)=(\prod_l\Lambda^j_l(\yy))^2 = \pm \det L^j(\yy)$  and consider the determinant 
$$
 P(\yy)=\prod _jP^j(\yy) =\pm \det J\widetilde K(\yy)\,. 
 $$
 
 Recall that for an $R\times R$-matrix with eigenvalues $\ka_1,\dots,\ka_R$ (counted with their muiltiplicities) 
 the discriminant of the determinant of this matrix 
 equals the product $\prod_{i\ne j}(\ka_1-\ka_j)$. This is a polynomial of the matrix' elements. 
 
 Next we define a ``poly-discriminant" $D(\rho)$, which is another   polynomial 
 of the matrix elements of  $J\widetilde K(\rho)$. Its definition is motivated by Lemma~\ref{l_nond}, and it is 
 different for the admissible and \sa\ sets $\A$.  Namely, if $\A$ is \sa, then 
 \smallskip

 -- for $r=1,\dots, M_*$  define  $D^r(\rho)$ as the discriminant of the determinant of the matrix
 $ L^r(\yy)\oplus L^{M_*+1}(\yy)\oplus\dots\oplus  L^M(\yy)$; 
 
 -- set $D(\rho) =
 D^1(\rho)\cdot\dots\cdot D^{M_*}(\rho)$.
 
 \noindent 
 This is a polynomial in the matrix coefficients of $J\widetilde K(\yy)$, so a polynomial of  $\sqrt\rho$.
It  vanishes if and only if 
 $\Am^r_m(\rho)$ equals $\pm \Am^l_k(\rho)$ for some $r, l, m$ and $k$, where either $r,l \ge M_*+1$ and
 $m\ne k$ if $r=l$, or $r\le M_*$ and $m=1$.
  \medskip
 
  If $\A$ is admissible, then  we:
 
 
 -- for $l\le M_*, r\ge M_*+1$ define $D^{l,r}(\rho)$ as  the discriminant of the determinant of the matrix
 $ L^l(\yy)\oplus  L^r(\yy)$; 
 
 -- set $D(\rho) = 
 \prod_{l\le M_*,   r\ge M_*+1} D^{l,r}(\rho) $.

 \noindent 
 This is a polynomial in the matrix coefficients of $J\widetilde K(\yy)$, so a polynomial in $\sqrt\rho$. It 
 vanishes if and only if 
 $\Am^r_1(\rho)$ equals $\pm \Am^l_k(\rho)$ for some $r\le M_*$, some $l\ge M_*+1$ and some $k$,
 or if $\Am^l_k(\rho)$ equals $\pm\Am^l_m(\rho)$ for some $l\ge M_*+1$ and some $k\ne m$.
 \medskip
 
 Finally, in the both cases we set
 $$
 M(\rho) =  \prod_{b\in \L_f }  \mu(b,\yy) 
  \prod_{\substack{b,b'\in \L_f \\ |b| \ne |b'| }} \big( \mu(b,\yy) - \mu(b',\yy)\big).
 $$
 This also is a polynomial in $\sqrt\rho$ which does not vanish identically due to \eqref{N110}.

 The set
 $$
 X=\{\yy\mid  P(\yy)\,D(\yy)\, M(\yy)
 =0\}
 $$
 is an algebraic variety, if written in the variable $\sqrt\rho$ (analytically diffeomorphic to the variable 
 $\rho\in[c_*,1]^\A$), and is  non-trivial by Lemma~\ref{laK}. 
 The open set $\D\setminus X$ is dense in $\D$ 
 and is formed by finitely many connected components. Denote them $Q_1,\dots, Q_L$.
 For any component $Q_{l}$ its boundary is a stratified  analytic manifold with finitely many smooth analytic components
 of dimension $<n$, see  \cite{BR, KrP}. The eigenvalues $\Lambda_j(\yy)$ and the corresponding eigenvectors are locally 
 analytic functions on the domains $Q_l$, but since some of these domains 
   may be not simply connected, then the functions may have 
 non-trivial monodromy, which would be inconvenient for us. But since each $Q_l$ is a domain with a regular boundary, then 
 by removing from it  finitely many  smooth closed hyper-surfaces we cut $Q_l$
  to a finite system of simply connected domains $Q^1_l, \dots, Q_l^{\hat n_l}$ 
 such that their union has the same measure as 
  $Q_l$ and  each domain  $Q_l^\mu$ lies on one side of its boundary.\footnote{For
 example, if $n=2$ and $\DD$ 
  is the annulus $A=\{1<\yy_1^2+\yy_2^2<2\}$, then we remove
from $A$ not the interval $\{\yy_2=0, 1<\yy_1<2\}=:J$ (this would lead to a simply connected domain which lies on both 
parts of the boundary $J$), but two intervals, $J$ and $-J$. 
}
We may realise these cuts (i.e. the
 hyper-surfaces)  as the zero-sets of certain 
 polynomial functions of $\yy$. Denote by $R_1(\yy)$ the 
product of the polynomials, corresponding to the cuts made,  and remove from $\DD\setminus X$ 
the zero-set of $R_1$. This zero-set contains all the cuts we made  (it may be bigger than the union of the cuts), 
 and still has zero measure. Again, $(\tilde Q_l\setminus X)\setminus \{\text{zero-set of} \ R_1\}$ 
  is a finite union of 
domains, where each one lies in some domain $Q^r_l$. 

Intersections of these new domains with the sets $\D^{j_*}_0$ (see \eqref{DD}) will be important for us by virtue of Lemma~\ref{laK},
and any fixed set $\D^{j_*}_0$, say $\D^{1}_0$, will be sufficient for out analysis. To agree the domains with $\D^1_0$ we note that 
the boundary of $\D^1_0$ in $\D$ is the zero-set of the polynomial
$$
R_2(\yy) = (\yy_1-(1-c_{**}))(\yy_2-c_{**})\dots(\yy_n-c_{**})\,,
$$
and  modify the set $X$ above to the set $\tilde X$, 
$$
\tilde X = \{\yy\in\D\mid \Rc(\yy)=0\}\,,\qquad \Rc (\yy) = P(\yy) D(\yy) M(\yy) R_1(\yy) R_2(\yy)\,. 
$$
As before, $\D\setminus \tilde X$ is a finite union of open domains with regular boundary. We still denote them $Q_l$:
   \be\label{components}
 \D\setminus \tilde X = Q_1\cup\dots\cup Q_{\mathbb J}\,,\qquad  \bJ<\infty\,.
 \ee
 A domain $Q_j$ in \eqref{components} may be non simply connected, but since each $Q_j$
 belongs to some domain $Q^r_l$, then the eigenvalues $\La(\yy)$ and the corresponding
 eigenvectors define in these domains single-valued analytic functions. Since every domain 
 $Q_l$ lies either in $\D^1_0$ or in its complement, we may enumerate the domains $Q_l$ in
 such a way that 
 \be\label{newJ}
\D_0^1 \setminus \tilde X = Q_1\cup\dots\cup Q_{\bJ_1}\,,\quad 1\le \bJ_1\le \bJ\,.
\ee
The domains $Q_l$  with  $l\le \bJ_1$ will play a special role in our argument. 
\smallskip
 
 We naturally extend $\tilde X$ to a complex-analytic subset $\tilde X^c$ of $\D_{c_1}$ (see \eqref{Dset}), 
 consider the set $\D_{c_1}\setminus \tilde X^c$, and for 
 any $\delta>0$ consider its closed sub-domain $\Da(\delta)$, 
$$
\Da(\delta) = \{ \yy\in\Da\mid |\Rc(\yy)| \ge   \delta
\}\subset \Da\setminus \tilde X^c\,.
$$
Since the factors, forming $\Rc$, are polynomials with bounded coefficients, then they are bounded in $\D_{c_1}$:
 \be\label{det}
 \|P\|_{C^1(\D_{c_1})} \le C_1\,, \dots,   \|R_2\|_{C^1(\D_{c_1})} \le C_1\,.
 \ee
So in the domain  $ \Da(\delta)$ the norms of the factors $P,\dots, R_2$, making $\Rc$, are bounded from 
below by $C_2\delta$, and similar estimates hold for the factors, making $P$, $D$ and $M$. 
Therefore, by the Kramer rule
\be\label{Kram}
\|(J\widetilde K)^{-1}(\rho)\| \le C_1 \delta^{-1}\qquad \forall \rho\in \Da(\delta)\,.
\ee
Similar  for   $\yy \in \Da(\delta)$ we have 
\be\label{K4}
  |\Am^j_{k}(\yy)  |\ge C^{-1} \delta \qquad \forall j,k\,,
\ee
\be\label{K44}
 |\mu(b,\rho) |
  \ge C^{-1} \delta\,, \quad
 |\mu(b,\rho) - \mu(b', \rho)|
  \ge C^{-1} \delta \quad \text{if $\ b,b'\in\L_f\ $ and $|b|\ne |b'|$\,,
  }
\ee
and 
\be\label{K04}
   |\Am^j_{k_1}(\yy) \pm\Am^r_{k_2} (\yy)|\ge C^{-1} \delta \quad \text{where}\;\; (j,k_1)\ne (r,k_2)\,.
\ee
In \eqref{K04} if the set $\A$ is\sa, then 
 the index $j$ is any and $r\ge M_*+1$, while if $\A$ is admissible, then  either $j\le M_*$ (and so 
 $k_1=1$)  and  $r\ge M_*+1$,
or $j=r\ge M_*+1$. The functions $\Am^j_{k}(\yy)$ are algebraic  functions on the complex
domain $\Da(\delta)$, but their restrictions to the real domains $Q_l$ split to branches which are well defined analytic functions.

By Lemma~\ref{lA1} (applied to the domain $\{\sqrt\rho\mid \rho\in[c_*,1]^{\A}$)
\be\label{K1}
\meas (\D\setminus \Da(\delta))\le C \delta^{\beta_{4}},
\ee
for some positive $C$ and $\beta_{4}$. Denote $c_2=c_1/2$, define set $\Db$  as
in \eqref{Dset} but replacing there $c_1$ with  $c_2$, and denote $\Db(\delta) =\Da(\delta) \cap \Db$. 
Obviously,
\be\label{two_sets}
\text{the set $\D_{c_2}(2\delta)$ lies in $\Da(\delta)$ with its $C^{-1}\delta$-vicinity\,.
}
\ee
Consider the eigenvalues $\pm i \Am_k(\yy)$. They analytically depend on $\yy\in\Da(\delta)$, where $|\Am_k|\le C_2$ for
each $k\le \bold N$ by \eqref{N1}. In view of \eqref{two_sets}, 
\be\label{K2}
|\frac{\p^l}{\p\yy^l} \Am_k (\yy)| \le C_l\delta^{-l}\qquad \forall\,\yy\in\Db(2\delta)\,,\ l\ge0\,,\  k\le \bold N\,,
\ee
by the Cauchy estimate.

\subsection{Diagonalising} 
For real $\yy$ the spectra of the operator $J\widetilde K(\yy)$ and of each operator 
 $L^l(\yy)$ (see \eqref{decom})  are invariant with respect to the involution
$z\mapsto -z$ and the 
complex conjugation. When the spectrum of  $J\tilde  K$ 
does not contain zero, we have three possibilities for its eigenvalues $i\Lambda_j$:

a) $i\Am_{j} \in i\R\setminus\{0\}$;

b)  $ i\Am_{j} \in  \R\setminus\{0\}$;

c)  $i\Am_{j} \in \C\setminus( \R \cup i\R)$. 

\noindent
The eigenvalues  of  the last type may be arranged in quadruples
 $\pm i\Am_{j}$, $\pm i\Am_{j+1}$,  where 
 $\Am_{j+1}=\bar\Am_{j}$. Due to \eqref{K04} the type of an eigenvalue does not 
 change while $\yy$ stays in a connected component  $Q_l$ of $\D\setminus \tilde X$. 
 We recall  our agreement \eqref{kokoko}; in particular, for $j\le M_*$ the eigenvalues
 $\pm i\Am_j$ have the type a).

 For $\rho\in Q_l$  denote by $\L_l^h$   the set of indices $j$ such that 
 $ \Lambda_j$ is of type b) or c), i.e. the corresponding eigenvalues $\pm i\Lambda_j(\rho)$ 
  are hyperbolic  (the cardinality of  $\L_l^h$ may depend on $l$, but not on $\rho\in Q_l$). 
  Now take any $\Lambda_j(\rho)$, where $\rho\in Q_l$,
 $j\in \L_l^h$ and  write it as $\Lambda_r^m(\rho)$, $m\le M$ (see \eqref{spectrum}). By the 
 symmetries of the spectrum, $\overline{\Lambda_r^m}(\rho)$ equals  $\Lambda_{r'}^m(\rho)$
 for some $r'$. Therefore, by \eqref{K04}, 
 \be\label{K044}
  |\Im \Lambda_j(\rho)| \ge \frac12 C^{-1} \delta \;\; 
  \text{ if $\Lambda_j(\yy)$ is of type b) or c)  and $\rho\in\Da(\delta)$\,.
  }
 \ee
 Finally, let us denote 
  $$
  \tilde Q_j=Q_j\cap \Da(2\delta)  \quad\text{for}\quad j\le\mathbb J\,.
 $$
 These are closed domains, connected if $\delta\ll1$, such  that
 $$
 \cup \tilde Q_l = \Da(2\delta)  \cap \D\,. 
 $$
 By construction, the  restrictions to the eigenvalues $\Am_j(\yy)$ to the  real
 domains $\tilde Q_l$ are single-valued analytic functions. 
   Since $\Da(2\delta)$ satisfies a natural version 
 of the estimate \eqref{K1}, then $\meas (\D\setminus \Da(2\delta) )\le C'\delta^{\beta_{4}}$. So
 \be\label{K1meas}
 \sum \meas(Q_j\setminus \tilde Q_j) = 
 \meas (\D\setminus \Da(2\delta) )\le C'\delta^{\beta_{4}}\,.
 \ee

In view of \eqref{K04} and \eqref{decom}, 
for $\yy\in  \tilde Q_l$, $l\le \mathbb J$,  
the matrix $J\widetilde K(\yy)$ has complex 
eigenvectors $U_l(\yy), 1\le l \le 2{\bold N}$, corresponding to the eigenvalues 
$\pm i\Am_j(\yy)$, which analytically depend on $\yy$. We normalise them to have unit length and numerate 
in such a way that the eigenvector $U_{2l}$ corresponds to $i\Am_l$ and $U_{2l-1}$ corresponds to $-i\Am_l$. 
We denote by $U(\yy)$ the complex  matrix
with the column-vectors $U_1(\yy),\dots, U_{2{\bold N}}(\yy)$. It  is  analytic  in $\yy\in \tilde Q_l$ and 
 diagonalises $J\widetilde K(\rho)$:
\be\label{diago}
U(\yy)^{-1} J\widetilde K(\yy)  (\yy)U(\yy)=L^0(\yy)=i\,\text{diag}\,\{\pm\Am_1(\yy),\dots,\pm\Am_{{n}}(\yy)\}. 
\ee
Clearly  $\|U(\yy)\|\le \sqrt{2{\bold N}}$. In view of \eqref{K04} and Lemma~\ref{l_sver} 
\be\label{K5}
\|U(\yy)^{-1}\|\le C_4\delta^{-\beta_{5}},\quad \forall\,\yy\in\tilde Q_l; \;\;
 \beta_{5}= 2{\bold N}-1. 
\ee

Now we will modify $U(\rho)$ to a symplectic transformation which still diagonalises  $J\widetilde K(\rho)$.
Denote by $\omega_2$ the symplectic form on the complex space $Y^f$,  
 $\omega_2(v_1,v_2)=\langle Jv_1,v_2\rangle$, where $\langle\cdot,\cdot\rangle$ is the 
 complex-bilinear form,  and  $J$ is the symplectic matrix (see Notation). 
 
  It is well known 
(see \cite{Ar, MS}) that $\omega_2(U_a,U_b)=0$, unless $\Am_a=-\Am_b$. That is,
\be\label{K6}
\omega_2(U_{2j-1}, U_{k}) = \delta_{2j, k} \pi_j(\yy),
\ee
where $\delta$ is the Kronecker symbol. It is clear that $|\pi_k(\yy)|\le1$. For the same 
reason as in Appendix~C (see there \eqref{argument}), estimate \eqref{K5} implies that 
\be\label{K7}
1\ge |\pi_k(\yy)|\ge C_4^{-1} \delta^{\beta_{5}}\qquad \forall k,\; \forall \yy\in\tilde Q_l\,,\quad \beta_{5}=2{\bold N}-1\,.
\ee
Let us re-normalise the vectors $U_k(\yy), 1\le k\le2{\bold N}$, to vectors $\tilde U_k(\yy)$, where
$$
\tilde U_{2j-1} (\yy)= \pi_j^{-1}(\yy)U_{2j-1}(\yy), \quad
\tilde U_{2j} (\yy)=  U_{2j}(\yy).
$$
The modified conjugating operator  is 
$
\widetilde U = U\cdot  \text{diag}\,( \pi_1^{-1},1,\pi_2^{-1},\dots, 1), 
$
and for its columns -- the eigen-vectors $\tilde U_j$ --  the relations \eqref{K6} hold 
with $\pi_j\equiv1$. So the transformation $\widetilde U(\yy)$ is symplectic, and we still have 
\be\label{Util}
\widetilde U(\yy)^{-1} J\widetilde K (\yy) 
\widetilde U(\yy) = L^0(\yy). 
\ee
 In view of \eqref{K5}, \eqref{K7} it satisfies 
\be\label{LL}
\|\widetilde U(\yy)\| + \|\widetilde U(\yy)^{-1}\| \le C_5 \delta^{-\beta_{5}},\qquad \forall \yy \in\tilde Q_l\,,\;\; \forall l
\,. 
\ee
The operator  $\widetilde U$ respects the decomposition \eqref{decom} and equals  a direct sum 
of $M$ symplectic transformations, acting in the spaces $Y^{fj}$. 

The operator $J\widetilde K(\yy)$ Whitney--smoothly depends on $\yy\in\DD$  and its eigenvalues satisfy
\eqref{K04}. Since the diagonalising transformation $\widetilde U$ respects the block-decomposition 
\eqref{decom}, satisfies \eqref{LL}, and  since \eqref{K04} holds for $r=j$ and any $k_1\ne k_2$, 
 then the basic perturbation theory for simple eigenvalues implies that 
$\widetilde U(\yy)$ smoothly depends on $\yy$ and 
\be\label{K8}
\sup_{\yy\in\DD} 
 \big( \|\p^j_\yy\widetilde U(\yy)\| + \|\p^j_\yy\widetilde U(\yy)^{-1} \|) \le C_j\delta^{-\beta(j) }
 \qquad \forall\, j\ge0\,.
 \ee
  The positive 
constants $\beta(j)$ depend only on $g(\cdot), m$ and $\A$. Their  explicit form is not important for us.

Since for $j\le M_*$ the operator $L^j(\rho)$  is given by the $2\times2$-matrix 
$\mu(b_j,\rho)J$,  then its normalised eigen-vectors are   $\ ^t(1,i)/\sqrt2$ and $\ ^t(1,-i)/\sqrt2$.
So the corresponding $j$-th  block of the operator $\widetilde U$ is 
$$
\widetilde U(\yy) \mid_{Y^{fj}} =   \frac1{\sqrt2} 
 \left(\begin{array}{ll}
1& 1\\
i& -i
\end{array}\right) =: \Upsilon,\qquad 1\le j\le M_*\,.
$$
Accordingly the transformation $\widetilde U$ may be written as
$$
\widetilde U(\yy) =\underbrace{ (\Upsilon\oplus\dots\oplus \Upsilon)}_{M_*\; \text{terms}}
\oplus \widetilde U_*(\yy)\,,
$$
and similar $L^0(\yy) =\diag\{\pm i\mu(b_j,\yy)  \}\oplus L^0_*(\yy)$. 

In accordance with a), b) and c), we decompose further  the complex diagonal 
operator $L^0_*$ as 
$$
L^0_*(\yy) = L^a(\yy) \oplus L^b(\yy) \oplus L^c(\yy). 
$$
 Re-ordering the 
eigenvalues $\pm i\Lambda_j$ we achieve that the elliptic eigenvalues, represented in the 
decomposition above as the eigenvalues of $ L^a(\yy)$,  are 
\be\label{ellpart}
\{ \pm i\Lambda_r(\yy), M_*< r\le M_{**}\}, \quad M_{**}\le \bf N\,,
\ee
and  the eigenvalues $\pm i\Lambda_r$, $r>M_{**}$, are hyperbolic (i.e., their real parts are non-zero).

The complex-diagonal  operator $L^a(\yy)$ is elliptic, and a direct sum of operators $\Upsilon^{-1}$ transform it to a 
hamiltonian operator, corresponding to the real Hamiltonian 
\be\label{hham}
\frac12\  \sum_{\text{$i\Lambda_j$ has type a)}} \pm\Lambda_j(\yy)(u_j^2 +v_j^2)\,. 
\ee
The sign $\pm$ depends on the Krein signature of the pair of eigenvalues $\pm i \Lambda_j(\yy)$, see
\cite{Ar}. Since the eigenvalues $i \Lambda_j(\yy)$ in \eqref{spectrum} are defined up to multiplication 
by $\pm1$, then changing the signs for some of them we achieve that all the signs in \eqref{hham}
are ``+".\footnote{So, in general, the functions $\Lambda_j(\yy)$ of type a) cannot be chosen positive.} 
This argument does not apply to the domains $\D_0^{j_*}$ as in Lemma~\ref{laK}, where the sign of each 
$\Lambda_j(\yy)$ is fixed by the agreement 
\eqref{agreem}. But for $\yy$ in that domain still all the signs are plus.
Indeed, for $\yy$ in the vicinity of $\yy_*$ this is true by Lemma~\ref{laK} and \eqref{agreem}. Let us take any 
point $\yy'\in\D_0^{j_*}$ and consider a deformation in $\D_0^{j_*}$ of any point, close to $\yy_*$, to $\yy'$. 
During the deformation the functions $\Lambda_j$ stay real and do not vanish by Lemma~\ref{laK}, so
the sign remains ``+". 

The operator $L^b(\yy)$ is real-diagonal hyperbolic, and we do not touch it. Consider the complex-diagonal
hyperbolic 
operator $L^c(\yy)$. It splits to a direct sum of 4-dimensional diagonal operators, where each one has 
eigenvalues $(\pm a(\yy) \pm i b(\yy))$, $a,b\ne0$. As we show in Appendix~C, a symplectic operator
$({\widetilde U}^{a,b})^{-1}$, satisfying the estimates \eqref{Wils}, transforms this complex-diagonal 
hamiltonian operator to the operator with a real Hamiltonian which is the  one-half of the quadratic form 
 \eqref{a,b}. This operator depends on $\yy$. In view of \eqref{K04}
$$
|a(\yy)|, |b(\yy)|  \ge C^{-1} \delta\qquad \forall\,\yy \in \tilde Q_l\,,\;\forall\, l\,,
$$
and for the same reason as above the derivatives $\p^j_\yy \Big( ({\widetilde U}^{a(\yy),b(\yy)})^{\pm1}\Big)$
satisfy estimates \eqref{K8}. 

Now consider  the operator 
$$
\widehat U(\yy) = \big(\Upsilon\oplus\dots\oplus\Upsilon  \big) \oplus\text{ id }
\oplus \big(\bigoplus_j {\widetilde U}^{a_j(\yy),b_j(\yy)} \big)\,,
$$
where the first direct sum in the r.h.s. acts in the sub-space, 
 corresponding to the eigenvalues $i\Lambda_j$ of type a), the identity transformation acts 
 in the sub-space, 
 corresponding to the eigenvalues  of type b), and the last direct sum corresponds to the eigenvalues 
 $\Lambda_j = a_j+ib_j$ of type c). The operator $(\widehat U(\yy))^{-1}$   transforms the complex-diagonal 
operator $L^0(\yy)$ to the hamiltonian operator, corresponding to   a Hamiltonian 
\be\label{rham}
\frac12
 \sum_{j=1}^{M_*}  \mu(b_j,\rho)  \Big( u_{b_j}^2 + v_{b_j}^2\Big)+ 
 \frac12
 \sum_{j=M_*+1}^{M_{**}}  \Lambda_j(\yy)  \Big( u_{b_j}^2+ v_{b_j}^2\Big) + \frac12
  \langle \widehat  K(\yy)\tilde\zeta^h_f,  \tilde\zeta^h_f\rangle\,.
\ee
Here the vector $ \tilde\zeta^h_f$ is formed by the hyperbolic components of the vector $ \tilde\zeta_f$, and 
the spectrum of the hamiltonian operator $J \widehat  K(\yy)$ is formed by the hyperbolic eigenvalues of
the operator $J\widetilde K(\yy)$. 
The operator  $\widehat U(\yy)$ satisfies the estimates \eqref{K8} with suitable exponents 
$\beta(j)$.

 The operator  $(\widehat U(\yy))^{-1}\circ \widetilde U(\yy)$ transforms the Hamiltonian \eqref{koko}
to the Hamiltonian above and also satisfies the estimates \eqref{K8} (with modified exponents 
$\beta(j)$).

\subsection{Final transformation}\label{s_44}
We have
\be\label{mesur}
\meas \cup_{j \in  \mathbb J} \tilde Q_j =
\meas \Da(2\delta) \cap (\D\setminus \tilde X) = \meas \Da(2\delta) \cap \D\,.
\ee
The number $\mathbb J$ does not depend on $\delta$, but the closed  domains $\tilde Q_j$ depend on it, $\tilde Q_j
=\tilde Q_j(\delta)$,  and for each $j$, 
$$
\tilde Q_j(\delta)\nearrow Q_j \quad \text{as}\quad \delta\to0\,,
$$
i.e., $\tilde Q_j(\delta_1) \subset \tilde Q_j(\delta_2) $ if $\delta_1>\delta_2$, and $\bigcup_{\delta>0} \tilde Q_j(\delta) =Q_j$.

For any $j\le \bJ$ let us consider the operator  $\widetilde U(\yy),\ \yy\in{\tilde  Q_j}$, as in \eqref{Util} 
and denote it $\widetilde U_j(\yy)$. 
Now we  define the operators $\widehat {\bold U}(\rho)$ and $ {\bold U}(\rho)$, $\rho\in\cup \tilde Q_j$,
by the following relations:
$$
\widehat {\bold U}(\rho) = 
\widetilde U_j(\rho),\;\;\;   {\bold U}(\rho) = 
(\hat U_j(\rho))^{-1} \widetilde U_j(\yy)
   \quad \text{if}\; \rho \in \tilde Q_j\,.
$$

With an eye on the relations \eqref{K1meas} and \eqref{mesur}, for 
 any $\beta_*>0$ and $\nu>0$ we choose $\delta>0$ such that 
\be\label{delta}
C\delta^{\beta_{4}} = \nu^{\beta_*},
\ee
where $C$ and $\beta_{4}$ are the constant and the exponent from \eqref{K1}. For convenience we denote (see \eqref{K7})
\be\label{barc}\bar c=\frac1{\beta_{4}}\quad \text{and} \quad \hat c=\beta_{5}\bar c=(2\bold N-1)\bar c .\ee

The Hamiltonian \eqref{transham} is written in the variables $(r,\theta, \tilde\zeta^\L)$, where 
$\tilde\zeta^\L=(\tilde\zeta_f,  \tilde\zeta_\infty)$, according to the decomposition $\L=\L_f\cup\L_\infty$.
Now we decompose the variable $\tilde\zeta_f$ further. Namely we write 
\be\label{deccom}
\L_f=\L_f^e\cup \L_f^h\,,
\ee
and
$$
\tilde\zeta_f =(\tilde\zeta_f^e, \tilde\zeta_f^h)\,,\quad \text{where}\;\; \tilde\zeta_f^e=(\tilde\zeta_b, b\in\L^e_f),\;\;
\tilde\zeta_f^h=(\tilde\zeta_b, b\in\L^h_f)\,,
$$
where the sets $\L_f^e$ and $\L_f^h$ correspond, respectively, to the elliptic and hyperbolic eigenvalues $i\Lambda_j(\yy)$. 
The sets $\L_f^e$, $\L_f^h$ and the decompositions above depend on the domain $\tilde Q_l$. In particular, 
in view of \eqref{agreem} and \eqref{newJ} 
$$
\L_f^h=\emptyset\quad\text{if}\quad \yy\in\D_0^1 \,.
$$
Recalling that $\mu(b_j,\rho) = \Lambda_j(\rho)$ if $ 1\le j\le M_*$, we write the quadratic 
Hamiltonian \eqref{rham}, where 
$\rho\in \cup \tilde Q_l$, as 
$$
\frac12
 \sum_{b\in\L^e_f} \Lambda_b(\yy)      \Big( u_{b}^2 + v_{b}^2\Big)+ 
  \frac12
  \langle \widehat  K(\yy)\tilde\zeta^h_f,  \tilde\zeta^h_f\rangle\,.
$$

The transformation $\bU$, constructed above, acts on the variables $\tilde\zeta_f$. We extend it identically to the
variables $(r,\theta, \tilde\zeta_\infty)$, and denote the extension (acting on the variables $(r,\theta, \tilde\zeta^\L)$)
as $\bU\oplus\,$id. The normal form transformation from Proposition~\ref{thm-HNF} and the constructions, made earlier in this
section, jointly yield the transformation 
$$
\tilde\Phi_\yy=\Phi_\yy\circ\Sigma\circ(\bU^{-1}\oplus\text{id})\,.
$$
We recall that $\cup_{j\in\bJ} \tilde Q_j(\nu) = \Da(2\delta)\cap \D$,  denote $\Da(2\delta)\cap \D = 
\tilde Q(\nu) $ and 
 sum up  properties of this transformation  in the form of a normal form 
 theorem:

\begin{theorem}\label{NFT} There exists a zero-measure Borel set $\Ca \subset[1,2]$ such that for any 
admissible set  $\A$ and any  $m\notin \Cc$ 
  there exist 
 real numbers
  $\beta_{*0},  \nu_0,\ga_* \in(0,1]$,  $ c_{**} \in(0,1/2]$, 
   where $\ga_*$ depends only on
$g(\cdot)$ and  $c_{**}$, $\beta_{*0}$ and  $\nu_0$ 
also depend on $\A ñ$ and $ m$,  with the following property:\\
For any $c_*\in (0, \tfrac12 c_{**}]$ the set $\D=[c_*,1]^n$ has an algebraic subset $\tilde X$ which is a zero-set of 
a polynomial of $\sqrt\yy$,  depending on $\A, m$, and 
 for any $0<\nu\le\nu_0$ and $0<\beta_* \le\beta_{*0}$
 there exists a   closed semi-analytic  set 
  $\tilde Q(\nu)  \subset\D\setminus\tilde X$,
 such that 
$\tilde Q (\nu)\nearrow \big( \D \setminus\tilde X\big)$ as  $\nu\to0$, and 
\be\label{mesmes}
\meas (\D \setminus  \tilde Q (\nu)
)
\le C\nu^{\beta_*}\,.
\ee
 For    $\yy\in  \tilde Q=\tilde Q(\nu)$ and $0\le\ga\le\ga_*$ 
 there exist  real holomorphic transformations
$$
\quad
\tilde\Phi_\yy : \O^\gamma(\tfrac 12, {\mu})=\{(r,\theta, u,v)\}
\to \Tg(\nu,1,1,\gamma)\,,\qquad {\mu}={c_*}/{2\sqrt 2}\,,
$$
which do not depend on $\ga$ in the sense that they coincide on the smallest set
 set $\O^{\gamma_*}(\tfrac 12,{\mu})$.
The transformations 
smoothly depend on $\yy$ and   satisfy 
$(\tilde\Phi_\yy)^*(-id\xi\wedge d\eta) =\nu( -dr\wedge d\theta - du\wedge dv) $. 
With respect to the symplectic structure $-dr\wedge d\theta - du\wedge dv$ the transformed 
system has the 
Hamiltonian $(H_2+P)\circ\tilde\Phi_\yy=:H_\yy$ (see \eqref{H1}), where 
\be\label{transf}
\begin{split}
H_\yy&=
\Omega(\yy)\cdot r +\frac12 
\sum_{a\in\L_\infty} \Lambda_a(\yy)(u_a^2+v_a^2)\\
&+ \frac\nu2   \Big(\sum_{b\in\L^e_f} \Lambda_b(\yy)      \Big( u_{b}^2 + v_{b}^2\Big)
 +  \langle \widehat K(\yy)\tilde\zeta_f^h,  \tilde\zeta_f^h\rangle\Big)
+\tilde f(r,\theta, \tilde\zeta; \yy). 
\end{split}
\ee
Here  the real symmetric operator $\widehat K(\yy)$
acts in a space of dimension $2|\L_f^h|$, $\L_f^h = \L_f\setminus \L_f^e$. The decomposition 
$\L_f = \L_f^e\cup \L_f^h$ depends on the component of the domain $\D\setminus \tilde X$ which contains $\yy$,  
and for some of these components the set $\L_f^h$ is empty (so the operator $\widehat K(\yy)$ is trivial). Moreover,

i)  the
 functions $\Omega$ and $\Lambda_a, a\in\L_\infty$, are the same as in Proposition~\ref{thm-HNF},
and 
  the function    $\tilde f$ satisfies 
\be\label{est111}
[\tilde f]^{\gamma, D}_{1/2, \mu, \tilde Q }\le C \nu^{-\hat c\beta_*}  \nu \,,
\quad   [\tilde f^T]^{\gamma, D}_{1/2, \mu, \tilde Q }\le C \nu^{-\hat c\beta_*}  \nu^{3/2} \,,\quad \mu = (c_*/2\sqrt2)\,.
\ee
 The functions $\Lambda_b(\yy)$, $b\in\L_f^e$, are real analytic in $\tilde Q $ and 
\be\label{H2}
\|\Lambda_b\|_{C^r(\tilde Q )}\le C_r\nu^{- r\bar c\beta_*} \; (r\ge0),\;\;
\; \; \forall\,\yy\in  \tilde Q
\,.
\ee
They  satisfy \eqref{K4} and  
for some connected components of $\D\setminus \tilde X$ also 
satisfy \eqref{aaa}.

ii)  The operator $\widehat K(\yy)$ smoothly  depends on $\yy\in \tilde Q $ and may be diagonalised by a 
complex symplectic operator $\wU(\yy)$:
$$
\wU(\yy)^{-1} J \widehat K(\yy) \wU(\yy) = i\diag \{\pm \tilde\Lambda_{j}(\yy),  1\le j\le |\L_f^h|\}\,.
$$
The eigenvalues $\tilde\Lambda_{j}(\yy)$   satisfy \eqref{H2} and 
\be\label{hyperb}
| \Im \tilde\Lambda_j(\yy) | \ge C^{-1} \nu^{\bar c \beta_*}\qquad  \forall\,\yy\in  \tilde Q \,, \;  \forall\, j\,.
\ee
The operator $\wU(\yy)$ smoothly depends on $\yy$ and satisfies 
\be\label{Ubound}
\sup_{\yy\in\tilde Q }
 \big( \|\p^j_\yy \wU (\yy)\| + \|\p^j_\yy \wU  (\yy)^{-1} \|) \le C_j\nu^{-\beta_*\beta(j) }\,,
 \qquad \forall\, j\ge0\,.
 \ee
 Accordingly,
 \be\label{normK}
 \sup_{\yy\in\tilde Q }
  \|\p^j_\yy  J\widehat K(\yy)\|  \le C_j\nu^{-\beta_*\beta'(j) }\,\quad \text{for}\quad  j\ge0\,
\ee
 (the exponents $\beta(j)$ and $ \beta'(j)$ depend  on $m, \A$ and  $j$).

iii) The domains $ \tilde Q (\nu)$ and the matrix $\widehat K(\yy)$ do
 not depend on the component $g_0$ of the nonlinearity $g$. 
The constants $ C, C', \bar C$ etc 
and the factors $\bar c$, $\hat c$ in the  exponents   in the estimates above  do not depend on $\nu\in(0,\nu_0]$.
\end{theorem}

\section{ Proof of the non-degeneracy Lemma \ref{l_nond}  }\label{s_2d}

Consider the decomposition  \eqref{decomp} of  the  hamiltonian operator $\H(\rho)$.
To simplify notation, in this section we suspend the agreement that $|L^r_f|=1$ for $r\le M_*$,
and changing the order of 
the direct summands  achieve that the indices $r_1$ and $r_2$, involved in \eqref{single}, are $r_1=1$ and
$r_2=2$. For $r=1,2$ we  will write elements 
of the set $\L^r_f$ as $a^r_j, 1\le j\le n_r$, and vectors of the space $Y^{fr}$ as
\be\label{vectors}
\zeta=
\big(\zeta_{a^r_j}=(\xi_{a^r_j} ,  \eta_{a^r_j} ), 1\le j\le n_r\big) = 
\big( ( \xi_{a^r_1} , \eta_{a^r_1}),\dots, ( \xi_{a^r_{n_r}} , \eta_{a^r_{n_r}})\big)\,.
\ee

Using \eqref{labA} and abusing notation, we will regard the mapping 
$\ell: \L_f\to \A$ also as a mapping $\ell: \L_f\to \{1,\dots,n\}$. 
Consider the points $ \ell(a^1_1),\dots, \ell(a^1_{n_1}) $ (they are different by \eqref{.2}). 
Changing if needed the
labelling \eqref{labA}  we achieve that
\be\label{achieve}
\{ \ell(a^1_1),\dots, \ell(a^1_{n_1}) \ni 1\,. 
\ee

 We  write  the operator $\H^r$ as  $\H^r=iM^r$,  where 
$$
M^r (\rho)= JK^r(\rho) = JK^{r\,d}(\rho) +JK^{r\,n/d}(\rho) =: M^{r\,d}(\rho) +M^{r\,n/d}(\rho)\,,
$$
and  the real block-matrices $M^{r\, d}=i^{-1}\H^{r,\,d}$, $\ M^{r\, n/d}=i^{-1}\H^{r,\,n/d}$ are given
by \eqref{diag}. Then $\{\pm\Lambda^r_j(\rho)\}$ are the eigenvalues of $M^r(\rho)$, and
$$
M^{r\,d}(\rho) = \text{diag}\ \left( 
 \left(\begin{array}{ll}
 \mu(a^r_j,\rho) & 0\\
0& - \mu(a^r_j,\rho)
\end{array}\right) ,\;  1\le j\le n_r
\right),
$$
where 
$
  \mu(a^r_j, \rho)
$
is given by \eqref{.1}. 

 Renumerating the eigenvalues we achieve that in \eqref{single} (with $r_1=1, r_2=2$),
 $\Lambda^1_j=\Lambda^1_1$ and $\Lambda^2_k=\Lambda^2_1$.
  As in the proof of Lemma \ref{laK},  consider the vector $\rho_*={}^t(1,0,\dots,0)$.  Let us abbreviate 
$$
\mu(a, \rho_*)= \mu(a)\qquad \forall\, a\,,
$$
where $\mu(a)$ depends only on $|a|$ by  \eqref{munew}.
   In view of \eqref{diag} $M^r(\rho_*)=M^{r\,d}(\rho_*)$ and thus
 $\Lambda^1_1(\rho_*)=\mu(a^1_1)$ and
$\Lambda^2_1(\rho_*)= \mu(a^2_1)$, if we numerate the elements of 
 $\L^1_f$ and $\L^2_f$ accordingly. As in the proof of  Lemma~\ref{laK}, $\mu(|a^r_1|)$ equals
 $\tfrac12 C_*\lambda^{-2}_{a^r_1}$ or $- C_* \lambda^{-1}_{\ell(a^r_1)} \lambda^{-1}_{a^r_1}$. Therefore
 the relation 
 $\mu(a^1_1) =\pm \mu(a^2_1) $
 is possible only if the sign is ``+" and $|a^1_1| = |a^2_1|$. So it remains to verify that under the lemma's assumption
 \be\label{sin}
 \Lambda^1_1(\rho)\not\equiv \Lambda^2_1(\rho)\quad\text{if}\quad |a^1_1| = |a^2_1|\,.
 \ee
  Since $ |a^1_1| = |a^2_1|$, then 
$$
\ell(a^1_1) = \ell(a^2_1)=:{a_{j_{\#}}}\in\A \;\; 
\text{ and } \Lambda^1_1(\rho_*)=\Lambda^2_1(\rho_*)=: \Lambda\,.
$$

To prove that $\Lambda^1_1(\rho)\not\equiv \Lambda^2_1(\rho)$
 we compare  variations of the two functions   around $\rho=\rho_*$. To do this it 
   is convenient to pass from $\yy$ to the  new parameter $y=(y_j)_1^n$, defined by
 $$
 y_j=\sqrt{\rho_{j}}, \quad j=1,\cdots ,n.
 $$
 Abusing notation we will sometime write $y_{a_j}$ instead of $y_j$.
 Take  any vector $x=(x_1, \dots,x_n)\in\R^n $, where $x_1=0$ and $x_j > 0$ if $j\ge2$,
 and consider  the following variation $y(\eps)$ of $y_*=(1,0,\cdots,0)$:
\be\label{yx}
 y_j(\eps) =
\begin{cases}
1& \text{ if} \ \  j=1,\\
 \eps x_j &  \text{ if} \ \  j\ge2.
        \end{cases} 
\ee
By \eqref{N11}, for  small $\eps$ the real matrix  $M^r(\eps):=
M^r(\rho(\eps))$ $(r=1,2)$ has a simple eigenvalue
$\Lambda^r_1(\eps)$, close to $\Lambda$. 
 We will show that for a suitable choice of vector $x$ the  functions  $ \Lambda^1_1(\eps)$ and
$ \Lambda^2_1(\eps)$ are different. More specifically, that  their jets at zero  of sufficiently high order are different.

Let $r$ be 1 or 2. We denote $\Lambda(\eps) = \Lambda^r_1(\rho(\eps))$,
 $M(\eps)= M^r(\rho(\eps))$ and denote by $M^d(\eps)$ and $M^{n/d}(\eps)$ the diagonal and 
non-diagonal parts of $M(\eps)$. The matrix
 $M^{n/d}(\eps)$   is formed by $2\times2$-blocks
\be\label{M}
\Big( M^{n/d}(\eps)\Big)^{a^r_j}_{a^r_k}  = 
 C_*
 \frac{{ {y_{\ell(a^r_k)} y_{\ell(a^r_j)}}}}{ \lambda_{ a^r_k}  \lambda_{a^r_j }}  \,\left(
 \left(\begin{array}{ll}
0& 1\\
-1& 0
\end{array}\right) \chi^+( a^r_k,a^r_j )+
 \left(\begin{array}{ll}
1& 0\\
0& -1
\end{array}\right)  \chi^-( a^r_k,a^r_j )
\right)\,,
\ee
(note that if $j=k$, then the block  vanishes). 

For $\eps=0$, $M(0) = M^{rd}(0)$ is a matrix with the single eigenvalue 
$\Lambda(0) = \mu(a^r_1, \rho_*)$, corresponding
to the eigen-vector  $\zeta(0) ={}^t(1,0,\dots,0)$. For small $\eps$ they analytically 
extend to a real eigenvector $\zeta(\eps)$ 
 of $M(\eps)$ with the eigenvalue $\Lambda(\eps)$, i.e. 
$$
M(\eps)\zeta(\eps) = \Lambda(\eps) \zeta(\eps)\,,\qquad |\zeta(\eps)|\equiv 1\,.
$$

We abbreviate
$\zeta= \zeta(0), 
M=M(0)  
$
and define similar $ \dot \zeta, \ddot\zeta, \Lambda,\dot\Lambda\dots$ etc,
where  the upper dot  stands for $d/d\eps$. 
We have
\be\label{y1}
M=M^d
=\text{diag}\big(\mu(a^r_1), -\mu(a^r_1), \dots, -\mu(a^r_{n_r}) \big)\,,
\ee
\be\label{y01}
 \dot M^d=0\,.
\ee

Since $ (M(\eps) - \Lambda(\eps))\zeta(\eps)\equiv 0$, then
\be\label{y3}
(M(\eps) - \Lambda(\eps)) \dot\zeta(\eps) = -\dot M(\eps) \zeta(\eps) +\dot\Lambda(\eps) \zeta(\eps). 
\ee
Jointly with \eqref{y1} and  \eqref{y01} this relation with $\eps=0$  implies that 
\be\label{y4}
(M^d-\Lambda)\dot \zeta =-\dot M^{n/d} \zeta +\dot\Lambda \zeta.
\ee
In view of \eqref{y1}  we have $\langle(M^d-\Am)\dot\zeta, \zeta\rangle =0$. 
We derive from here and from \eqref{y4} that 
\be\label{y5}
\dot\Am=\langle \dot M^{n/d}\zeta,\zeta\rangle=0\,.
\ee
Let us denote by $\pi$ the linear projection 
$ \ 
\pi: \R^{2n_r} \to \R^{2n_r}  
$
which makes zero the first component of a vector to which it applies. Then $M^d-\Lambda$ is an isomorphism of the space 
$\pi  \R^{2n_r} $, and the vectors $\dot\zeta$ and $-\dot M \zeta+ \dot\Lambda\zeta = \dot M^{n/d}\zeta$ belong to
$\pi  \R^{2n_r} $. So we get from \eqref{y4} that 
\be\label{y6}
\dot\zeta= -(M^d-\Am)^{-1} \dot M^{n/d}\zeta\,,
\ee
where the equality holds in the space $\pi  \R^{2n_r} $.  
 Differentiating \eqref{y3} we find that 
\be\label{y7}
(M(\eps)-\Am(\eps))\ddot\zeta(\eps)= -\ddot M(\eps) \zeta(\eps) -2\dot M(\eps)\dot\zeta(\eps)
+\ddot\Am(\eps)\zeta(\eps)  +2\dot\Am(\eps)\dot\zeta(\eps)\,.
\ee
Similar to the derivation of \eqref{y5} (and using that $\langle\zeta, \dot\zeta\rangle=0$ since $|\zeta(\eps)|\equiv1$), 
we get from \eqref{y7} and \eqref{y5} that 
\be\label{y9}
\begin{split}
\ddot\Am =\langle\ddot M\zeta,\zeta\rangle+2\langle\dot M\dot\zeta,\zeta\rangle
=\langle\ddot M\zeta,\zeta\rangle +2\langle(M-\Am)^{-1}\dot M^{n/d}\zeta, {}^t(\dot M)\zeta\rangle\,.
\end{split}
\ee
Since for each $\eps$ and every $j$
$$
\frac{d^2}{d\eps^2}\rho_j(\eps)=\frac{d^2}{d\eps^2} y^2_{j} (\eps) =2 x^2_{j}\,,
\qquad
\frac{d^2}{d\eps^2} y_{1}(\eps) y_j(\eps)=0\,,
$$
and since $ \langle \ddot M \zeta, \zeta\rangle =\langle \ddot M^d \zeta, \zeta\rangle $, then 
\be\label{ddot}
\langle \ddot M \zeta, \zeta\rangle =
\frac{d^2}{d\eps^2}  \mu(a^r_1, \rho(\eps))\!\mid_{\eps=0} \,= 
C_* \lambda^{-1}_{a_{j_{\#}}}  \big( 3  \lambda^{-1}_{a_{j_{\#}}} x_{j_{\#}}^2
 -2 \sum^{n}_{j=2} x_{j}^2\lambda^{-1}_{a_j}\big)=:k_1\,. 
\ee
Note that $k_1$ does not depend on  $r$.

Now consider the second term in the r.h.s. 
 \eqref{y9}. For any $a,b \in\L^r_f$ we see that 
$\ 
\frac{d}{d\eps} (y_{\ell(a)}(\eps)y_{\ell(b)}(\eps))\mid_{\eps=0} \ 
$
  is non-zero if exactly one  of the numbers $\ell(a), \ell(b )$ is $a_1$, and this derivative 
  equals $x_{\ell (c)}$, 
where $c\in\{a,b\}$, $\ell (c)\ne a_1$. 
Therefore,  by  \eqref{M}, 
\be\label{y100}
\begin{split}
&(\dot M^{n/d}\zeta )_{{a^r_j}} =
\frac{C_*}{\lambda_{a_{j_{\#}}}}  (\xi^o_{{a^r_j}}, -\eta^o_{{a^r_j}}),\qquad {a^r_j}\in\L^r_f\,,
\\ 
& \xi^o_{{a^r_j}}= \frac{  \phi(a_1^r, a_j^r)  
}{\lambda_{{a^r_j}}}\chi^-(a^r_1, {a^r_j}), \quad   
\eta^o_{{a^r_j}}= \frac{   \phi(a_1^r, a_j^r)   }{\lambda_{{a^r_j}}}\chi^+(a^r_1, {a^r_j})\,,
\end{split}
\end{equation}
where $ \phi(a_1^r, a_1^r) =0$ and for $j\ne1$ 
$$
 \phi(a_1^r, a_j^r) =
\begin{cases}
x_{\ell(a^r_j)} & \text{ if} \ \ {j_{\#}}=1\,, \\
 x_{{{j_{\#}}}}  &  \text{ if} \ \  \ell(a^r_j)=a_1   \,,\\
         0 &  \text{ if} \ \  {j_{\#}}\ne1,\ \ell(a^r_j)\ne a_1\,. 
        \end{cases} \qquad 
$$
Since $\chi^\pm(a^r_1, a^r_1)=0$, then $\xi^o_{a^r_1} = \eta^o_{a^r_1}=0$.

In view of \eqref{obv}, at most one of the numbers $\xi^o_{{a^r_j}}, \eta^o_{{a^r_j}}$ is non-zero. 
By \eqref{y100}, 
\be\label{y11}
((M-\Lambda)^{-1}\dot M^{n/d}\zeta)_{{a^r_j}} =
\frac{C_*}{\lambda_{a_{j_{\#}}}} (\xi^{oo}_{{a^r_j}}, \,\eta^{oo}_{{a^r_j}}), 
\ee
where $\xi^{oo}_{{a^r_j}}= \eta^{oo}_{{a^r_j}}=0$ if $j=1$, 
and otherwise 
$$
 \xi^{oo}_{{a^r_j}}= 
\frac{ \phi(a_1^r, a_j^r) \chi^-(a^r_1, {a^r_j})  }{  \lambda_{{a^r_j}} (  \mu({a^r_j})-\mu(a^r_1)  )} , \;\;\;\;
\eta^{oo}_{{a^r_j}}=
\frac{ \phi(a_1^r, a_j^r)  \chi^+(a^r_1, {a^r_j})  }{  \lambda_{{a^r_j}} ( \mu({a^r_j})+\mu(a^r_1))} \,.
$$
Here 
$\ 
\mu({a^r_j}) =\frac12 C_*\lambda^{-2}_{a_1}$ if $\ell(a^r_j)=a_1$ and 
$\ \mu(a^r_j) = -C_*\lambda^{-1}_{a^r_l} \lambda^{-1}_{a_1}$
 if $\ell(a^r_j)\ne a_1$.

Similar,
$$
({}^t\dot M\zeta  )_{{a^r_j}} =
\frac{C_*}{\lambda_{a_{j_{\#}}}}  (\xi^o_{{a^r_j}}, \eta^o_{{a^r_j}}),
$$ 
so  the second term in the r.h.s. of \eqref{y9} equals 
\be\label{k2}
 \frac{C_*^2}{\lambda^2_{a_{j_{\#}}}}\sum_{j=2}^{n_r}   \frac{   \phi(a_1^r, a_j^r)^2 }{ \lambda^2_{a^r_j}}
 \Big(\frac{\chi^-(a_1^r, a^r_j)}{\mu(a^r_j)-\mu(a^r_1)  } +
 \frac{\chi^+(a_1^r, a^r_j)}{  \mu(a^r_j)+\mu(a^r_1)  }
 \Big)=: k_2(r)\ .
\ee

Finally, we  have seen that 
$$
\Am^r_1(\rho(\eps)) = \Lambda^1_1(\rho_*)+\frac12 \eps^2 k_1+ \frac12 \eps^2k_2(r) +O(\eps^3),\quad r=1,2,
$$
where $k_1$  does not depend on $r$. Since $a_1^r\sim a^r_j$ for each $r$ and each  $j$ (see \eqref{class}), 
 then for $j>1$ 
  at least one of the coefficients
$\chi^\pm(a_1^r,a^r_j)$ is non-zero. As  $\chi^+\cdot\chi^-\equiv0$, then
\be\label{non-zer}
  \frac{\chi^-(a_1^r, a^r_j)}{\mu(a^r_j)-\mu(a^r_1)  } +
 \frac{\chi^+(a_1^r, a^r_j)}{  \mu(a^r_j)+\mu(a^r_1)  }
  \ne0 \qquad \forall\,r,\;\; \forall\,j >1  \,.
\ee
We see that  the sum, defining $k_2(r)$, is a  non-trivial quadratic polynomial of the quantities 
$\phi(a_1^r, a_j^r)$ if  $n_r\ge2$, and vanishes if $n_r=1$. 

The following lemma is crucial for the proof. 
\begin{lemma}\label{l_triv}
If the set $\A$ is\sa \  and $|a|=|b|$, $a\ne b$, 
and $\chi^+(a,a')\ne0$,  $\chi^+(b,b')\ne0$, 
or  $\chi^-(a,a')\ne0$,  $\chi^-(b,b')\ne0$, 
then
$|a'|\ne |b'|$. 
\end{lemma}

\begin{proof}
 Let first consider the case when $\chi^+\ne0$.  \\
We know that $\ell(a)=\ell(b)=:{a_{j_{\#}}}$.  Assume that $|a'|=|b'|$. Then $\ell(a')=\ell(b')=: {{a_{j_\flat}}} \in\A$. 
Denote ${a_{j_{\#}}}+{{a_{j_\flat}}}=c$. Then $c\ne0$ since the set $\A$ is admissible.  As
$(a,a'), (b,b') \in (\L_f\times \L_f)_+$, then we have 
$\ 
|{a_{j_{\#}}}-c| = |a-c| = |b-c|\,. 
$
As $|{a_{j_{\#}}}| = |a| = |b|$, then the three points ${a_{j_{\#}}}, a$ and $b$ lie  in the intersection of two circles, one centred in
the origin and another centred in $c ={a_{j_{\#}}}+{{a_{j_\flat}}}$. Since $\A$ is\sa, then  ${a_{j_{\#}}}\an c$ (see 
\eqref{ddd}). So among the 
three point  two are equal, which is a contradiction. Hence,   $|a'|\ne|b'|$ as stated.

The case  $\chi^-\ne0$ is similar.  
\end{proof}

We claim that this lemma implies that 
\be\label{hi}
\Am_1^1(\rho(\eps))\not\equiv \Am^2_1(\rho(\eps)) \quad\text{for a suitable choice of the vector $x$ in
} \eqref{yx}  \,,
\ee
so \eqref{sin} is valid and Lemma \ref{l_nond} holds. To prove \eqref{hi} 
 we  consider two cases. 
\smallskip

\noindent 
{\it Case 1:  ${j_{\#}}=1$.}   Then $ \phi(a_1^r, a_j^r) = x_{\ell(a_j^r)}$. Denoting 
$\ 
 \frac{C_*^2}{\lambda^2_{a_1}}  \,  \frac{x^2_{\ell(a^r_j)}}{ \lambda^2_{a^r_j}} =: z_{\ell(a^r_j)}
$
we see that 
 $k_2(1)$ and $k_2(2)$  are linear functions of the variables 
 $z_{a_1},\dots, z_{\ell_n}$. 

i) Assume that $\chi^-(a_1^r, a_j^r)=1$  for some $r\in \{1,2\}$ and some $j>1$.  Denote $\ell(a_j^r)=a_{{j_*}}$.  Then 
${j_*}\ne j_{\#}$ and 
$$
k_2(r) =  \frac{z_{a_{{j_*}}}}{\mu(a_j^r) -\mu(a_1^r)} +\dots\,, 
$$
where $\dots$ is independent from $ z_{{j_*}} $.  Now let $r'= \{1,2\}\setminus \{r\}$, and find $j'$ such that 
$\ell (a^{r'}_{j'}) =a_{{j_*}}$.  If such $j'$ does not exist, then $k_2(r')$ does not depend on $z_{{j_*}}$. 
Accordingly, for a suitable  $x$ we have $k_2(r)\ne k_2(r')$, and \eqref{hi} holds. If $n_2=1$, then $r=1$ and $r'=2$. So
$j'$ 
does not exists and \eqref{hi} is established. 

If $j'$ exists, then  $n_1, n_2\ge2$, so the set $\A$ is\sa. 
By Lemma~\ref{l_triv}
 $\chi^-  (a^{r'}_{1} ,  a^{r'}_{j'}   ) =0 $  since $\chi^-(a_1^r, a_j^r)=1$ and 
\be\label{kkk}
|a_1^r| = |a_1^{r'}|, \qquad |a_j^r| = |a_{j'}^{r'}|\,.
\ee
 So
 $$
k_2(r') =  z_{{j_*}}\,\frac{ \chi^+  (a^{r'}_{1} ,  a^{r'}_{j'}   )  }{\mu(a_j^{r'}) +\mu(a_1^{r'})} +\dots\,.
$$
 Since $\chi^+$ equals 1 or 0, then using again \eqref{kkk} and the fact that $\mu(a)$ only depends on $|a|$, 
we see that $k_2(r)\ne k_2(r')$ for a suitable $x$, so  \eqref{hi}  again  holds. 
\smallskip

ii) If $\chi^-(a_1^r, a_j^r)=0$ for all $j$ and $r$,  then $\chi^+(a_1^r, a_j^r)=1$ for some $r$ and $j$. Define $z_{{j_*}}$ as above. 
Then the coefficient in $k_2(r)$ in 
 front of $z_{{j_*}}$ is non-zero, while for $k_2(r')$ it vanishes. This is obvious if $n_{r'}=1$. Otherwise $\A$ is\sa\ 
 and it holds   by Lemma~\ref{l_triv}
(and since $\chi^-\equiv0$). So \eqref{hi} again holds. 

\medskip

\noindent 
{\it Case 2:  ${j_{\#}}\neq 1$. Then by \eqref{achieve} there exists $a^1_j\in \L^r_f$ such that   $\ell(a^r_j)=a_1$. So 
$\chi^+(a ^1_1, a ^1_j)\ne 0$ or $\chi^-(a ^1_1, a ^1_j)\ne 0$. }
Then $ \phi(a_1 ^1, a_j ^1) = x_{a_{j_{\#}} }$, 
 the sum in \eqref{k2} is non-trivial and 
 for the same reason as in Case~1  \eqref{hi} holds.
 
   This completes the proof of   Lemma~\ref{l_nond}.

\section{KAM }\label{s_KAM}

\subsection{An abstract KAM result}\label{s_hakan}
 We first recall the abstract KAM theorem from \cite{EGK1}, adapting the
result and  the notation to the present context. Consider a Hamiltonian 
$H(r,\theta, u, v;\yy)$ of the form  \eqref{transf}.
Denote 
\be\label{rel1}
\L_\infty\cup \L^e_f = \tilde\L_\infty, \quad \L_f^h=  \tilde \L_f,\quad \nu\widehat K(\yy) = \bH(\yy)\,,
\ee
and re-denote (see\eqref{Lam})
\be\label{rel2}
\nu\Lambda_b(\yy) =: \Lambda_b(\yy)\,,\quad \la:=0
 \qquad\text{if} \quad b\in\L_f^\e \; \;\text{and}\;\;
  a\in\L_f=\L_f^e\cup \L_f^h\,.
\ee
Then the Hamiltonian reads 
\be\label{HHH}
H=
\Omega(\yy)\cdot r +\frac12 \sum_{a\in\tilde\L_\infty} \Lambda_a(\yy)(u_a^2+v_a^2)\\
+\frac{1}2 \langle \bH\tilde\zeta_f,  \tilde\zeta_f\rangle +  f(r,\theta, \tilde\zeta; \yy)\,,\quad
\tilde\zeta=(u,v)\,.
\ee
Assume that the parameter $\yy$ belongs to  a closed ball in $\R^d$ of a radius  at most one,  
which we denote $\D_0$. The Hamiltonian $H$  is regarded as a perturbation  of the quadratic  Hamiltonian
$$
h=\Omega(\yy)\cdot r+\frac12 \sum_{a\in\tilde\L_\infty} \Lambda_a(\yy)(u_a^2+ v_a^2) +
\frac12\, \langle \bH(\yy)\tilde\zeta_f,\tilde\zeta_f\rangle\,.
$$
We will 
assume that $h$ satisfies the following  assumptions A1 -- A3, depending on constants
\be\label{const}
 \delta_0, c,  \beta> 0,\;\; s_*\in\N\,,
\ee
where $c$ is such that the set $\A$ is contained in the ball $\{|a|\le c$.

For $a\in \tilde\L_f \cup\tilde\L_\infty\cup\{\emptyset\}$ we define 
\be\label{bla}
[a]=\begin{cases}
\tilde\L_f & \text{ if} \ \ a\in \tilde\L_f,\\
\L^e_f \cup\{ b\in\L_\infty: |b|\le c\} & \text{ if} \ \ a\in \L^e_f\;\;\text{or}\; a\in\L_\infty\;\;\text{and}\; |a|\le c\,,
\\
\{b\in \L_\infty\mid  |b|=|a|\} &  \text{ if} \ \ a\in \L_\infty \;\;\text{and}\; |a|> c\,,
\\
        \{\emptyset\} &  \text{ if} \ \  a=\emptyset\,.
        \end{cases} \qquad 
\ee

\smallskip

\noindent{\bf Hypothesis A1} (spectral asymptotic.) 
 For all $\yy\in\D_0$ we have \\
 
(a)\; $|\La|\ge \delta_0$ $\;\ \forall\, a\in\tilde \L_\infty$;\\

(b) \; $| \La-|a|^{2}|\le c \langle a\rangle^{-\beta}$ $\;\ \forall\, a\in\tilde \L_\infty$;\\

(c) \; $\|(J\bH (\yy))^{-1}\| \le \frac1{\delta_0}\,,\;\;\  \|(\La(\yy) I -iJ \bH(\yy))^{-1}\| \le \frac1{\delta_0}\;\;\  \forall\, a\in \tilde\L_\infty\,;
$\\

(d) 
 $|\Lambda_a(\rho) +\Lambda_b(\rho)| \ge \delta_0$ for all $a,b\in \tilde\L_\infty$;
  \\

(e)
 $|\Lambda_a(\rho) -\Lambda_b(\rho)| \ge \delta_0$ if $a,b\in \tilde\L_\infty$ and $[a]\ne[b]$.
\bigskip

\noindent
{\bf Hypothesis A2}  (transversality).
 For each $k\in  \Z^n\setminus \{0\}$ and every vector-function 
$\Omega'(\yy)$ such that $|\Omega'-\Omega|_{C^{s_*}(\D)}\le \delta_0$ there exists 
a unit vector $\zz=\zz(k)\in\R^n$, satisfying 
\be\label{o}
|\p_\zz \langle k, \Omega'(\rho)\rangle |\ge \delta_0\qquad \forall\, \rho\in\D_0. 
\ee
Besides the following properties (i)-(iii) hold for each $k\in  \Z^n\setminus \{0\}$:

(i)  For any $a,b \in\tilde\L_\infty\cup \{\emptyset\}$ such that $(a,b)\ne (\emptyset, \emptyset)$,
consider the following operator, acting on the space of $[a]\times[b]$-matrices \footnote{so if $b=\emptyset$, this is the space 
$\C^{[a]}$.}
$$
L(\yy): X\mapsto ({\Omega'}(\yy)\cdot k)X \pm Q(\yy)_{[a]} X
+ X Q(\rho)_{[b]}\,.
$$
Here $Q(\yy)_{[a]}$ is the diagonal matrix diag$\{\Lambda_{a'}(\rho) : a'\in[a]\}$, 
  and $Q(\yy)_{[\emptyset]}=0$. Then either 
\be\label{invert}
\|L(\rho)^{-1}\|\le \delta_0^{-1}\qquad \forall\, \rho\in\D_0\,,
\ee
or there exists a unit vector $\frak z$ such that 
$$
|\langle v,\p_\zz L(\rho)v\rangle| \ge\delta_0\qquad \forall\, \rho\in\D_0\,,
$$
for each vector $v$ of  unit length. 

(ii) Denote $m=2|\tilde\L_f|$ and consider the following operator in  $\C^m$, interpreted as a space of 
row-vectors:
$$
L(\yy,\lambda): X\mapsto ({\Omega'}(\yy)\cdot k)X +\lambda X+i XJ\bH(\yy)\,.
$$
 Then   
$$
\|L^{-1}(\yy,\Lambda_a)\| \le\delta_0^{-1}\qquad \forall\, \yy\in\D_0,\;\; a\in\tilde\L_\infty\,.
$$

(iii)   For any $a,b \in \tilde\L_f \cup \{\emptyset\}$ such that $(a,b)\ne (\emptyset, \emptyset)$,
consider the  operator, acting on the space of $[a]\times[b]$-matrices:
$$
L(\yy):  X\mapsto ({k\cdot\Omega'}( \yy))X - iJ\bH(\yy)_{[a]}X +    iXJ\bH(\yy)_{[b]}
$$
(the operator $\bH(\yy)_{[a]}$ equals $\bH$ if $a\in\tilde\L_f$ and equals 0 if $a=\emptyset$, and similar 
with $\bH(\yy)_{[b]}$). 
Then  the following alternative holds: either $L(\yy)$ satisfies \eqref{invert}, or there exists an integer $1\le j\le s_*$
such that 
\be\label{altern1}
|\p_\zz ^j\det L(\rho)| \ge \delta_0 \|L(\rho)\|_{C^j(\D_0)}^{\text{dim} -1}\qquad \forall\, \rho\in \D_0\,.
\ee
Here dim$\,= (\text{dim}\,\tilde\L_f)^2$
 if $a,b\in\tilde\L_f$ and dim$\,=\text{dim}\,\tilde\L_f$ if $a$ or $b$ is the empty set. 

\bigskip

\noindent
{\bf Hypothesis A3}  (a Melnikov condition). There exist $\tau>0$, ${\yy_*}\in\D_0$ 
and $C>0$ 
such that 
\be\label{melnikov}
|\langle k,\Om(\yy_*)\rangle -(\Lambda_a({\yy_*}) - \Lambda_b({\yy_*}))| \ge C|k|^{-\tau}\quad
\forall\, k\in\Z^n, k\ne0,\;\text{if}\;
 a,b\in \tilde\L_\infty \setminus[0]
 \ee
(cf. \eqref{bla}).

Recall that 
 the domains $\O^\ga(\s,\mu)$ and the classes ${\Tc}^{\ga,D}(\s,\mu,\D_0)$ were defined  at the beginning of Section~3. Denote
$$
\chi =  | \p_\yy\Omega(\yy)|_{C^{s_*-1}} 
+\sup_{a\in\tilde\L_\infty}   |\p_\yy\La(\yy)|_{C^{s_*-1}} +
\|\p_\yy \bH\|_{C^{s_*-1}} \,.
$$
Consider the perturbation $f(r,\theta, \zeta; \yy)$ and assume that
$$
\eps=[f^T]^{\ga,D}_{\s,\mu,\D_0}<\infty\,,\quad  \xi =[f]^{\ga,D}_{\s,\mu,\D_0}<\infty\,,
$$
for some $\gamma, \sigma, \mu\in(0,1]$.
 We are now in position to state the abstract KAM theorem from \cite{EGK1}. More precisely, the result below follows from 
  Corollary~3.7 of that work. 
 
\begin{theorem}\label{main} Assume that Hypotheses~A1-A3   hold for $\yy\in\D_0$. Then there exist
$\eps_0, \ka, \bar\b, C>0$ 
such that if for a suitable  $\aleph>0$  we have 
\be\label{epsest}
\chi, \xi\,= O(\delta_0^{1-\aleph})\quad \text{and}\quad
\eps  \le \eps_0 \delta_0^{1+\ka\aleph} =: \eps_*\,,
\ee
then 
there is  a Borel set
$\D'\subset \D_0$ with 
$\
\text{meas}(\D_0\setminus \D')\leq  C\eps^{\bar\b}$,  
and for all  $\yy\in \D'$ the  following holds: 
 
 \noindent 
 There exists a $C^{s_*}$-smooth mapping 
 $$
 {\mathfrak F} :\O^0(\s/2,\mu/2)\times\D'   
 \to \O^0(\s,\mu)\,,\quad (r,\theta,\tilde\zeta;\yy)\mapsto {\mathfrak F}_\yy(r,\theta,\tilde\zeta)\,,
 $$
 defining for $\rho\in\D'$ 
  real holomorphic symplectomorphisms
${\mathfrak F}_\yy :\O^0(\s/2,\mu/2)\to \O^0(\s,\mu)$, satisfying  for any $x\in \O^0(\s/2,\mu/2)$, $\rho\in \D'$ and $|j|\le1$ 
the estimates
\be\label{Hest}
\|  \p_\yy^j ({\mathfrak F}_\yy(x) -x) \|_{0} \le C  {\frac{\eps}{\eps_*}}\,, \qquad
\|  \p_\yy^j (  d{\mathfrak F}_\yy (x)- I) \|_{0,0} \le C  {\frac{\eps}{\eps_*}}\,, 
\ee
 such that 
\be\label{nf}
H\circ {\mathfrak F}_\yy=
\tilde \Om(\yy)\cdot r +\frac 1 2 \langle \zeta,  A(\yy)\zeta\rangle +g(r,\theta,\zeta;\yy)'.
\ee
Here 
\be\label{invari}
\partial_\zeta g=\partial_r g=\partial^2_{\zeta\zeta}g=0\; \;\text{for}\;\;
\zeta=r=0\,,
\ee
 $\tilde\Om=\tilde\Om(\yy)$ is a new frequency vector satisfying 
 \be\label{estimOM} 
 \|\tilde\Om-\Om\|_{\mathcal C^{s_*}}\leq C   \delta_0^{1+\aleph} \,,
 \ee
  and $ A:\L\times\L \to \M_{2\times 2}(\yy)$ 
  is an  infinite real symmetric matrix, belonging to 
  $\M_0^D$. It is of  the form $A=A_f\oplus A_\infty$, where 
  \be\label{k6}
\| \p_\yy^\alpha(
 A_f(\yy) - \bH(\yy) )\|\le C      \delta_0^{1+\aleph}  \,,\quad |\alpha|\le s_*\,.
\ee
The operator $A_\infty$ is such that 
  $(A_{\infty})_{ a}^b=0$ if $[a]\ne[b]$ (see \eqref{bla}),  and 
   all eigenvalues of the hamiltonian operator $JA_\infty$ are pure 
  imaginary. 
  
  The exponent $\bar\b$ and the constants $\eps_0$ and $\ka$ 
  depends only  on $\A$, $\tau$, $s_*$ and $\dim \tilde\L_f$
  (note that they  do not depend on the radius of the ball $\D_0$). 
  The constant  $C$ does not depend on $\delta_0, \aleph$, nor on the domain $\D_0$.
\end{theorem}

So for $\yy\in\D'$ the torus  ${\mathfrak F}_\yy \big(\{0\}\times\T^n\times\{0\}\big)$ 
is invariant for the hamiltonian system with 
the Hamiltonian $H(\cdot;\yy)$ given by \eqref{HHH}, and the hamiltonian 
 flow on this torus is conjugated by the map ${\mathfrak F}_\yy$ with the 
linear flow, defined by the Hamiltonian \eqref{nf} on the torus $(\{0\}\times\T^n\times\{0\})$.

Next  we show that Theorem \ref{main} applies to the Hamiltonian \eqref{transf}. To state the corresponding result we recall that 
in  Theorem  \ref{NFT}, assuming that
$
m\notin \Cc\,,
$
 we put the beam equation to the normal form \eqref{transf} when $\yy$ belongs to the closed 
 domain $\tilde Q(\nu)\subset\D=[c_*,1]^n$ and $0<\nu\le\nu_0$. 
 The domain $\tilde Q(\nu)$ was constructed in Section~\ref{s_4} as the union $\tilde Q(\nu) =  \cup_{j=1}^{\,\mathbb J} \tilde Q_j(\nu)$, where 
  $\mathbb J$ does not  depend on the small parameter $\nu$ and  any domain  $\tilde Q_j(\nu)$ lies in the corresponding connected component  $Q_j$ 
   of the set $\D\setminus\tilde X$. The 
   domains $\tilde Q_j(\nu)$  grow when $\nu$ decays and satisfy \eqref{mesmes}. 
   
   Let us define
\be \label{Dnu}
\D_\nu= \begin{cases}
 \tilde Q(\nu) & \text{if $\A$ is\sa },\\
 \tilde Q(\nu) \cap \D_0^1 & \text{if not}
\end{cases}
\ee
(see \eqref{DD}). 
We notice that $\D_\nu \subset \D_{\nu'}$ for $\nu \geq \nu'$.

\begin{theorem}\label{prop:6.3} There exists a zero-measure Borel set $\mathcal C\subset [1,2]$ 
such that for any admissible set  $\A$ and any  $m\notin \Cc$, may be found 
a real number $c_{**}\in(0,1/2]$, depending on $g(\cdot)$, $\A$ and $ m$,
such that for any $c_*\in(0, \tfrac12 c_{**}]$ there exist  $\beta_{*0},  \nu_0  \in(0,1]$,  depending  on
$g(\cdot)$, $\A$,  $ m$ and $c_*$,   with the following property:

\noindent
 For  any closed ball $\D_0\subset \D_{\nu_0} \subset\D=[c_*, 1]^n$ and any $\nu\le\nu_0, \beta_*\le\beta_{*0} $ 
 there exist a Borel set ${\tilde\D_{0\nu}}\subset \nu\D_0\subset \nu\D=[c_*\nu,\nu]^n$, 
 depending on $\nu, g(\cdot), c_*,\A, m, \D_0$ and
   satisfying 
\be\label{estimD}
\meas( \D_0\setminus {\tilde\D_{0\nu}})\leq C \nu^{{\bar\beta+n}} \,,\qquad \bar\beta>0\,,
\ee
  a $C^1$-mapping
$\ 
U:\ \T^n\times {\tilde\D_{0\nu}}\to Y_0^{\Z^d R}=Y_0^{\A \,R}\times Y_0^{\L\, R}   \,,
$ 
analytic in the first argument and satisfying 
\be\label{estimU}
\norma{U(\theta,I) -(\sqrt I e^{i\theta},\sqrt I e^{-i\theta},0)}_{0} \leq C \nu^{ 1- c_1 {\beta}_*}\,,
\ee
and a $C^1$-smooth vector function $\om': {\tilde\D_{0\nu}}\to \R^n$, satisfying 
\be\label{estimom'}\om'(I)=\om+M\, I+  O(|I|^{1+ c_2 {\beta}_*})\,,
\ee
where the matrix $M$ is given in \eqref{Om},
 such that
 
 i)  for $I\in\tilde\D_{0\nu}$ and $\theta\in\T^n$ the curve
\be\label{solution2}
t\mapsto U(\theta+t\om'(I),I)
\ee
is a solution of the beam equation \eqref{beam2}. Accordingly, for each $I\in\tilde\D_{0\nu}$ the analytic torus $U(\T^n\times\{ I\})$ is invariant for equation \eqref{beam2}.

ii) The  solution \eqref{solution2}  is linearly stable if and only if the operator $\widehat K(\yy)$ in the normal form 
\eqref{transf} is trivial (equivalently if  the hamiltonian operator  $iJK(\yy)$,  corresponding to the Hamiltonian \eqref{K}, is elliptic),
and this stability does not depend on $\yy\in\tilde\D_{0\nu}$. 
 In particular, this happens if in  \eqref{Dnu}   $|\A|=1$ or $d=1$, or if $\yy\in\tilde Q \cap \D_0^1$. 
 \\
 The constants $c_1$, $c_2$ and $C$ and the exponent ${\bar\beta}$ depend on $\A$, $m, c_*$ and $g(\cdot)$, but not on the ball  $\D_0\subset \D_{\nu_0}$.  
\end{theorem}

\subsection{Proof of Theorem \ref{prop:6.3}
}\label{s_assump}
In this section 
 we denote by  $C, C_1$ etc and $c, c_1$ etc  various constants, depending only 
on $g(\cdot), m, \A$  (not on $\D_0$ or $\nu$). \\ 
First we verify that the  Hypotheses~A1-A3 of Theorem~\ref{main} are satisfied uniformly for $\yy\in \D_\nu$ (see \eqref{Dnu}) and $\nu\le\nu_0$,
 where $\nu_0$ is defined in Theorem \ref{NFT}:
 
\begin{proposition}\label{p_KAM} 

There exists a zero-measure Borel set $\mathcal C\subset [1,2]$ 
such that for any admissible set  $\A$ and any  $m\notin \Cc$, may be found 
a real number $c_{**}\in(0,1/2]$, depending on $g(\cdot)$, $\A$ and $ m$,
such that for any $c_*\in(0, \tfrac12 c_{**}]$ there exists $\nu_0, \beta_{*0}   \in(0,1]$,  depending  on
$g(\cdot)$, $\A$,  $ m$ and $c_*$,   with the following property:

\noindent 
For  $0<\nu\le \nu_0$ and $0<\b_*\le \b_{*0}$ the  system \eqref{transf} with
the notation \eqref{rel1}, \eqref{rel2}   satisfies the Hypotheses~A1-A3
of Theorem~\ref{main} for all $\rho\in\D_\nu\subset[c_*,1]^n$, where 
\be\label{choice}
\delta_0=\nu^{1+{\beta_*\hren}}  \,, \quad
c=2 \max\{\langle a\rangle^3, a\in\A\},\quad \beta=2\,,\quad
s_*=4\, | \tilde\L_f|^2
\ee
with some 
\be\label{alpha}
 \hren>(\bar c+\b(0))\,.
\ee
The constant $\bar c$ and $\b(0)$ are defined in Theorem~\ref{NFT}, they depend only on $m$, $g(\cdot)$ and $\A$.
 \end{proposition}

\begin{proof}
We will check the validity of the three hypotheses. \\
First we note that using \eqref{Lam}, \eqref{estimla},  \eqref{N1}, \eqref{K4} and \eqref{delta} we get 
\be\label{basic1}
\tfrac12 +\tfrac12 |a|^2\le \La \le 2|a|^2 +1\,,\quad
|\La-\la|_{C^s(\D_0)} \le C_3 \nu |a|^{-2}\qquad \forall\, s\,,\;\forall\,
a\in\L_\infty\,,
\ee
\be\label{basic2}
C_1\nu^{1+\bar c \beta_*} \le |\La| \le C_2\nu\qquad \forall\, a\in\L^e_f
\ee
(we recall \eqref{rel2}). 
 The function  $\Omega(\yy)\in\R^n$ is defined in \eqref{Om}, so 
\be\label{ddet}
\Omega(\yy) = \omega +\nu M\yy,\qquad \det M\ne0\,,
\ee
and $\bH$ is a symmetric real 
linear operator in the space $Y^f$ (cf \eqref{Yf}).  Its norm satisfies
\be\label{basic3} 
\| {\bf H}(\rho)\| \le C \nu^{1-\beta_* \beta'(0)}\,,
\ee
see Theorem \ref{NFT}.ii).

\noindent
{\it Hypothesis~A1}. Relations 
 (a) and (b)  immediately follow from \eqref{basic1}, \eqref{basic2} and \eqref{alpha} 

  To prove (c) note that the operator $\wU$ conjugates   $J\bH$ and $(\Lambda_a I-iJ\bH)$ with 
 diagonal operators with the eigenvalues $\pm i\Lambda_j^h(\rho)$ and $\Lambda_a(\rho) \pm \Lambda_j^h(\rho)$, respectively. By  \eqref{basic2} and  \eqref{hyperb} 
 the norms of the 
eigenvalues are $\ge C^{-1} \nu^{1+\bar c\beta_*}$. Since the norms of  $\wU$ and its inverse are bounded 
by \eqref{Ubound}, then  the required estimate follows by \eqref{alpha}. 

Condition (e) follows from  \eqref{basic1}, \eqref{basic2} and \eqref{bla}. 
 
 Now consider (d).\footnote{This is the only condition of Theorem~\ref{main} which we cannot verify for any 
 domain $\tilde D_j$ without assuming that the set $\A$ is\sa.
 }
  If $a\in \L$ and $b\in\L_\infty$, then again the relation follows from  \eqref{basic1} and \eqref{basic2}.
    Next, let $a,b\in\L_f$. Let us write $\La$ and $\Lb$ as 
 $\Lambda^j_r$ and $\Lambda^k_m$, $j\le k$. If $j=k$, then the condition follows from \eqref{K04}, \eqref{delta}
 (from \eqref{K4} if $m=r$).
  If $j\le M_*<k$, then it again follows from \eqref{K04}. If $j,k\le M_*$,
 then $\Lambda^j_r=\Lambda^j_1=\mu(b_j,\yy)$  and $\Lambda^k_m=\mu(b_m,\yy)$, so the relation follows
 from \eqref{K44}. Finally, let $j,k> M_*$. Then if the set $\A$ is\sa,   the required 
 relation follows from \eqref{K04}, while if $\yy \in \D_0^1$, then it follows from \eqref{aaa}. 
\smallskip

\smallskip
\noindent
{\it Hypothesis~A2}. By \eqref{ddet}, $\p_\zz\Omega(\rho) = \nu M\zz$. Choosing 
\be\label{zet}
\zz= \frac{{}^t\!M k}{ |{}^t\!M k|}
\ee
and using that $|\Omega' -\Omega|_{C^{s_*}}\le \delta_0$
we achieve that  $\p_\zz\langle k, \Omega'(\rho)\rangle \ge C\nu$, so \eqref{o} holds.  

To verify (i) we restrict ourselves to the more complicated case when $a,b\ne\emptyset$. Then
$L(\rho)$ is a diagonal operator with the eigenvalues 
$$
\lambda_{a b}^k :=\langle k,\Omega'(\rho)\rangle +\Lambda_a(\rho) \pm \Lambda_b(\rho)\,\quad
a\in[a],\; b\in[b]\,.
$$
Clearly 
$$
|\lambda_{a b}^k -( \langle k, \omega\rangle +\lambda_a \pm \lambda_b)| \le C\nu |k|
$$
(we recall \eqref{rel2}). 
Therefore by Propositions \ref{D1D2} and \ref{prop-D3} the first alternative in (i) holds, unless 
\be\label{unless}
|k|\ge C \nu^{-c} 
\ee
for some (fixed) $c>0$. But if we choose $\zz$ as in \eqref{zet}, then $\p_\zz L(\rho)$ becomes a 
diagonal matrix with the diagonal elements  bigger than $|{}^tMk| - C\nu |k| - C_1\nu$. 
So if $k$ satisfies \eqref{unless}, then the second alternative in (i) holds. 
\medskip

To verify (ii) we write $L(\rho, \Lambda_a)$ as 
$$
L = (\langle k,\Omega'\rangle +\Lambda_a(\rho) )I +i\nu JH_0^h\,.
$$
 The transformation  $\wU$ conjugates  $L$ with  the diagonal 
 operator with the eigenvalues 
 $\lambda^k_{a j}=:  \langle k,\Omega'\rangle +\Lambda_a(\rho) \pm \nu i\Lambda^h_j$. 
 In view of  \eqref{hyperb}, 
 $ |\lambda^k_{a j}|\ge |\Im \lambda^k_{a j}|\ge C^{-1}   \nu^{ 1+\bar c\beta_*}$. This 
 implies (ii) by \eqref{alpha}. 
 \medskip

 It remains to verify (iii). As before,  we restrict ourselves to  the more complicated case $a,b\in\tilde\L_f$. Let us denote 
 $$
 \lambda(\rho) := \langle k,\Omega'(\rho)\rangle = \langle k,\omega\rangle + \nu\langle k, M \rho\rangle +
 \langle k,(\Omega' -\Omega)(\rho)\rangle\,,
 $$
 and write the operator   $L(\rho)$ as 
 $$
 L(\rho) = \lambda(\rho) I +L^0(\rho)\,,\quad L^0(\yy) X = [X, iJ{\bf H}(\yy)]\,.
 $$ 
  In view of \eqref{basic3}, 
 \be\label{9.0}
 \|L^0\|_{C^j} \le C_j \nu^{1- c_j\beta_*}\qquad\text{for}\; j\ge0\,.
 \ee
 Now it is easy to see that if
  $|\langle k,\omega| \rangle \ge C(\nu^{1-c_0\beta_*} +\nu|k|)$ with a sufficiently big $C$, 
 then the first alternative in (iii) holds. 
 
 So it remains
 to consider the case when
 \be\label{9.1}
 |\langle k,\omega \rangle | \le C(\nu^{1-c_0\beta_*} +\nu|k|)\,.
 \ee
 By Proposition \ref{D1D2} the l.h.s. is bigger than $\kappa|k|^{-n^2}$. Assuming that $\beta_{*0}\ll1$, we derive from this and 
   \eqref{9.1}  that 
   \be\label{9.2}
 |k|  \geq C \nu^{-1/(1+n^2)}\,.
 \ee
 In view of \eqref{9.0}-\eqref{9.2}, again  if $\beta_{*0}\ll1$, we have:
 \be\label{9.5}
 |\lambda(\rho) 
 | \le C\nu (\nu^{-c_0\beta_*} +|k|)\le C_1 \nu |k|\,,
 \ee
 \be\label{9.6}
 |(\p_\rho)^j \lambda(\rho)| \le C_j |k| \delta_0,\qquad 1\le j\le s_*\,,
 \ee
 \be\label{9.3}
 \|L\|_{C^j} \le C\nu (\nu^{-c_j\beta_*} + |k|)  +C_j|k|\delta_0\,.
  \ee
  
  Denote det$\,L(\rho) = D(\rho)$. Then 
  $$
  D(\rho) = \prod_{j,k\in \L_f^h} \prod_{\sigma_1, \sigma_2=\pm} 
  (\lambda(\rho) +\sigma_1 \nu \Lambda_j^h(\rho) -\sigma_2 \nu \Lambda_k^h(\rho))\,.
  $$
    Choosing $\zz$ as in 
  \eqref{zet} we get 
  $$
  \p_\zz\lambda(\yy) \ge C^{-1} |k| \nu - |k| \delta_0 \ge \frac12 C^{-1} |k|\nu\,.
  $$
  In view of \eqref{9.5}, \eqref{9.6} and \eqref{9.0} 
  this implies that 
  $$
  |\p_\zz\lambda | \gtrsim  |\lambda(\rho)|,\quad  |\p_\zz\lambda |\gg  |(\p_\rho)^j \lambda(\rho)|,\quad
   |\p_\zz\lambda |\gg |(\p_\rho)^j L^0|\,.
  $$
  Let us denote $2\,|\tilde\L_f|=m$; then $s_*=m^2$.  Chose in \eqref{altern1} $j=s_*=m^2$. 
  Then 
  $\p_\zz^{s_*} D(\rho)$ is a small perturbation of $(\p_\zz\lambda(\rho))^{m^2}$ since the latter is the leading term 
  of the former:  all other terms, forming $\p_\zz^{s_*} D(\rho)$, are much smaller. So
   we get that 
     $$
   | \p_\zz^{s_*} D(\rho)| \ge C_1^{-1} (|k|\nu)^{m^2}\,.
  $$
In the same time, in view of \eqref{9.3} 
the r.h.s. of \eqref{altern1}  is bounded from above by 
$$
C_m\delta_0 (\nu^{(m^2-1)(1-c_j \beta_*)} + \nu^{m^2-1} |k|^{m^2-1})\,.
$$
 This implies \eqref{altern1}, if  
 $$(|k|\nu)^{m^2}\nu^{-1-{\beta_*\hren}}\geq C_1C_m ((|k|\nu)^{m^2-1}+\nu^{(m^2-1)(1-c_j\b_*)})$$
 which is achieved as soon as $\beta_{*0}\leq\frac{m^2}{c_j(1+m^2)(m^2-1)}$ since $\hren$ is positive. 

 \medskip
 
\noindent
{\it Hypothesis~A3}.   The  required inequality follows from Proposition \ref{prop-D3} since 
the  divisor, corresponding to \eqref{melnikov}  where $a,b\not\in\L_f$, 
cannot be resonant.
\end{proof}

Now we will use 
 Proposition \ref{p_KAM}  to derive Theorem \ref{prop:6.3} from  Theorem \ref{main}.
 
 Let us take $\ga_*$, $\mu$, $\s =\frac12$ and $\hat c>0$ as in the Theorem~\ref{NFT} (see also \eqref{barc}), and take 
   $\hren,\beta_{*0}, \nu_0>0$ and $\delta_0 = \nu^{1+{\beta_*\hren}}$ 
  as in  Proposition~\ref{p_KAM}.  Since $\nu\le\nu_0$, then  $\D_0\subset\D_{\nu_0}\subset \D_\nu$ and 
   the proposition applies. So the Hypotheses~A1-A3 of  Theorem~\ref{main} are fulfilled, and to show that the theorem 
    is applicable to the Hamiltonian~\eqref{transf} with
   $\yy\in\D_0$ we have to verify that the quantities $\chi, \xi$ and $\eps$ meet the relation \eqref{epsest}. To do this let 
    us write  the estimates 
   \eqref{est111}  as 
   $$
  \xi:= [\tilde f]^{\gamma, D}_{\sigma, \mu, \D_0}\le C \nu^{1-\hat c\beta_*}\,,
\quad   \eps:=  [\tilde f^T]^{\gamma, D}_{\sigma, \mu, \D_0}\le C \nu^{3/2 - \hat c\beta_*} \,,
   $$
   and note that trivially  $\chi\le C\nu^{1-\beta_*\beta'(s_*)}$.
   This implies \eqref{epsest},  written as, 
$$
\chi  \le C \delta_0^{1-\aleph}\,,\quad   \xi \le C \delta_0^{1-\aleph}\,,\quad 
\quad  
\eps \le  \eps_0\delta_0^{1+\ka\aleph}=:\eps_*   \,,
$$
if $\nu_0$ is sufficiently small and 
\be\label{hhh}
\frac{{\hren}+\beta'(s_*) }{1+{\beta_*\hren}} \beta_*\leq \aleph\,,\quad 
 \frac{{\hren}+\hat c }{1+{\beta_*\hren}} \beta_* \leq \aleph\,,\quad
\aleph <\frac{1/2-{\beta_*\hren}-\hat c\b_*}{\ka(1+{\beta_*\hren})}\,.
\ee
Let us chose 
\be\label{aleph}
\hren = 2(\bar c+\beta(0))\,,\qquad \aleph = 4n\beta_* \big( \hren+\hat c+\beta'(s_*)+\b'(0) \big)=\b_*\tilde c\,.
\ee
Then \eqref{alpha} and \eqref{hhh} hold if $\b_{*0}$  is sufficiently small, and the relations \eqref{epsest}  are fulfilled. 
 
 Now we apply Theorem~\ref{main} to the Hamiltonian $\tilde H_\rho$ with $\rho\in\D_0$. We get 
   that there exist positive  constants $\bar \b$, $C_1$ 
  with the property that for $0<\b_*\leq\b_{*0}$ and $0<\nu\le\nu_0$ there exist 
  a Borel set ${\D'}\subset\D_0$,  satisfying
 $$
 \meas(\D_0\setminus {\D'})\leq C_1 \eps^{\bar\b}\leq  C_1 \nu^{(3/2-\hat c\b_*)\bar\b}\leq C\nu^{\bar \b}, 
 $$ 
 and a $C^{s_*}$-smooth mapping ${\mathfrak F} :\O^0(1/4,\mu/2)\times{\D'}
 \to \O^0(1/2,\mu)$  such that   for  $\rho\in {\D'}$ 
\be\label{nff}
H_\yy  \circ {\mathfrak F}_\yy=
 \om'(\yy)\cdot r +\frac 1 2 \langle \zeta,  A(\yy)\zeta\rangle +g(r,\theta,\zeta;\yy),
\ee
where $\partial_\zeta g=\partial_r g=\partial^2_{\zeta\zeta}g=0$ for $\zeta=r=0$.
As a consequence, the torus $ {\mathfrak F}_\yy(\{0\}\times\T^n\times \{0\}) $ is  invariant  for the Hamiltonian 
$H_\rho$, and the the mapping  $ {\mathfrak F}_\yy$ linearises the hamiltonian flow on this torus, i.e. 
solutions of the hamiltonian equation on the torus 
read   $ {\mathfrak F}_\yy(0,\theta_0+t\tilde\Om(\rho), 0)$, $\theta_0\in\T^n$. 
 So, setting ${\tilde\D_{\nu0}}:=  \nu{\D'}$, ${\tilde\D_{\nu0}} = \{I\}$, 
  we have
$\meas(\nu\D_0\setminus \tilde\D_{\nu0})\leq   C_1 \nu^{{\bar\b}+n}\,,$
and defining  the mapping $U$ as 
$$
U:\T^n\times {\tilde\D_{\nu0}}\to Y_0^{\Z^d R}\,,\quad 
U(\theta,I)=\tilde{\Phi}_{\rho}\circ {\mathfrak F}_\yy(0,\theta,0)\quad \text{with } \rho=\nu^{-1}I\,,
$$
where $\tilde{\Phi}_{\rho}$ is the mapping from Theorem \ref{NFT}, 
  we obtain that the curve \eqref{solution2} is a solution of \eqref{beam2}. As ${\mathfrak F}_\yy$ is close to the identity
   (see \eqref{Hest}), we deduce that 
 \be\label{estimdist'}
\dist ({\mathfrak F}_\yy(0,\cdot,0) (\T^n)  ,T^n_I)  \leq  C\frac{\eps}{\eps_*}\leq 
C\nu^{ \frac32-\hat c\b_*-(1+{\beta_*\hren})(1+2\ka\aleph)}\leq C\nu^{\frac12-c_\flat\b_*}
\ee
 with  $c_\flat=(2+3\ka)\tilde c$.
In particular, the torus $ {\mathfrak F}_\yy(0,\cdot,0) (\T^n) $ lies in the domain  $\Tg(\nu,1,1,0)$. In view of \eqref{invari}
it  is invariant for 
the beam equation  \eqref{beam2}. 
Recall that
$$
\tilde\Phi_\yy=\Phi_\yy\circ\Sigma\circ(\bU^{-1}\oplus\text{id})\,,
$$
where 
$\bU^{-1}\oplus\text{id}$ only moves the $\L_f$-variables and satisfies \eqref{LL}, $\Sigma$ is the change from the complex to the real variables and 
$\Phi_\rho$   satisfies \eqref{theRem}.  Combining  \eqref{theRem}, \eqref{LL} and \eqref{estimdist'}   we get that
$$
\norma{U(I,\theta)^\A-  (\sqrt I e^{i\theta}, \sqrt I e^{-i\theta})  }\leq C  \,\nu^{\frac12}\ \nu^{-\hat c\b_*}\ \nu^{ \frac12-\hat c_\flat\b_*}
= C  \, \nu^{ 1-(\hat c+c_\flat)\b_*}
$$
where  $U(I,\theta)=(U(I,\theta)^\A,U(I,\theta)^\L)\in   Y_0^{\Z^d R} = Y_0^{\A \,R} \times Y_0^{\L\, R}$, and $ (\sqrt I e^{i\theta}, \sqrt I e^{-i\theta}) $ 
stands for the vector $\big( (\xi_a, \eta_a), a\in\A\big), \xi_a\equiv \bar\eta_a$, as in \eqref{ac-an}. 

Similarly we verify using \eqref{theRem2}, \eqref{LL} and \eqref{Hest}   that 
$$\norma{U(I,\theta)^\L}_0\leq C  \, \nu^{ 1-(\hat c+c_\flat)\b_*}\,.$$
 This proves \eqref{estimU} with $c_1=\hat c+c_\flat$.
On the other hand \eqref{estimOM} leads to 
$$
\|\Om-\Om'\|_{\mathcal C^1}\leq C\nu^{(1+{\beta_*\hren})(1+\aleph)}\leq C  \nu^{1+ c_2 {\beta}_*}$$
with $c_2=\tilde c+2(\bar c+\b(0))$. Thus defining $\om'(I)=\Om'(\nu^{-1}I)$ and using \eqref{Om} we get 
$$|\om'(I)-\om-M\, I|\leq C \nu^{1+ c_2 {\beta}_*}.$$

  Finally,  since  in  Theorem~\ref{main} the infinite real symmetric matrix
   $ A$   is  of  the form $A=A_f\oplus A_\infty$, where 
  $$ 
  A_f(\yy) = \nu \widehat K(\yy)+O(\delta_0^{1+\aleph})= \nu \widehat K(\yy)+O(  \nu^{(1+{\beta_*\hren})(1+\aleph)} )
  $$
  and    $A_\infty$ is a  block--diagonal   matrix such that 
   all eigenvalues of the hamiltonian operator $JA_\infty$ are pure 
  imaginary (see \eqref{k6}), 
  then  the linear stability of  the  constructed invariant   torus $U(\T^n\times\{I\})$ is determined by
    the stability of the matrix $A_f(\yy)$,  $\yy=\nu^{-1} I$.
   If the set $\L_f^h$ is not trivial, then by \eqref{hyperb} the operator $J\widehat K$ admits an eigenvalue whose  real part is 
    larger than $ \nu^{\bar c \b_*}$. 
   Let us  denote 
   $\ 
    \nu^{\beta_*\b'(0)} J \widehat K = L_1$, $ \nu^{-1+\beta_*\b'_0} J A_f = L_2\,
   $ ($\b'(0)$ is defined in \eqref{normK}).
   Then
   \smallskip
   
   \noindent
   i) $\| L_1\|\le C$ by \eqref{normK},
   
   \noindent 
   ii) $L_1$ has an eigenvalues whose real part is $\ge \nu^{\beta_*(\bar c+ \b'(0))}$, 
   
   \noindent 
   iii) $\|L_1 - L_2\| \le C\nu^{-1+\beta_* \b'(0)} \delta_0^{1+\aleph} \le C\nu^\aleph$. 
   \smallskip
   
   \noindent
   By i), iii) and Lemma \ref{l_cart} 
    the distance between the spectra of the operators $L_1$ and $L_2$ is 
    bounded by  $ C\nu^{\aleph/2n}\leq C \nu^{2\b^*(\hat c+\b'(0))}$ 
     where we used \eqref{aleph}. Then in view of ii)
     the operator     $L_2$ has an eigenvalue with a nontrivial real part. Accordingly, the hamiltonian operator $JA_f$ is unstable. 
  \endproof
 
\subsection{Proofs  of Theorems \ref{t73} and \ref{t72}
}\label{sKAM}

Everywhere in this section ``ball" means ``closed ball".
Let us fix $\b_*>0$ small enough so that Theorem \ref{prop:6.3} applies on any ball $\D_0\subset \D_\nu$ (see \eqref{Dnu}) with $\nu\leq \nu_0$. \\
We start with a construction which allows to apply Theorem \ref{prop:6.3} to prove the two theorems. For any $\ga>0$ 
let us find $\nu'\equiv \nu'(\b_*)>0 $ so small that  
\be\label{D1000}
\sum_{j=1}^{\bJ} 
\meas(Q_j  \setminus\Dj(\nu'))\le \frac14\ga \sum\meas  Q_j = \frac 14 \ga \meas\D
\ee
(see \eqref{components} and 
 \eqref{K1meas}).  Since each $\tilde Q_\nu$ is a component of a closed semi-analytic set, 
 then its interior Int$\,\tilde Q_\nu$ has the same measure as the set itself. 
By the Vitali theorem we can find a countable family of non-intersecting balls in Int$\,\tilde Q_\nu$
such that 
their union feel up  this domain  up to a zero-measure set. Therefore in $\Dj(\nu')$ exist $N_j=N_j(\gamma)$ 
 non-intersecting balls $B^1_j,\dots, B_j^{N_j}$, $B^r_j=B^r_j(\ga)$,  such that 
\be\label{D10}
\meas \big(\Dj(\nu')\setminus
\cup_{r=1}^{N_j}B_j^r\big)\le \frac1{4\bJ}\ga \meas(Q_j)\,,\qquad j=1,\dots,\mathbb J\,.
\ee
Note that $\cup_{r=1}^{N_j}B_j^r \subset \tilde Q_j(\nu) $  if $\nu\le \nu'$. 

Now to each ball $B^r_j$ and every  $\nu\le\nu'$ we apply Theorem \ref{prop:6.3} to construct a set
$(B_j^r)'(\nu)$, corresponding to the tori, persisting in the perturbed equation, and for each $j$ 
find 
$\nu_j\in(0,\nu']$ such that 
\be\label{D100}
\meas \big( 
\cup_{r=1}^{N_j}B_j^r(\ga)\setminus  \cup_{r=1}^{N_j}(B_j^r)'(\ga,\nu)
  \big)\le \frac1{4\bJ}\ga \meas(Q_j) \quad \text{if}\quad \nu\le\nu_j\,.
\ee
 This $\nu_j$ depends on  $N_j$.

\medskip
\noindent 
{\it Proof of  Theorem \ref{t73}.}  Let us consider the domains $Q_l\subset\D_0^1$. They correspond to $l\le\bJ_1$, 
see \eqref{newJ}, and  feel in  $\D^1_0$ up to a set of zero measure. 
 Since $c_*\le\tfrac12 c_{**}\le 1/4$, then 
\be\label{D11}
\D^1_0\subset \{\yy \mid\tfrac12\le\ \|\yy\| \le2\}\,,\qquad \meas \D^1_0\ge \tfrac12 c_{**}^n>0\,.
\ee
 Next in the construction above we choose $\ga = 1/2$, find 
the corresponding $\nu_0=\min(\nu_1,\cdots,\nu_{J_1})>0$ and for $\nu\le\nu_0$
construct 
 the sets
$(B^l_r)'(1/2, \nu)$, $l\le \bJ_1, r\le N_l$. 
Denote by 
$\B(\nu)$ their union,  and denote $\D_\nu = \nu \B (\nu) \subset [0,\nu]^n$. Now we set
$$
{\fJ} = \cup_{j=0}^\infty \D_{ \nu^{(j)}}\,,\qquad \nu^{(j)}=5^{-j}\nu_0\,.
$$
In view of  \eqref{D11}, 
  ${\fJ}$ is a disjoint union of  Borel sets. Since by \eqref{D1000}-\eqref{D100} 
 $\meas(\D^1_0\setminus\B(\nu^{(j)})\le \tfrac38 \meas \D^1_0$ and ${\fJ}\cap [0,\nu^{(j)}]^n\supset \D_{\nu^{(j)}}=\nu^{(j)} \B (\nu^{(j)})$,
  we see that the $\liminf$ in \eqref{posdens}
is $\ge\,\frac 38\meas \D^1_0$. So  ${\fJ}$ has a positive density at the origin.

 Now we define the mapping 
$\ 
U : \T^n \times {\fJ} \to Y^R
$
by the relation
$$
U\mid_{\nu^{(j)} (B^r_1)'(1/2, \nu^{(j)})} = U^{\nu^{(j)},  (B^r_1)'(1/2, \nu^{(j)})} \qquad \forall\, r\le N_1,\; \; j\ge0\,,
$$
where $ U^{\nu^{(j)}, (B^r_1)'(1/2, \nu^{(j)})} $ is the mapping from Theorem \ref{prop:6.3}, corresponding to $\nu=\nu^{(j)}$
and $\D_0 = B^r_1$, and define the mapping $\omega':  {\fJ} \to \R^n$ similarly. 

These maps obviously are continuous. 
Since for any vector $I\in \D_{\nu^{(j)}}\subset  {\fJ}$ the norm of $I$ is equivalent to $\nu^{(j)}$,
then by Theorem \ref{prop:6.3} the maps  satisfy all assertions of Theorem~\ref{t73}, apart from those 
related to the triviality of the hyperbolic operator $JK$.  If $d=1$ or $|\A|=1$, then $JK=0$ by 
 Examples~\ref{Ex41} and \ref{n=1}. Examples of non-trivial operators $JK$ (when $d\ge2$ and
$|\A|\ge2$) are given in Appendix~B. So Theorem~\ref{t73} is proved. 
\medskip

\noindent
{\it Proof of  Theorem \ref{t72}.} Since now  the set $\A$ is\sa, then 
 Theorem~\ref{prop:6.3} applies to 
every ball in every domain
$\tilde Q_j(\nu)$, $j\le\mathbb J$. For any $\ga>0$ let us chose $c_*=c_*(\ga)$ such that 
$
\meas ([0,1]^n\setminus \D)\le \tfrac14\ga.
$
 Next for each  $j\le\mathbb J$ we find $\nu_j=\nu_j(\ga)$ and the 
collection of balls $\{B_j^r(\ga)  , r\le N_j \}$  and sets $\{(B_j^r)'(\ga, \nu)  , r\le N_j , \nu\le\nu_j\}$ 
as in \eqref{D10},  \eqref{D100}.  Denote
\be\label{D12}
\nu(\ga) = \min\{\nu_j(\ga), j\le\mathbb J\}\,,\qquad
 \B(\ga,\nu)= \cup_{j=1}^{\,\mathbb J} \cup_{r=1}^{N_j} (B_j^r)'(\ga, \nu)\,,\quad 0<\nu\le \nu(\ga)\,.
\ee
Note that $\nu: (0,1] \to (0,1],\; \ga\mapsto \nu(\ga)$, is a non-increasing function which goes to zero
with $\ga$.  From \eqref{D1000},
\eqref{D10}, \eqref{D100} we have 
\be\label{D13}
\meas( [0,1]^n\setminus\B(\ga,\nu)) 
\le \frac14\ga +\sum_{j=1}^{\bJ} \meas(\tilde Q_j\setminus\cup_r (B_j^r)' (\ga,\nu)  )
\le \frac34\ga\,,
\ee
for any $\nu\le \nu(\ga)$.

Let $\nu_k=2^{-k},\; k\ge1$, and let $\{\ga_k\}$ be a non-increasing sequence of positive numbers,
converging to zero so slowly that 
$\nu(\ga_k) \ge\nu_k$.  For $k\ge1$ denote
$$
K_k = [0,\nu_k]^n\,,\qquad \Gamma_k=\p K_{k+1}  \cap (0,\nu_k)^n
$$
($\Gamma_k$ lies in the interior  of $K_k$). 
Let $O_k$ be an $\eps$-vicinity of $\Gamma_k$ in $K_k$ ($\eps>0$) so small that 
$\ 
\meas O_k\le \tfrac14 \ga_k\nu_k^n.
$ 
For every $k$ let
$\ 
{\B_k} = \nu_k \B(\ga_k, \nu_k)\subset K_k\,. 
$
Then 
\be\label{kkkk}
\meas(K_k \setminus  {\B_k}) \le \frac34  \ga_k \nu_k^n\,.
\ee

Finally, we set
$$
\fJ(m, \A) = \cup_{k=1} ^\infty \big( {\B_k}\setminus (O_k\cup K_{k+1}))\,.
$$
This is a disjoint union of Borel sets, and we  derive from \eqref{kkkk} that $\fJ(m, \A) $ has   density one at the origin. 

To construct the mappings $U:\T^n\times \fJ(m, \A) \to Y^R$ and $\omega': \fJ(m, \A) \to \R^n$, 
for each $k\ge1$ we
define them on the sets  $\nu_k (B^r_j)'(\ga_k, \nu_k)$, forming the set ${\B_k}$, using Theorem \ref{prop:6.3}.
We do this 
exactly as above, when proving Theorem~\ref{t73}. Next we restrict the maps to the sets 
${\B_k}\setminus (O_k\cup K_{k+1})$,  forming  $\fJ(m, \A)$. Our construction implies that the
mappings are continuous.  The  estimates \eqref{dist1} and \eqref{dist11} with suitable constants 
$C,c$ follow from Theorem~\ref{prop:6.3}, if we note that 
for $I\in K_k\setminus K_{k+1}$ the norm of $I$ is equivalent to $\nu_k$, provided that 
 $\gamma_k\to0$ sufficiently slow.

The analysis of the linear stability of the constructed solutions is the same as before. This proves Theorem~\ref{t72}.

\begin{remark}\label{lastrem}
Let $\A\subset\Z^d$ be an admissible set. 
The Hypothesis~A1(d)   is the only assumption of Theorem \ref{main} which we cannot verify for  the Hamiltonian \eqref{transham}
and all domains  $\tilde Q_j$  if $d\ge3$.  Accordingly, if under the assumptions of Theorem~\ref{t73} 
      the Hypothesis~A1(d)     holds for  all     $\yy\in  \tilde Q_j$ and  every $j$, 
    then the  assertion of Theorem~\ref{t72} is true  for this $\A$.   
       Similar, if the Hypothesis~A1(d) holds for some domain $ \tilde Q_j$, 
    then the assertion of Theorem \ref{prop:6.3} is valid for any ball $\D_0\subset\tilde Q_j$. 
    
         If for $\yy\in Q_j$ the operator $\widehat K$ is non-trivial  and the assertion of Theorem~\ref{prop:6.3}
         holds for balls $\D_0\subset \in Q_j$ (e.g. the set $\A$ is\sa), 
         then the constructed KAM-solutions
         are linearly unstable. 
         In this case the set $\fJ$ in Theorem~\ref{t73}     contains
     a subset $\fJ_h$ (corresponding to the scaling of the component $\tilde Q_j$), having positive density
     at the origin, filled in with linearly unstable KAM-solutions. 
         
    \end{remark}

\section{Conclusions.}
 The set ${\frak A}(d,\A) = U(\T^n\times\fJ)$, where $\A\subset\Z^d$ is an admissible set and $U$ and $\fJ$ are 
 constructed in Theorems~\ref{t73}, \ref{t72}, is invariant for the beam equation (written in the form \eqref{beam2}), 
 and is filled in with its time-quasiperiodic solutions. The assertion ii) of Theorem~\ref{t73} implies that the Hausdorff
 dimension of this set equals $2n$.     Now let 
 $$
  {\frak A} = 
   \cup_{\A\subset \Z^k} 
    {\frak A} (d,\A)\,.
 $$
 This set  is formed by  time-quasiperiodic solutions of \eqref{beam} and has infinite Hausdorff dimension. For $d=1$ it is linearly stable.
 But  for $d\ge2$ some solutions, forming the set (e.g. those, corresponding to $|\A|=1$) are linearly stable, while 
 in view of the examples in Appendix~B some others  with $|\A|\ge2$ are linearly unstable.

 For $d\ge2$ the unstable parts of the sets  ${\frak A}$  creates around them 
  some local instabilities. It is unclear for us 
  wether these instabilities have anything to do with
 the phenomenon  of the energy cascade to high frequencies, predicted by the theory of wave
 turbulence for small-amplitude solutions  of space-multidimensional hamiltonian PDEs. The linear instability 
 of solutions and the energy cascade to high frequencies on various time-scales are now topics 
 of major interest for the nonlinear PDE community, e.g. see  in \cite{Tao}.

We note that the fact that KAM-solutions of high dimensional PDEs may be linearly unstable is not 
new: in \cite{GY3} the instability of some KAM-solutions for the 2d cubic NLS equation was
observed (see there Remark~1.1), while in \cite{PP1, PP2} algebraic reasons for  the 
instability of KAM-solutions for  multidimensional  NLS equations were discussed. 
\medskip

Our study of the beam equation \eqref{beam}  leads to several natural questions. 
One is to find  a sufficient condition for an admissible set $\A\subset\Z^d$, such that $d, |\A|\ge2$, to guarantee that 
 the hamiltonian operator $J\widehat K(\yy)$ in Theorem~\ref{NFTl} is non-trivial  for $\yy$ in some component 
$\tilde Q_l$ of the set $\tilde Q$ (we recall that for some  components of $\tilde Q$ it always is trivial). Cf.
Remark~\ref{r_1}.4). 

 If  for some $\A$ this property is fulfilled and the assertion of 
Theorem~\ref{t72} holds (e.g. the set $\A$ is\sa), then by  Remark~\ref{lastrem}
the set $\fJ$ has a subset $\fJ_h$, having positive density at the origin, such that for $\yy\in\fJ_h$ 
the corresponding  KAM-solutions of eq.~\eqref{beam}  are 
linearly unstable. 

Another question is to 
study the persistence of small-amplitude linear solutions \eqref{sol} in the beam equation \eqref{beam} 
 for the case when the set $\A$
is not admissible. 

A third question concerns the role of the Hypothesis A1(d) in  Section \ref{s_hakan}.
In the notation of that section,  do the majority of the 
 invariant tori $\T^n\times\{0\}\times\{0\}$ of the Hamiltonian $h$ persist as  invariant
 tori for the Hamiltonian $H$, if the condition A1(d) is violated and $\La+\Lb\equiv0$
 for some $a,b\in\tilde\L_\infty$? 
 
 We recall that the condition A1(d) is the only one which 
 we can check for\sa\  sets $\A$, but not for admissible. 

\smallskip

\appendix

\section{Proof of Lemma \ref{lemP} }
For any $\gamma\ge0$ let us denote by $Z_\gamma$ the space of complex sequences $v=(v_s, s\in\Z^d)$
with finite norm $\|v\|_\gamma$, defined by the same relation as the norm in the space $Y_\gamma$. 
For $v\in Z_\gamma$ we will denote by $\F(v)= u(x)$ the Fourier-transform of $v$, 
 $u(x)=\sum v_s e^{is\cdot x}$. By Example~\ref{analyt} if $u(x)$ is a bounded real holomorphic 
function in $\T^n_{\sigma'}$, then $\F^{-1} u\in Z_\sigma$ for $\sigma<\sigma'$.

Let $F$ be the Fourier-image of the nonlinearity $g$, i.e. 
$\ 
F(v) = \F^{-1} g(x, \F(v)(x)).
$

\begin{lemma}\label{l1}
For sufficiently small $\mu_*>0, \ga_*>0$ and for all $0\le\ga\le\ga_*$,

i) $F$ defines a real holomorphic  mapping $\O_{\mu_*}(Z_{\ga}) \to Z_\ga$, 

ii) $\nabla F$ defines a real holomorphic mapping $\O_{\mu_*}(Z_{\ga}) \to M_\ga$, where $M_\ga$ is the space of matrices
$A:\Z^d\times \Z^d \to \C$, satisfying
$
|A|_\ga :=\sup |A_a^b| \,e^{\ga |a-b|}<\infty\,.
$
\end{lemma}

\begin{proof}
i) For sufficiently small $\sigma', \mu>0$ the 
 nonlinearity $g$ defines a real holomorphic function  $g:\T^d_{\sigma'}\times \O_\mu (\C)\to \C$ and the norm
 of this function  is bounded by some constant $M$. We may write it as
$\ 
g(x,u) = \sum_{r=3}^\infty g_r(x) u^r\,, 
$
where $g_r(x) = \frac1{r!} \frac{\p^r}{\p u^r}g(x,u)\!\mid_{u=0}$. So $g_r(x)$ is holomorphic in $x\in\T^d_{\sigma'}$ 
and by the Cauchy estimate $|g_r|\le M\mu^{-r}$. So 
$$
\|\F^{-1} g_r\|_\ga\le C_{\sigma} M\mu^{-r}\quad \forall\, 0\le\ga\le \sigma\,,
$$
for any $\sigma<\sigma'$.
Cf.  Example~\ref{analyt}.
We may write $F(v)$ as
\be\label{b1}
F(v) = \sum_{r=3}^\infty (\F^{-1} g_r)\star \underbrace {v\star\dots\star v}_{r}
\,.
\ee
Since the space $Z_\ga$ is an algebra with respect to the convolution (see Lemma~1.1 in \cite{EK08}), 
the $r$-th term of the sum is bounded as follows:
\be\label{b2}
\| (\F^{-1} g_r)\star \underbrace {v\star\dots\star v}_{r}\|_\ga\le C_1
 C^{r+1} \mu^{-r} \|v\|^r_\ga\,.
\ee
This implies the assertion with $\ga_*=\sigma $ and a suitable $\mu_*>0$.
\medskip

ii) For $r\ge3$ consider the $r$-th term in the sum for $g(x,u(x))$ and denote by $G_r$ its
Fourier-image, $G_r(v) = \F^{-1} (g_r u^r)$, $u=\F(v)$. Then
$$
(\nabla G_r(v))_a^b = r (2\pi)^{-d} \int e^{-ia\cdot x} g_r(x) u^{r-1} e^{ib\cdot x}\, dx\,. 
$$
Applying \eqref{b2} (with $r$ convolutions instead of $r+1$) we see that 
\be\label{b3}
|(\nabla G_r(v))_a^b| \le C_2
C^{r} \mu^{-r} \|v\|_\ga^{r-1}\langle b-a\rangle^{-d^*}
e^{-\ga|b-a|}\,.
\ee
So $|\nabla G_r(v)|_\ga \le C^{r} \mu^{-r} \|v\|_\ga^{r-1}$, which implies the second assertion of the lemma.
\end{proof}

\noindent
\begin{proof}[ Proof of  Lemma \ref{lemP}.]  Let us consider the functional $P(\zeta)$ as  in \eqref{H1}, 
and write it as
$\ 
P(\zeta) = p\circ \Upsilon\circ D^{-1}\zeta\,. 
$
Here $D$ is the operator, defined in Section~\ref{s3.1},  $\Upsilon$ is the bounded operator 
$$
\Upsilon: Y_\ga\to Z_\ga,\qquad \zeta\to v,\;\; v_s=\frac{\xi_s+\eta_{-s}}{\sqrt2}\;\;\forall\, s,
$$
and $p(v) = \int G(x,(\F^{-1}v)(x))\,dx$. Lemma \ref{l1} with $g$ replaced by $G$ immediately implies that 
$P$ is a real holomorphic function on $\O_{\mu_*}(Y_{\ga_*})$ with suitable $\mu_*, \ga_*>0$.

 Next, since 
$$
\nabla P(\zeta) =D^{-1}\circ{}^t\Upsilon \circ \nabla p(\Upsilon\circ D^{-1}\zeta)\,,
$$
where $\nabla P=F$ is the map in Lemma \ref{l1}, then $\nabla P$ defines a real holomorphic mapping 
$\O_{\mu_*}  (Y_{\ga_*}) \to Y_{\ga_*}$. \\
Further,
$$
\nabla^2 P(\zeta) =D^{-1}({}^t\Upsilon \ \nabla^2 p(\Upsilon\circ D^{-1}\zeta)\ \Upsilon)D^{-1}\,.
$$
Since for any $A\in M_\ga$ the matrix ${}^t\Upsilon A \Upsilon$
is given by the relation 
$$
({}^t\Upsilon A \Upsilon)^b_a =\frac12 \sum_{a'=\pm a, \,b'=\pm b}
 A^{b'}_{a'}\,,
$$
then
$
|D^{-1}({}^t\Upsilon A \Upsilon )D^{-1}|^D_\ga \le 2| A|_\ga.
$
So
$$
|\nabla^2 P(\zeta)|^D_\ga \le 2 |\nabla^2p(\zeta)|_\ga = 2 | \nabla F(\zeta)|_\ga\,,
$$
and in view of item ii) of Lemma \ref{l1}, the mapping 
$$
\nabla^2_\ga P:\O_{\mu_*}(Y_\ga) \to \M^D_\ga,\qquad 0\le\ga\le\ga_*\,,
$$
is real holomorphic and  bounded in norm by a $\ga$-independent constant. 
 \end{proof}

\section{Examples}
In this appendix we discuss some examples of  hamiltonian operators $\H(\yy)= iJK(\yy)$ defined in \eqref{diag}, corresponding to 
various dimensions $d$ and   sets $\A$. In particular we are interested in examples which give rise to partially hyperbolic KAM solutions. 

\smallskip

\noindent{\bf Examples with $({\L_f}\times {\L_f})_+=\emptyset$.}\\
As we noticed in \eqref{Lf+=0}, if $({\L_f}\times {\L_f})_+=\emptyset$ then $\H$ is Hermitian, so the constructed 
 KAM-solutions  are linearly stable. This is always the case when $d=1$.\\
When $d=2$ and $\A=\{(k,0),(0,\ell)\}$ with the additional assumption that neither $k^2$ nor $\ell^2$ can be written as 
the sum of squares of two natural numbers,  we also have $({\L_f}\times {\L_f})_+=\emptyset$.\\
Similar examples can be  constructed in higher dimension, for instance for $d=3$ we can take 
$\A=\{(1,0,0),(0,2,0)\}$ or $\A=\{(1,0,0),(0,2,0),(0,0,3)\}$.\\
We note that in \cite{GY2}  the authors perturb solutions \eqref{sol}, corresponding to 
 set $\A$ for which $({\L_f}\times {\L_f})_+=\emptyset$ 
and $({\L_f}\times {\L_f})_-=\emptyset$. This significantly simplifies 
 the analysis since in that case there is no matrix $K$ in the mormal form
  \eqref{HNF} and  the unperturbed quadratic Hamiltonian is diagonal.

\smallskip

\noindent{\bf Examples with $({\L_f}\times {\L_f})_+\neq\emptyset$.}
In this case hyperbolic directions may appear as we show below.\\
The choice $\A=\{(j,k),(0,-k)\}$ leads to $((j,-k),(0,k))\in ({\L_f}\times {\L_f})_+$.\\
Note that this example can be  plunged in higher dimensions, e.g. the 3d-set 
$\A=\{(j,k,0),(0,-k,0)\}$ leads to a non trivial $({\L_f}\times {\L_f})_+$.

\smallskip

\noindent{\bf Examples with hyperbolic directions}\\
Here we  give  examples of normal forms with hyperbolic eigenvalues, first in  
  dimension two, then -- in higher dimensions.  That is, for the beam  equation \eqref{beam} we will find 
   admissible sets $\A$ such that the corresponding matrices $iJK(\yy)$ in the normal form \eqref{HNF} have
 unstable directions. Then by  Theorem~\ref{t73}
   the  time-quasiperiodic solutions of \eqref{beam},  constructed in the theorem,  are linearly unstable.

We begin with dimension $d=2$. Let 
$$
\A=\{(0,1),(1,-1)\}\,.
$$
We easily compute using \eqref{L++}, \eqref{L+-} that 
$$\L_f=\big\{ (0,-1),(1,0),(-1,0),(1,1), (-1,1),(-1,-1)\big)\}\,,
$$
and
$$
( {\L_f}\times {\L_f})_+=\{\big( (0,-1),(1,1)\big);\big( (1,1),(0,-1)\big)\}, \qquad
( {\L_f}\times {\L_f})_-=\emptyset.
$$
So in this case the decomposition  \eqref{decomp} of the hamiltonian operator $\H(\yy)= iJK(\yy)$ reads
$$
\H(\yy)=\H_1(\yy)\oplus\H_1(\yy)\oplus\H_3(\yy)\oplus\H_4(\yy)\oplus\H_5(\yy)\,,
$$
where $\H_1(\yy)\oplus\H_1(\yy)\oplus\H_3(\yy)\oplus\H_4(\yy)$ is a diagonal operator with purely imaginary eigenvalues and  $\H_5(\yy)$ is
an operator in $\C^4$ which may have  hyperbolic eigenvalues. That is, now $M=5$ and $M_*=4$.
\\
  Let us denote $\zeta_1=(\xi_1,\eta_1)$ (reps. $\zeta_2=(\xi_2,\eta_2)$) the $(\xi,\eta)$-variables corresponding to the mode $(0,-1)$ (reps. $(1,1)$). We also denote $\yy_1=\yy_{(1,0)}$, $\yy_2=\yy_{(1,-1)}$, $\lambda_1=\sqrt{1+m}$ and $\lambda_2=\sqrt{4+m}$. By construction $\H_5(\yy)$ is the restriction of the Hamiltonian $  \langle K(m,\yy)\zeta_f, \zeta_f\rangle$
to the modes $(\xi_1,\eta_1)$ and $(\xi_2,\eta_2)$. 
We   calculate using \eqref{K}
that 
\be\label{hr}\langle \H_5(\yy)(\zeta_1,\zeta_2),(\zeta_1,\zeta_2)\rangle=\beta(\yy) \xi_1\eta_1+\ga(\yy) \xi_2\eta_2+\alpha(\yy) (\eta_1\eta_2+\xi_1\xi_2)\,,
\ee
where
$$
\alpha(\yy)= \frac6{4\pi^{2}} \frac{\sqrt{\yy_1\yy_2}}{\lambda_1\lambda_2}\,,\quad
\beta(\yy)= \frac3{4\pi^{2}}\frac1{\lambda_1}\Big( \frac{\yy_1}{\lambda_1}-\frac{2\yy_2}{\lambda_2} \Big) \,,\quad
\ga(\yy)= \frac3{4\pi^{2}} \frac1{\lambda_2}\Big( \frac{\yy_2}{\lambda_2} -\frac{2\yy_1}{\lambda_1}\Big) \,.
$$
Thus the linear hamiltonian 
system, governing the two modes, reads\footnote{Recall that the symplectic two-form is: $-i\sum d\xi\wedge d\eta$.}
\ben \left\{\begin{array}{ll}
 \dot \xi_1 &=-i(\beta \xi_1+\alpha \eta_2)\\
 \dot \eta_1 &=i(\beta \eta_1+\alpha \xi_2)\\
 \dot \xi_2 &=-i(\ga \xi_2+\alpha \eta_1)\\
 \dot \eta_2 &=i(\ga \eta_2+\alpha \xi_1).
\end{array}\right.
\een
So the hamiltonian operator $\H_5$ has the matrix $iM$, where 
\ben 
 M= \left(\begin{array}{cccc}
 -\beta &0&0&-\alpha\\
  0&\beta &\alpha&0\\
0&-\alpha&    -\ga &0\\
   \alpha &0&0&\ga\\
\end{array}\right).
\een
We can calculate  its characteristic polynomial of $M$ explicitly to obtain after factorisation 
$$\det (M-\lambda I)=\big(\lambda^2+(\ga-\beta)\lambda -\beta\ga+\alpha^2\big)\big(\lambda^2-(\ga-\beta)\lambda -\beta\ga+\alpha^2\big)\,.
$$
Then we compute discriminant of the quadratic  polynomial $\lambda^2+(\ga-\beta)\lambda -\beta\ga+\alpha^2$,
$$\Delta= (\beta+\ga)^2-4\alpha^2.$$
Choosing  $\yy_1=\yy_2=\yy$    we get
$$\beta+\ga= 3(2\pi)^{-2} \yy\Big(\frac{1}{\lambda_1^2}+ \frac{1}{\lambda_2^2}-\frac{4}{\lambda_1\lambda_2} \Big),\quad \alpha=6(2\pi)^{-2} \yy\frac{1}{\lambda_1\lambda_2}\,,
$$
and
\begin{align*}\Delta=\frac{9 \rho}{(2\pi)^4}\Big(\frac1{\lambda_1^2}+\frac1{\lambda_2^2}\Big)\Big(\frac1{\lambda_1^2}+\frac1{\lambda_2^2}
-\frac{8}{\lambda_1\lambda_2}\Big)
\leq \frac{9  \rho}{(2\pi)^4}\Big(\frac1{\lambda_1^2}+\frac1{\lambda_2^2}\Big)\Big(\frac1{\lambda_1^2}-\frac7{\lambda_2^2}
\Big)\,.
\end{align*}
Thus,  $\Delta <0$ for all $m\in[1,2]$, and 
 $M$ has eigenvalues with non vanishing imaginary parts. Accordingly, the  hamiltonian operators $\H_5$ and $\H$
have  hyperbolic directions. Actually, since the discriminant of the   polynomial 
$\lambda^2-(\ga-\beta)\lambda -\beta\ga+\alpha^2$  also equals $\Delta$, the hamiltonian operator $\H_5$ has only hyperbolic directions.

\medskip

This example can be generalised to any dimension $d\geq 3$. Let us do it for $d=3$.
Let 
\be\label{AAA}
\A=\{(0,1,0),(1,-1,0)\}.
\ee
We verify that $\L_f$ contains 16 points,  that $(\L_f\times\L_f)_-=\emptyset$ and 
\begin{align*}(\L_f\times\L_f)_+=\{&((0,-1,0),(1,1,0)); ((1,1,0),(0,-1,0));\\
&((1,0,-1),(0,0,1)); ((0,0,1),(1,0,-1));\\
&((1,0,1),(0,0,-1)); ((0,0,-1),(1,0,1))\}\,.
\end{align*}
I.e. $(\L_f\times\L_f)_+$ contains three pairs of symmetric couples $(a,b),(b,a)$ which give  rise to three non trivial  
$2\times2$-blocks in the matrix $\H$. 
Now $M=13$, $M_*=10$ and the decomposition  \eqref{decomp} reads 
$$
\H(\yy)=\H_1(\yy)\oplus\cdots \oplus\H_{13}(\yy)\,.
$$
Here 
 $\H_1(\yy)\oplus\cdots \oplus\H_{10}(\yy)$ is the  diagonal part of $\H$  with purely imaginary eigenvalues, while the operators 
$\H_{11}(\yy)$, $\H_{12}(\yy)$, $\H_{13}(\yy)$ correspond to non-diagonal $4\times4$--matrices. 

Denoting 
 $\yy_1=\yy_{(0,1,0)}$ and $\yy_2=\yy_{(1,-1,0)}$ we find that   the restriction of the Hamiltonian $  \langle K(m,\yy)\zeta_f, \zeta_f\rangle$
to the modes $(\xi_1,\eta_1):=(\xi_{(0,-1,0)},\eta_{(0,-1,0)})$ and $(\xi_2,\eta_2):=(\xi_{(1,1,0)},\eta_{(1,1,0)})$ is governed by the Hamiltonian $h_r(\yy_1,\yy_2)$ given in \eqref{hr}, as in the 2d case. Similarly  the restrictions of the Hamiltonian $  \langle K(m,\yy)\zeta_f, \zeta_f\rangle$
to the pair of modes $(\xi_{(1,0,-1)},\eta_{(1,0,-1)})$ and $(\xi_{(0,0,1)},\eta_{(0,0,1)})$ and 
to the pair of modes $(\xi_{(1,0,1)},\eta_{(1,0,1)})$ and $(\xi_{(0,0,-1)},\eta_{(0,0,-1)})$ are given by the same 
Hamiltonian \eqref{hr}. So $\H_{11}(\yy)\equiv\H_{12}(\yy)\equiv\H_{13}(\yy)$ and for $\yy_1=\yy_2$ we have 3 hyperbolic directions, one in each block $Y^{f11}$, $Y^{f12}$ and  $Y^{f13}$ (see \eqref{deco}) with the same eigenvalues. 

We notice that  the eigenvalues are identically the same for all three blocks,  thus the relation \eqref{single} is violated.
This does not contradict Lemma~\ref{l_nond} since the set \eqref{AAA} is not\sa. Indeed, denoting 
$a=(0,1,0)$, $b=(1,-1,0)$ we see that $c:=a+b=(1,0,0)$. So three points $(0,-1,0), (0,0,\pm1)\in\bS_{|a|}$ all lie at the distance $\sqrt2$ from $c$.
Hence,  it is not true that $a\ann b$.

\section{Some linear algebra}

\begin{lemma}\label{l_sver}
Let $L$ be an $N\times N$-complex matrix with eigenvalues $\am_1,\dots,\am_N$ such that 
$|\am_j-\am_k|\ge\delta>0$ for all $j\ne k$, and the normalised eigenvectors $\xi_1,\dots,\xi_N$.\footnote{we recall that they should be 
regarded as column-vectors.}
Consider the $N\times N$-matrix $U=(\xi_1\xi_2\dots\xi_N)$, so that 
\be\label{diagg}
U^{-1} L U =\text{diag}\,\{\am_1,\dots,\am_N\} =: \Lambda\,.
\ee
Then
\be\label{z0}
\| U^{-1}\| \le \sqrt{N}\, (2\delta^{-1}\|L\|)^{N-1}\,.
\ee
\end{lemma}

\begin{proof}\footnote{We learned this short proof from V.~\v{S}ver\'ak.
}
Let $Ux=y$, where $\|y\|=1$. We have to estimate the norm of $x$. To do this
we will estimate the components $x_j$ of that vector.

From \eqref{diagg} we have that 
\be\label{diag1}
P(L) U= U P(\Lambda)\,, 
\ee
  for any polynomial $P$. Now for $j=1,\dots, N$ consider the 
Lagrangian polynomials $P_j$,
$$
P_j(z) = \frac{(z-\am_1)\dots\widehat{(z-\am_j)}\dots(z-\am_N)}
{(\am_j-\am_1)\dots\widehat{(\am_j-\am_j)}\dots(\am_j-\am_N)}\,,
$$
where the over-hat means that the corresponding factor is  omitted.  Then $P_j(\am_l)= \delta_{j,l}$.
Therefore   $P_j(\Lambda) =\diag (0,\dots,\underset{j}{1},\dots,0)$. Applying \eqref{diag1} with $P=P_j$
to the vector $x$ we get:
$$
P_j(L)y = UP_j(\Lambda)x = U \big(^t(0,\dots, x_j,\dots,0)\big)=x_j\xi_j\,.
$$
Therefore 
$$
|x_j|=\|P_j(L)y  \|  \le \|P_j(L)\|    \le   \frac{(2\|L\|)^{N-1}}{\delta^{N-1}}  \,,
$$
since $\|L-\am_j E\| \le 2\|L\|$. From this we find that 
$
\|x\| \le \sqrt{N}(2\delta^{-1}\|L\|)^{N-1}\|y\|,
$
and the required estimate is established.
\end{proof}

As an example of applying  estimate \eqref{z0}, consider in the symplectic space 
$\big(\R^4=\{(p_1, p_2, q_1, q_2)\}, dp\wedge dq =:\om_2 \big)$ the symmetric  matrix $A=A^{a,b}$,
corresponding to the  quadratic form
\be\label{a,b}
a (p_1q_1+p_2q_2) + b (p_1q_2 -p_2q_1),\qquad a,b\ne0\,,
\ee
and the hamiltonian operator $JA$. It has the eigenvalues $(\pm a\pm ib)$, see \cite{Ar}, Appendix~6. 
 So the spectrum of $JA$ is simple, and we can diagonalise it as in the lemma above:
$\ 
U^{-1} JA U =\diag \{\pm a\pm ib\}\,.
$
Clearly $\|U\|\le 2$, and by \eqref{z0}
$$
\|U^{-1}\| \le 2(\|A\|/\min(|a|,|b|))^3=:T\,.
$$
Let us enumerate the eigenvalues $(\pm a\pm ib)$ as follows:
 $\lambda_1=a+ib$, $\lambda_2=-a+ib$,
$\lambda_3=-a-ib$,  $\lambda_4=a-ib$, 
and let $\xi_1, \dots, \xi_4$ be the corresponding 
eigenvectors. Then $\om_2(\xi_a, \xi_b)=0$,  unless $\{a,b\} = \{1,3\}$ or $\{a,b\} = \{2,4\}$. Consider 
$$
\om_2(\xi_1,\xi_3) =: t_{1,3}\,,\quad \om_2(\xi_2,\xi_4) =: t_{2,4}\,.
$$
 Find a unit vector $\xi_1^d\in\C^4$ such that 
$\om_2(\xi_1, \xi_1^d)=1$, and decompose it as 
\be\label{z-}
\xi_1^d = x_1\xi_1+\dots +x_4\xi_4\,,\qquad x_j\in\C\,.
\ee
Then $\|x\|\le \|U^{-1}\|\le T$. We have 
\be\label{argument}
1=\om_2(\xi_1, \xi_1^d) = \sum x_j \om_2(\xi_1, \xi_j) = x_3 t_{1,3}\le \|x\|  t_{1,3}\,.
\ee
So $1\ge |t_{1,3}|\ge T^{-1}$. Similar $1\ge |t_{2,4}|\ge T^{-1}$.

Now let us modify the eigenvectors as follows:
$$
\tilde \xi_1 = (t_{1,2})^{-1} \xi_1,\quad \tilde \xi_2 = (t_{2,4})^{-1} \xi_2,\quad
 \tilde \xi_3 =  \xi_3,\quad \tilde \xi_4 =  \xi_4\,.
$$
Let 
$\{e_1=e_{p_1}, e_2=e_{p_2},e_3=e_{q_1},e_4=e_{q_2}\}$
be the standard euclidean base of $\R^4$.  Then $\om_2(e_a, e_b)=\om_2(\tilde \xi_a, \tilde  \xi_b)$ for all $a,b$, 
so the modified  transformation $\widetilde U = U\diag\{ (t_{1,2})^{-1}, (t_{2,4})^{-1},1,1\}$ is symplectic. It still diagonalises $JA$,
$\ 
{\widetilde U}^{-1} JA \widetilde U = \diag \{\pm a\pm ib\}, 
$
and satisfies 
\be\label{Wils}
\|\widetilde U^{-1}\|\le T,\qquad \|\widetilde U\| \le 2T\,.
\ee

\begin{lemma}\label{l_cart}
Consider two $N\times N$-matrices $A_1, A_2$, real or complex. Then the distance between their spectra is bounded by
$C \|A_1 - A_2\|^{1/N}$, where $C$ depends only on $N$ and the norms of the two matrices.
\end{lemma}
\begin{proof}
Consider the characteristic polynomial of $A_1$. The classical Cartan theorem (see \cite{Lev}, Section 1.7) tells that the
subset $S_\eps(A_1)$
 of the complex plain, where this polynomial  is smaller than $\eps$,  
may be covered by a finite collection of complex discs such that the sum
of their radii equals $ 2 e\,(\eps)^{1/N}$. The set $S_\eps(A_1)$ contains the eigenvalues of the matrix $A_2$
(i.e., the zeroes of its characteristic polynomial) if we chose  $\eps=$Const$\,\|A_1 - A_2\|$. This  implies the assertion.
\end{proof}

\section{An estimate for polynomial functions}

\begin{lemma}\label{lA1}
Let $F(x)$ be a non-trivial real polynomial of degree $\bar d$, restricted to a bounded domain $\mathcal K\subset\R^n$
with a piece-wise smooth boundary.  Then there exists a positive constant
  $C_F$ such that 
 \be\label{z1}
 \meas\{x\in K^n\mid  |F(x)|<\eps\} \le C_F\eps^{1/\bar d},\qquad \forall\,\eps\in(0,1]\,.
 \ee
 \end{lemma}
 \begin{proof}
 By the compactness argument it suffices to prove this in the vicinity of any point
  $x^0\in \mathcal K\subset\R^n$, where $F(x^0)=0$.
 So we have reduced the problem to the case when 
 \be\label{z3}
 F:B^n_\varkappa:=\{ |x|<\varkappa\} \to \R, \qquad \varkappa>0\,, 
  \ee
and   $F$ is a non-trivial polynomial of degree  $\bar d$, $F(0)=0$. 
For a unit vector $\xi\in R^n$ consider the polynomial of one variable $z\mapsto F(z\xi)$. For a generic $\xi$ 
it has the form $C^Fz^{\bar d}+\dots$, $C^F\ne0$. Rotating the coordinate system we achieve that $\xi=(1,0,\dots,0)$. 
Denote 
  $$
  x=(x_1,\dots,x_n)=(x_1,\bar x),\qquad \bar x=(x_2,\dots, x_n)\,.
  $$
  Then 
  $$
  F(x) = F(x_1, \bar x) = C_d(\bar x) x_1^d+\dots + C_0(\bar x)\,,\qquad C_d(0) = C^F\,,
  $$
  where each $C_j$ is a polynomial of $\bar x$ whose coefficients are bounded in terms of $F$. Decreasing $\varkappa$ 
  if needed we achieve that 
  $$
  |C_d(\bar x)| \ge \tfrac12\, C^F\qquad \forall\, \bar x\in B_\varkappa^{n-1}\,.
  $$
  Lemma \ref{v.112} with $n=\bar d$ applies to the function $x_1\mapsto F(x_1, \bar x) $, $\bar x\in B_\varkappa^{n-1}$, and implies that 
  $$
  \meas\{x_1\in [-\varkappa, \varkappa]: |F(x_1;\bar x)| \le \eps\} \le C^{'F} \eps^{1/\bar d}\,.
  $$
      Jointly with  the Fubini theorem this inequality establishes for the function \eqref{z3} estimate 
     \eqref{z1} with $K^n$ replaced by  $B_\varkappa^{n}$  and implies the assertion of the lemma.
 \end{proof}
 
 \section{ Admissible and strongly admissible random
$R$-sets are typical}
 In this appendix we prove \eqref{admis} and \eqref{w3}. 
 
 \noindent
 {\it Proof of \eqref{admis}}. Clearly
 \be\label{clear}
 \PP(\Omega \setminus \Omega_1) \le 
 \binom{n}2\, \PP \{|\xi^1| = |\xi^2|\}\,,
 \ee
 and
 $$
 \PP \{|\xi^1| = |\xi^2|\} = |\bB(R)|^{-2} C^*\,,\qquad C^*=   \sum_{ \substack{ (a,b) \in\bB(R)\times \bB(R) \\ |a| = |b|}   }  1 
 \,.
 $$
  Denote by $B^+(R)$ the subset of $\R^d$ which is the union of standard 
   1-cubes with centres in points of $\bB(R)$
  and denote by $K(R)$ the cube $\{ x\in\R^d \mid |x_j|\le R\ \forall\,j\}$. Then 
  $$
  C^* \le
   \int_{B^+(R)}  \int_{B^+(R)} \chi_{ |\, |x| - |y|\, |\le \sqrt{d}}\, dx\,dy \le
   \int_{K(R+ 1/2 ) }  \int_{K(R+ 1/2)} \chi_{ |\, |x| - |y|\, |\le \sqrt{d}}\, dx\,dy \,.
  $$
  A straightforward (but a bit cumbersome) calculation shows that the r.h.s. is $\ \le C(d) R^{-1}$. 
  Therefore $\PP(\Omega \setminus \Omega_1) \le C(n,d) R^{-1}$. This and \eqref{clear} implies 
  \eqref{admis}. 
  \bigskip
  
  \noindent
 {\it Proof of \eqref{w3}}. Let us denote $A_d = \pi^{d/2} / \Gamma(\frac{d+2}2)$, where $\Gamma$ is the gamma-function. 
 Then, by the celebrated result of Vinogradov and Chen, for $d\ge2$ we have 
 $$
   \big|  | \bB(R)| - A_d R^d \big| \le C_{\theta_d} R^{\theta_d}\,
   \qquad \forall\, R>0\,,
 $$
 for any $ \theta_d > d-2$ for $d\ge4$ and $\theta_3 = 4/3$; e.g. see \cite{Vinogr}. Since 
 $\ 
 |\bS(R)| \le |\bB(R+\eps)| - |\bB(R-\eps)|
 $ 
 for every $\eps>0$, then
 \be\label{VC}
 \Gamma_{R,d}:= 
 |\bS(R)| \le 2C_\theta R^\theta\qquad \forall\, R>0\,,
 \ee
 with $\theta = \theta_d$ as above.\footnote{It is known (see \cite{Har}, Theorem~338) that 
 $\Gamma_{R,2}\le C_\delta R^\delta$ for every $\delta>0$. Writing $\Gamma_{R,3}$ as an integral
 in the counting measure $\sum_{s\in\Z^3}\delta(\cdot-s)$, $\Gamma_{R,3} = \int_{\bS(R)} 1$, 
 applying to this integral the Fubini theorem and the estimate for $\Gamma_{R',2}$, $0\le R'\le R$,
 we find that $\Gamma_{R,3}\le C'_\delta R^{1+\delta}$ for each $\delta>0$, which is better than the 
 estimate, obtained from the 
 Vinogradov--Chen result. But the latter is sufficient for us. 
 }
 
 Below we restrict ourselves to the case $d=3$ since for higher dimension the argument is similar, but
 more cumbersome. We have that 
 \be\label{w4}
 1-\PP\{ \xi^1 \ann \xi^2\} = |\bB(R)|^{-2} C^{**}\,,\quad C^{**}=
  \#  \{ (a,b) \in \bB(R) \times \bB(R)\mid \no a\ann b\}\,, 
 \ee
  and, denoting $a+b=c$, that 
  \be\label{w5}
C^{**}    \le 
      \#  \{ (a,c) \in \bB(2R) \times \bB(2R)\mid \no a\an c\}\,.
  \ee

  Now we will estimate the r.h.s. of \eqref{w5}, redenoting $2R$ back to $R$. That is, will 
  estimate the cardinality
  of the set
  $$
  X = \{ (a,b) \in \bB(R) \times \bB(R)\mid \no a\an b \}\,. 
  $$
 It is clear that $(a,b) \in X$, $a\ne0$, iff there exist points $a', \ap \in \bS(|a|)$ such that $b$ lies in the line 
 $\Pi_{\aaa}$, which is perpendicular to the triangle $(\aaa)$ and passes through its centre, so it also passes 
 through the origin. Let $v = v_{\aaa}$ be a primitive integer vector in the direction of  $\Pi_{\aaa}$. 
 For any $a\in\Z^d, a\ne0$, denote
 $$
 \Delta(a) = \big\{\, \{a',\ap\} \subset \bS(|a|)\setminus \{a\}\mid a'\ne \ap \big\}\,.
 $$
 Then
 $$
 |\Delta(a)| < \Gamma_{|a|,3}^2 \le C^2_\theta R^{2\theta},\qquad \theta = \theta_3\,, 
 $$
 see \eqref{VC}. For a fixed $a\in \bB(R)\setminus \{0\}$ consider the mapping 
 $$
 \Delta(a) \ni \{a', \ap\} \mapsto v=v_{\aaa}\,.
 $$
 It is  clear that each direction $v=v_{\aaa}$  gives rise to at most $2R |v|^{-1}$ points 
 $b$ such that $(a,b)\in X$. So, denoting 
 $$
 X_a = \{ b\in \bB(R) \mid (a,b) \in X\}\,,
 $$
 we have
 $$
 |X_a| \le 2R \sum   |v_{\aaa}|^{-1}\,,\quad \text{if}\; a\ne0\,,
 $$
 where the summation goes through all different vectors $v$, corresponding to various  $\{a', \ap\}\in\Delta(a)$. 
  As $|v|^{-1}$ is the bigger the smaller $|v|$ is, we 
 see that 
  the r.h.s. is $\,\le 2R\sum_{v\in\bB(R')\setminus\{0\}}|v|^{-1} $, 
 where $R'$ is any number 
 such that $|\bB(R')| \ge |\Delta(a)|$. Since $|\Delta(a)| \le \Gamma_{|a|,3}^2$, then choosing 
 $R'=R'_a=C\Gamma_{|a|,3}^{2/3}$ we get for any $a\in\bB(R)\setminus \{0\}$ that 
 \begin{equation*}
 \begin{split}
 |X_a| \le 2CR \sum_{ \bB(R'_a)\setminus\{0\}}|v|^{-1} \le C_1 R \int_{ B(R'_a)}|x|^{-1}\,dx  
 \le C_2 R(R'_a)^2 = C_3 R \,\Gamma_{|a|,3}^{4/3}\,.
 \end{split}
 \end{equation*}
 By  \eqref{obvi}, $X_0 = \{0\}$. So 
 $$
 |X| = \sum_{a\in\bB(R)} |X_a| \le 1+ CR \sum_{a\in \bB(R)\setminus\{0\}}  \Gamma_{|a|,3}^{4/3}\,.
 $$
 Evoking the estimate \eqref{VC} we finally get that  
  \begin{equation*}
 \begin{split}
 |X| \le C_1R \sum_{a\in \bB(R)\setminus\{0\}}  |a|^{{\frac43\theta_3}}
 \le C_2R \int_{B(R)} |x|^{{\frac43\theta_3}}\,dx\le C_3 R^{1+3+{\frac43\theta_3}} = C_3 R^{5+7/9}\,.
 \end{split}
 \end{equation*}
 Jointly with  \eqref{w4}, \eqref{w5} and the definition of the set $X$ this 
  implies the required relation \eqref{w3} with $\ka=2/9$.

\end{document}